\documentclass[10pt]{amsart}

\usepackage{xspace}
\usepackage{graphicx}
\usepackage{longtable}
\usepackage{mathtools}
\usepackage{latexsym}
\usepackage{mathrsfs}
\usepackage{stmaryrd}
\usepackage{fullpage}
\usepackage{amsfonts,amsmath,amscd,amssymb}
\usepackage{hyperref}
\input{xy}
\xyoption{all}
\xyoption{curve}

\hypersetup{
	colorlinks=true,
	linkcolor=blue,
	filecolor=magenta,      
	urlcolor=cyan,
}

\newcommand{\bC}{{\mathbb C}}
\newcommand{\bD}{{\mathbb D}}

\newcommand{\bK}{{\mathbb K}}
\newcommand{\bQ}{{\mathbb Q}}
\newcommand{\bR}{{\mathbb R}}

\newcommand{\bZ}{{\mathbb Z}}
\newcommand{\bN}{{\mathbb N}}

\newcommand{\g}{{\mathfrak g}}
\newcommand{\fl}{{\mathfrak l}}
\newcommand{\h}{{\mathfrak h}}

\newcommand{\fo}{{\mathfrak o}}
\newcommand{\p}{{\mathfrak p}}
\newcommand{\fs}{{\mathfrak s}}

\newcommand{\sA}{{\mathscr A}}
\newcommand{\sC}{{\mathscr C}}

\newcommand{\cA}{{\mathcal A}}

\renewcommand{\b}{{\mathfrak b}}

\newcommand{\GL}{{{\mbox{\rm GL}}}}

\newcommand{\adj}{{{\mbox{\rm adj}}}}

\newcommand{\Ad}{{{\mbox{\rm Ad}}}}

\newcommand{\Ext}{{{\mbox{\rm Ext}}}}
\newcommand{\Rank}{{{\mbox{\rm Rank}}}}
\newcommand{\Lie}{{{\mbox{\rm Lie}}}}

\newcommand{\pr}{{{\mbox{\rm pr}}}}

\newcommand{\Hom}{{{\mbox{\rm Hom}}}}

\newcommand{\Id}{{{\mbox{\rm Id}}}}

\newcommand{\im}{{{\mbox{\rm im}}}}

\newcommand{\Sta}{{{\mbox{\rm Stab}}}}

\newcommand{\Rep}{{{\mbox{\rm Rep}}}}

\newcommand{\Tilt}{{{\mbox{\rm Tilt}}}}

\newcommand{\End}{{{\mbox{\rm End}}}}

\newcommand{\Mod}{{{\mbox{\rm Mod}}}}

\newcommand{\twopartdef}[4]{\left\{
	\begin{array}{ll}
		#1 & \mbox{if } #2 \\
		#3 & \mbox{if } #4
	\end{array}
	\right.}

\newcommand{\threepartdef}[6]{\left\{
	\begin{array}{ll}
		#1 & \mbox{if } #2 \\
		#3 & \mbox{if } #4 \\
		#5 & \mbox{if } #6
	\end{array}
	\right.}

\newcommand{\fourpartdef}[8]{\left\{
	\begin{array}{ll}
		#1 & \mbox{if } #2 \\
		#3 & \mbox{if } #4 \\
		#5 & \mbox{if } #6 \\
		#7 & \mbox{if } #8
	\end{array}
	\right.}

\newtheorem{theorem}{Theorem}[section]
\newtheorem{prop}[theorem]{Proposition}
\newtheorem{lemma}[theorem]{Lemma}
\newtheorem*{dfn}{Definition}
\newtheorem{cor}[theorem]{Corollary}

\newtheorem{rmk}{Remark}
\newtheorem{nota}{Notation}

\newcommand{\arxiv}[1]{{\tt arXiv:#1}}

\begin{document}
	\title{Tilting modules and highest weight theory for reduced enveloping algebras}
	\author{Matthew Westaway}
	\email{M.P.Westaway@bham.ac.uk}
	\address{School of Mathematics, University of Birmingham, Birmingham, B15 2TT, UK}
	\date{\today}
	\subjclass[2020]{Primary 17B10; Secondary 17B35; 17B45; 17B50}
	\keywords{Lie algebras, reduced enveloping algebras, standard Levi form, tilting modules, translation functors}
	
	\begin{abstract}
		Let $G$ be a reductive algebraic group over an algebraically closed field of characteristic $p>0$, and let $\g$ be its Lie algebra. Given $\chi\in\g^{*}$ in standard Levi form, we study a category $\sC_\chi$ of graded representations of the reduced enveloping algebra $U_\chi(\g)$. Specifically, we study the effect of translation functors and wall-crossing functors on various highest-weight-theoretic objects in $\sC_\chi$, including tilting modules. We also develop the theory of canonical $\Delta$-flags and $\overline{\nabla}$-sections of  $\Delta$-flags, in analogy with similar concepts for algebraic groups studied by Riche and Williamson.
	\end{abstract}
	\maketitle
	
	\tableofcontents

\section{Introduction}\label{s1}

Let $G$ be a reductive algebraic group over an algebraically closed field $\bK$ of characteristic $p>0$, and let $\g$ be its Lie algebra. In positive characteristic, unlike over $\bC$, the finite-dimensional representation theory of $G$ is very different from that of $\g$. However, in certain situations there are noticeable similarities between the two theories. There has been much recent progress in the representation theory of algebraic groups; a natural question is whether this can be extended to the theory for Lie algebras. In this paper, we generalise some important results from the world of algebraic groups to the world of Lie algebras, particularly focusing on results of central importance to \cite{RW1}.

One of the main goals of the representation theory of algebraic groups is to understand the category $\Rep(G)$ of finite-dimensional rational $G$-modules. This category turns out to have a lot of nice properties. By the geometric Satake equivalence it can be understood geometrically, as a certain category of perverse sheaves on an affine Grassmannian \cite{MV}. Its principal block can be equipped with a well-behaved action of the diagrammatic Hecke category \cite{BR,C}. Furthermore, and of greatest relevance for this paper, it is a (lower finite) highest weight category \cite{BS,CPS,RAGS}.

What does it mean to be a highest weight category? For present purposes, it suffices to say that there exists a partially-ordered indexing set $\Lambda$ of the simple modules $L(\lambda)$, $\lambda\in \Lambda$, in $\Rep(G)$ and two particular families of modules $\Delta(\lambda)$ and $\nabla(\lambda)$, $\lambda\in\Lambda$, such that we have the following diagram $$\Delta(\lambda)\twoheadrightarrow L(\lambda)\hookrightarrow \nabla(\lambda).$$ The simple module $L(\lambda)$ is the irreducible head of $\Delta(\lambda)$ and the irreducible socle of $\nabla(\lambda)$. In $\Rep(G)$, and in many other highest weight categories, the objects $\Delta(\lambda)$ and $\nabla(\lambda)$ (which we call standard and costandard modules respectively) are better understood than the simple modules $L(\lambda)$. A key question is therefore to determine the composition multiplicities $$[\Delta(\lambda):L(\mu)]\quad\mbox{and}\quad[\nabla(\lambda):L(\mu)]$$ for $\lambda,\mu\in\Lambda$ (these multiplicities will coincide in $\Rep(G)$, and in any category with a duality).

A special case of this problem is determining the composition multiplicities in the principal block of $\Rep(G)$. In other words, if $W_p$ is the affine Weyl group corresponding to $G$ and we identify $\Lambda$ with a subset of the character group $X(T)$ (for some maximal torus $T$ in $G$), then the linkage principle tells us that $$[\Delta(\lambda):L(\mu)]\neq 0 \quad\implies \quad \mu\in W_p\cdot \lambda,$$ where $W_p\cdot\lambda$ denotes the orbit of $\lambda$ under the dot-action of $W_p$. The principal block $\Rep_0(G)$ consists of those $M\in \Rep(G)$ such that $$[M:L(\mu)]\neq 0 \quad \implies \quad \mu\in W_p\cdot 0.$$ In the principal block, then, we must determine $$[\Delta(w\cdot 0):L(u\cdot 0)]\quad\mbox{and}\quad[\nabla(w\cdot 0):L(u\cdot 0)]$$ for $w,u\in W_p$.

In 1980, Lusztig \cite{L2} conjectured an answer to this question in terms of Kazhdan-Lusztig polynomials. This conjecture was verified by Andersen, Janzten and Soergel in \cite{AJS} for $p\gg 0$ (concluding a programme initiated by Lusztig in \cite{L90a,L90b}, the other steps of which having been completed in \cite{KT95,KT96,KL93,KL94a,KL94b,L94,L95}), however Williamson in \cite{W1,W2} demonstrated the failure of the conjecture for $G=\GL_n$ and some $p>n$. In \cite[Conj. 5.1, Eqn (1.11)]{RW1}, Riche and Williamson made a new conjecture for the answer to this question, this time in terms of so-called $p$-Kazhdan-Lusztig polynomials. Riche and Williamson show in \cite{RW1}, in an argument heavily relying on the highest weight theory of $\Rep(G)$, that their conjecture would follow from the existence of an action of the Hecke category on the principal block of $\Rep(G)$. This action exists due to separate arguments of Bezrukavnikov-Riche \cite{BR} and Ciappara \cite{C}; in fact, Riche and Williamson's conjecture was proved in \cite{AMRW} and \cite{RW2} using two alternative arguments before such an action was found. There thus currently exist at least three proofs of Riche and Williamson's conjecture.

Given a reductive algebraic group $G$ over an algebraically closed field $\bK$ of characteristic $p>0$, its Lie algebra $\g=\Lie(G)$ is a restricted Lie algebra. It is known that each simple module for $\g$ is a simple module for a reduced enveloping algebra $$U_\chi(\g)\coloneqq\frac{U(\g)}{\langle x^p-x^{[p]}-\chi(x)^p\mid x\in\g\rangle}$$ where $\chi\in\g^{*}$. Understanding the representation theory of these reduced enveloping algebras is thus key in understanding the representation theory of $\g$. In general, the representation theory of the reduced enveloping algebras $U_\chi(\g)$ can be quite complicated; for certain $\chi$, however, more is known. In particular, for $\chi$ in standard Levi form the representation theory of $U_\chi(\g)$ behaves in some respects like the representation theory of $G$. One important characteristic of $\chi$ being in standard Levi form is that $U_\chi(\g)$ can be graded in a convenient way by a certain abelian group, and so we may look at the category of graded $U_\chi(\g)$-modules. This ends up being a better-behaved category than the one of ungraded modules and so we focus our attention on it; in fact, we focus on a certain subcategory $\sC_\chi$ which captures all the structure of the larger category. This category was studied in \cite{Jan4,Jan} and has recently also been studied in \cite{Li,LS,West}.

In \cite{GW}, Goodwin and the author show that $\sC_\chi$ is, in the language of \cite{BS}, an essentially finite fibered highest weight category. This is similar to being a highest weight category, but with some additional subtlety. The diagram we discussed above now becomes 
$$\Delta(\lambda)\twoheadrightarrow \overline{\Delta}(\lambda)\twoheadrightarrow L_\chi(\lambda)\hookrightarrow \overline{\nabla}(\lambda)\hookrightarrow  \nabla(\lambda),$$ and we have to distinguish between standard modules $\Delta(\lambda)$ and proper standard modules $\overline{\Delta}(\lambda)$ (and, analogously, costandard and proper costandard modules $\nabla(\lambda)$ and $\overline{\nabla}(\lambda)$). We may once again look to compute $$[\overline{\Delta}(\lambda):L_\chi(\mu)]\quad \mbox{and}\quad [\overline{\nabla}(\lambda):L_\chi(\mu)]$$ for $\lambda,\mu\in\Lambda_\chi$; as before, we focus initially on the principal block of $\sC_\chi$. There is, once again, a conjecture of Lusztig regarding the answer to this question \cite{L3} (see also \cite[\S 11.24]{Jan}). For $p\gg 0$, related questions have been studied and answered in \cite{BL}.

In this paper we extend the highest-weight-theoretic results of \cite{RW1} to the category $\sC_\chi$; we hope that this will ultimately allow us to argue as in \cite{RW1} to gain a deeper understanding of the composition multiplicities $[\overline{\Delta}(\lambda):L_\chi(\mu)]$ in $\sC_\chi$. The key highest weight theory ingredients in \cite{RW1} are the theories of: tilting modules; translation and wall-crossing functors (and their effects on tilting modules); canonical $\nabla$-flags; and sections of $\nabla$-flags.

Tilting theory in $\sC_\chi$ is discussed in \cite{Li,GW} and is further fleshed out in Subsection~\ref{ss3.3}. In particular, there are indecomposable tilting modules $T_\chi(\lambda)$ for $\lambda\in\Lambda$. Translation functors are discussed in \cite{Jan4,Jan}. We demonstrate the effect of translation functors on standard objects in Propositions~\ref{TransFunct} and \ref{TransFunct2}, and the effect on tilting modules in Proposition~\ref{TransTilt}, Corollary~\ref{TransTilt2} and Corollary~\ref{RefTransWall}. In particular, Corollary~\ref{RefTransWall} states the following, in analogy to \cite[II.E.11]{RAGS}. Here $W_{I,p}$ denotes the subgroup of the affine Weyl group generated by those affine reflections corresponding to roots lying in $\bZ I$, $W^{I,0}$ denotes a certain subset of $W_p$ defined Notation~\ref{NOT}, and $d$ denotes a particular function from the alcoves of $X(T)$ to $\bZ$ described in Subsection~\ref{ss4.2}.

\begin{theorem}\label{RefTransWall2}
	Let $s\in W_p$ be a reflection in a wall of the fundamental alcove of $X(T)$ and let $\mu\in X(T)$ lie on such wall.
	\begin{enumerate}
		\item If $wsw^{-1}\notin W_{I,p}$ and $ws\cdot 0<w\cdot 0$ then $$\Theta_s(T_\chi(w\cdot 0))=T_\chi(w\cdot 0)\oplus T_\chi(w\cdot 0)$$ and $$\Theta_s(T_\chi(ws\cdot 0))=T_\chi(w\cdot 0)\oplus T_{< d(w\cdot 0)}$$  where $T_{< d(w\cdot 0)}$ is a tilting module which decomposes as a direct sum of indecomposable tilting modules indexed by $u\cdot 0$ for $u\in W^{I,0}$ with $d(u\cdot 0)< d(w\cdot 0)$.
		\item If $wsw^{-1}\in W_{I,p}$ then $$\Theta_s(T_\chi(
		w\cdot 0))=T_\chi(w\cdot 0)\oplus T_\chi(w\cdot 0).$$
	\end{enumerate}
\end{theorem}

We see how the theory of canonical $\nabla$-flags and sections of $\nabla$-flags extends to the category $\sC_\chi$ in Subsection~\ref{ss5.1}. Key in \cite{RW1} is understanding the effect of translation functors on sections of $\nabla$-flags; we extend this to $\sC_\chi$ in Theorems~\ref{LMSect} and \ref{MLSect}. We conclude with Theorem~\ref{TruncWall}, an analogue of \cite[Prop 3.10]{RW1}. As a sample of the kind of results we get, the following is Theorem~\ref{BasTilt1}. In it, the maps $\iota_{w\cdot 0}$ denote fixed inclusions $\Delta(w\cdot 0)\hookrightarrow T_\chi(w\cdot 0)$.
	
\begin{theorem}\label{BasTilt1Int}
	Let $s\in W_p$ be a reflection in a wall of the fundamental alcove of $X(T)$ and let $\mu\in X(T)$ lie on such wall. Let $M$ in the principal block of $\sC_\chi$ have a costandard filtration. Let $u\in W^{I,0}$ be such that $usu^{-1}\notin W_{I,p}$ and $us\cdot 0<u\cdot 0$. Let $f_1,\ldots,f_n$ be a family of morphisms in $\Hom_{\sC_\chi}(T_\chi(u\cdot 0),M)$ such that the maps $f_1\circ \iota_{u\cdot 0},\ldots,f_n\circ \iota_{u\cdot 0}$ form a basis of $\Hom_{\sC_\chi}(\Delta(u\cdot 0),M)$. Let also $g_1,\ldots,g_m$ be a family of morphisms in $\Hom_{\sC_\chi}(T_\chi(us\cdot 0),M)$ such that the maps $g_1\circ \iota_{us\cdot 0},\ldots,g_m\circ \iota_{us\cdot 0}$ form a basis of $\Hom_{\sC_\chi}(\Delta(us\cdot 0),M)$. Then there exist families of morphisms $F_1,\ldots,F_r:T_\chi(u\cdot 0)\to \Theta_s(T_\chi(u\cdot 0))$ and $G_1,\ldots,G_t:T_\chi(u\cdot 0)\to \Theta_s(T_\chi(us\cdot 0))$ such that the compositions $$\Delta(u\cdot 0)\hookrightarrow T_\chi(u\cdot 0)\xrightarrow{F_i} \Theta_s(T_\chi(u\cdot 0))\xrightarrow{\Theta_s(f_j)} \Theta_s(M)$$ and $$\Delta(u\cdot 0)\hookrightarrow T_\chi(u\cdot\lambda)\xrightarrow{G_k} \Theta_s(T_\chi(us\cdot 0))\xrightarrow{\Theta_s(g_l)} \Theta_s(M)$$ form a basis of $\Hom_{\sC_\chi}(\Delta(u\cdot 0),\Theta_s(M))$.
\end{theorem}

The layout of this paper is as follows. In Section~\ref{s2} we give some preliminary facts about the category $\sC_\chi$; in Subsection~\ref{ss2.1} we establish some notation, in Subsection~\ref{ss2.2} we define the category $\sC_\chi$ and discuss some properties of it, and in Subsection~\ref{ss2.3} we explain the duality $\bD$ on $\sC_\chi$. We then, in Section~\ref{s3}, introduce highest weight theory. In Subsection~\ref{ss3.1} we recall some generalities about highest weight theory, following \cite{BS}, and in Subsection~\ref{ss3.2} we see how this applies to $\sC_\chi$. In Subsection~\ref{ss3.3} we prove some results about maps between tilting modules and other modules in $\sC_\chi$ which will be key in  what follows. Section~\ref{s4} focuses on translation functors in $\sC_\chi$, beginning with some generalities in Subsection~\ref{ss4.1} and exploring the effect of translation functors on certain modules in Subsection~\ref{ss4.2}. We conclude in Section~\ref{s5} by extending the theory of canonical $\nabla$-flags and their sections to $\sC_\chi$. In Subsection~\ref{ss5.1} we generalise the definitions and main parts of the theory; in Subsections~\ref{ss5.2} and \ref{ss5.3} we see how translation onto and off of a wall affects sections of $\nabla$-flags; and finally, in Subsection~\ref{ss5.4}, we put this together to conclude with some remarks about maps between wall-crossing functors.

{\bf Acknowledgments:} The author would like to thank Simon Goodwin for useful discussions regarding aspects of this paper. The author was supported during this research by a research fellowship from the Royal Commission for the Exhibition of 1851.

\section{Preliminaries on $\sC_\chi$}\label{s2}
\subsection{Notation and Set-up}\label{ss2.1}
Let $G$ be a reductive algebraic group over an algebraically closed field $\bK$ of characteristic $p>0$. We fix a maximal torus $T$ of $G$ and a Borel subgroup $B$ of $G$ containing $T$. Set $X(T)$ to be the character group and $Y(T)$ to be the cocharacter group of $T$, and let $\langle \cdot,\cdot\rangle: X(T)\times Y(T)\to \bZ$ be the perfect pairing between $X(T)$ and $Y(T)$. Note that $X(T)$ and $Y(T)$ are both isomorphic to $\bZ^{\tiny \Rank(T)}$.

We denote by $\g, \b$ and $\h$ the Lie algebras of $G$, $B$ and $T$, respectively. The Lie algebra $\g$ is restricted, with $p$-th power map denoted $x\mapsto x^{[p]}$; this also restricts to a $p$-th power map on $\b$ and on $\h$.

We write $\Phi\subseteq X(T)$ for the root system of $\g$ corresponding to $T$ so that  $\g=\h\oplus\bigoplus_{\alpha\in\Phi}\g_\alpha$, where $$\g_\alpha\coloneqq\{x\in\g\mid \Ad(t)(x)=\alpha(t)x\mbox{ for all }t\in T\}.$$ Denote by $\Phi^{+}$ the positive roots corresponding to $B$ and by $\Phi_s$ the associated simple roots. Given $\alpha\in \Phi$, we write $\alpha^\vee\in Y(T)$ for the corresponding coroot. For $\alpha\in \Phi$, we define $h_\alpha\coloneqq d\alpha^\vee(1)\in\h$, and we choose $e_\alpha\in \g_\alpha$ and $e_{-\alpha}\in\g_{-\alpha}$ so that $[e_\alpha,e_{-\alpha}]=h_\alpha$. 

To each $\alpha\in \Phi_s$ we associate the simple reflection $s_\alpha:X(T)\to X(T)$ given by $s_\alpha(\lambda)=\lambda-\langle\lambda,\alpha^\vee\rangle\alpha$ and, for $m\in\bZ$, the translation $t_{\alpha,m}:X(T)\to X(T)$ given by $t_\alpha(\lambda)=\lambda+m\alpha$. The group generated by the simple reflections $s_\alpha$ for $\alpha\in \Phi_s$ will be denoted $W$ (the Weyl group) and the group generated by the simple reflections $s_\alpha$ and the translations $t_{\alpha,mp}$ for $\alpha\in \Phi_s$ and $m\in\bZ$ will be denoted $W_{p}$ (the affine Weyl group). We will also write $s_{\alpha,mp}=t_{\alpha,mp}\circ s_\alpha$; $W_p$ is then also generated by the $s_{\alpha,mp}$ for $\alpha\in\Phi_s$ and $m\in\bZ$.

We denote by $\rho\in X(T)\otimes_{\bZ}\bQ$ the half-sum of positive roots $\frac{1}{2}\sum_{\alpha\in\Phi^{+}}\alpha$; this will in fact lie in $X(T)$ under the assumptions we are soon to make. The dot-action of $W$ (resp. $W_p$) on $X(T)$ is then defined by $$w\cdot\lambda=w(\lambda+\rho)-\rho$$ for $w\in W$ (resp. $W_p$) and $\lambda\in X(T)$. Both this action and the usual action descend to actions on $X(T)\otimes_{\bZ}\bR$ and on $\h^{*}$. Given $\lambda\in X(T)$ (or $X(T)\otimes_{\bZ}\bR$, or $\h^{*}$), we denote by $W\cdot\lambda$ the orbit of $\lambda$ under the dot-action of $W$ and by $W(\lambda)$ the orbit under the usual action of $W$ (although we shall rarely use the latter). On the other hand, we will write $\Sta_{W}(\lambda)$ for the stabiliser of $\lambda$ under the dot-action, since we shall never need to consider such stabiliser for the usual action. We will use analogous notation for $W_p$ and any other groups for which it makes sense. 

We assume throughout this paper that $(G,p)$ satisfy Jantzen's standard assumptions; that is to say (A) the derived subgroup of $G$ is simply connected, (B) $p$ is good for $G$, and (C) there exists a non-degenerate symmetric bilinear form on $\g$. Assumption (C) implies that there exists a $G$-equivariant isomorphism $\Theta:\g\xrightarrow{\sim}\g^{*}$, where $G$ acts via the adjoint/coadjoint actions respectively. We say that $\chi\in\g^{*}$ is {\bf semisimple} (resp. {\bf nilpotent}) if $\Theta^{-1}(\chi)$ is. Note that $\chi$ is nilpotent if and only if there exists $g\in G$ such that $g\cdot\chi(\b)=0$; when $\chi$ is nilpotent, we will therefore normally assume that $\chi(\b)=0$. 

We denote by $U(\g)$ the universal enveloping algebra of $\g$ and for each $\chi\in\g^{*}$ we define the {\bf reduced enveloping algebra} $$U_\chi(\g)=\frac{U(\g)}{\langle x^p-x^{[p]}-\chi(x)^p\mid x\in\g\rangle};$$ this is an associative algebra of dimension $p^{\dim\g}$. Every simple $\g$-module is a simple $U_\chi(\g)$-module for some $\chi\in\g^{*}$ (see, for example, \cite{Jan2}). Using Jantzen's standard assumptions, we may reduce most questions we are interested in to the case when $\chi$ is nilpotent.

Given $I\subseteq \Phi_s$, we define a linear form $\chi_I\in\g^{*}$ by $$\chi_I(\b)=0\quad\mbox{and}\quad \chi_I(e_{-\alpha})=\twopartdef{1}{\alpha\in I,}{0}{\alpha\notin I,}$$ for all $\alpha\in\Phi^{+}$. We say that such $\chi_I$ is in {\bf standard Levi form}. Given such $I$ we denote by $W_I\subseteq W$ the subgroup generated by the $s_\alpha$ for $\alpha\in I$ and by $W_{I,p}\subseteq W_p$ the subgroup generated by the $s_\alpha$ and the $t_{\alpha,mp}$ for $\alpha\in I$ and $m\in\bZ$. Furthermore, we denote by $w_0$ the longest element of $W$ and by $w_I$ the longest element of $W_I$. 

{\bf Important note:} From now on, we only consider $\chi$ in standard Levi form and we assume that we have fixed $I\subseteq\Phi_s$ to determine it. We therefore omit $I$ from the notation going forward.

\subsection{The Category $\sC_\chi$}\label{ss2.2}

The Lie algebra $\g$ can be equipped with an $X(T)$-grading such that $\h$ lies in grade 0 and each $\g_\alpha$ lies in grade $\alpha$. This can be extended to an $X(T)$-grading of $U(\g)$. When $\chi$ is in standard Levi form corresponding to $I\subseteq\Phi_s$, this grading descends to an $X(T)/\bZ I$-grading of $U_\chi(\g)$.

We may thus consider the category of $X(T)/\bZ I$-graded $U_\chi(\g)$-modules. In fact, we consider a subcategory which we denote by $\sC_\chi$; this category has been previously studied by Jantzen in \cite{Jan4,Jan}. This category is defined as follows: Its objects are finite-dimensional $U_\chi(\g)$-modules $M$ which, as $\bK$-vector spaces, decompose as $$M=\bigoplus_{\lambda+\bZ I\in X(T)/\bZ I} M_{\lambda+\bZ I}$$ for some subspaces $M_{\lambda+\bZ I}$, subject to the following conditions:
\begin{enumerate}
	\item[(1)] For any $\sigma+\bZ I, \lambda +\bZ I\in X(T)/\bZ I$, we have $(U_\chi(\g)_{\sigma+\bZ I}) M_{\lambda+\bZ I}\subseteq M_{\sigma + \lambda+\bZ I}$.
	\item[(2)] For each $\lambda+\bZ I\in X/\bZ I$ there is a $\bK$-vector space decomposition $$M_{\lambda+\bZ I}=\bigoplus_{\substack{d\mu\in \h^{*} \\ \mu\in \lambda+\bZ I +pX(T)}} M_{\lambda+\bZ I}^{d\mu}$$ with the property that $$hm=d\mu(h)m$$ for each $h\in \h$, $\mu\in \lambda+\bZ I+pX(T)$, and $m\in M_{\lambda+\bZ I}^{d\mu}$.
\end{enumerate}
The category of $X(T)/\bZ I$-graded $U_\chi(\g)$-modules then decomposes as a direct sum of copies of this category (see \cite[\S 11.5]{Jan}).

Let us define $\g_I=\h\oplus\bigoplus_{\alpha\in\Phi\cap\bZ I} \g_\alpha$, $\b_I=\h\oplus\bigoplus_{\alpha\in\Phi^{+}\cap\bZ I} \g_\alpha$ and $\p_I=\g_I\oplus\bigoplus_{\alpha\in \Phi^{+}\setminus \bZ I}\g_\alpha$. We also define a number of categories similar to $\sC_\chi$, but defined for certain subalgebras of $\g$ (such as $\g_I$) instead. In lieu of giving an explicit definition for each of them, we simply list the notation in Table~\ref{tab1}; it is a straightforward exercise to adapt the definition of $\sC_\chi$ to each of these settings (noting, however, that we are always working with $X(T)/\bZ I$-gradings, even in those categories where other gradings exist). The reader may consult \cite{West} for more detail; we have tried to use consistent notation from that paper.

	\begin{table}
	\label{tab1}
	\begin{center}
		\bgroup
		\renewcommand{\arraystretch}{2}
		\caption{Analogues of the category $\sC_\chi$}
		\begin{tabular}{|| p{10em} | p{10em} ||}
		\hline
		$\mbox{Subalgebra of }\g$ & $\mbox{Notation for category}$ \\ [0.5ex]
		\hline
		$\g$ & $\sC_\chi$\\ [0.5ex]
		\hline
		$\b$ & $\sC_\chi'$\\ [0.5ex]
		\hline
		$\h$ & $\sC_\chi''$ \\ [0.5ex]
		\hline
		$\g_I$ & $\sC_\chi^{I}$\\ [0.5ex]
		\hline
		$\p_I$ & $\sC_\chi^{I,+}$\\ [0.5ex]
		\hline
	\end{tabular}
\egroup
\end{center}
\end{table}

There are certain functors between these categories that we will use at times throughout this paper. Two of these are extension functors, $\sC_\chi''\to\sC_\chi'$ and $\sC_\chi^I\to\sC_\chi^{I,+}$, where we trivially extend the module structure. We do not give names to these two functors. Two more functors that we care about are the induction functors $$Z_\chi:\sC_\chi'\to \sC_\chi,\qquad Z_\chi(M)=U_\chi(\g)\otimes_{U_\chi(\b)} M$$ and $$\Gamma_\chi:\sC_\chi^{I,+}\to \sC_\chi,\qquad \Gamma_\chi(M)=U_\chi(\g)\otimes_{U_\chi(\p_I)} M.$$ Further details of these constructions can be found in \cite[\S 4]{West}.

Given $\lambda\in X(T)$, we define $\bK_\lambda\in\sC_\chi''$ to be the $U_\chi(\h)$-module which lies entirely in grade $\lambda+\bZ I$ and on which $h\in\h$ acts via $d\lambda(h)$. Extending this to an object in $\sC_\chi'$, we define $$Z_\chi(\lambda)\coloneqq Z_\chi(\bK_\lambda),$$ which we call the {\bf baby Verma module} corresponding to $\lambda$. Each $Z_\chi(\lambda)$ has a unique simple quotient $L_\chi(\lambda)\in\sC_\chi$ (see \cite[\S 10.2]{Jan}).

Furthermore, we will write $$Z_{\chi,I}(\lambda)\coloneqq U_\chi(\g_I)\otimes_{U_0(\b_I)}\bK_{\lambda}\in \sC_\chi^I,$$ which we view as lying entirely in grade $\lambda+\bZ I$. This is irreducible in $\sC_\chi^I$; furthermore, we have $Z_{\chi}(\lambda)_{\lambda+\bZ I}\cong Z_{\chi,I}(d\lambda)$ as objects in $\sC_\chi^I$ (see \cite[\S 2.8]{Jan4}). We define $Q_{\chi,I}(\lambda)\in\sC_\chi^I$ to be the projective cover of $Z_{\chi,I}(\lambda)$. We may then define $$Q_{\chi}^I(\lambda)\coloneqq \Gamma_\chi(Q_{\chi,I}(\lambda))\in\sC_\chi,$$ where we trivially extend $Q_{\chi,I}(\lambda)$ to an object in $\sC_\chi^{I,+}$. Finally, we denote by $Q_\chi(\lambda)\in\sC_\chi$ the projective cover of $L_\chi(\lambda)\in\sC_\chi$. It is shown in \cite[\S 11.9]{Jan} that, given $\lambda,\mu\in X(T)$, $$L_\chi(\lambda)\cong L_\chi(\mu) \iff Z_\chi(\lambda)\cong Z_\chi(\mu)\iff\lambda\in W_{I,p}\cdot\mu.$$

A module $M\in\sC_\chi$ is said to have a {\bf $Q$-filtration} if there exists a filtration $$0=M_0\subseteq M_1\subseteq M_2\subseteq \cdots\subseteq M_n=M$$ of $M$ such that, for each $1\leq i\leq n$, $M_i/M_{i-1}\cong Q_{\chi}^I(\mu_i)$ for some $\mu_i\in X(T)$. If $M$ has a $Q$-filtration then we denote by $(M:Q_\chi^I(\lambda))$ the number of times $Q_{\chi}^I(\lambda)$ appears as a factor of this filtration (it is shown in \cite[Prop. 8.16]{West} that this does not depend on the filtration; this would also follow from \cite[Thm 3.11]{BS} using that $\sC_\chi$ is an essentially finite fibered highest weight category \cite{GW}). We similarly say that $M\in\sC_\chi$ has a {\bf $Z$-filtration} if there exists a filtration $$0=M_0\subseteq M_1\subseteq M_2\subseteq \cdots\subseteq M_n=M$$ of $M$ such that, for each $1\leq i\leq n$, $M_i/M_{i-1}\cong Z_{\chi}(\mu_i)$ for some $\mu_i\in X(T)$.

If $L\in\sC_\chi$ is a simple module and $M\in\sC_\chi$ then we write $[M:L]$ for the number of times $L$ appears in the composition series of $M$. For $\lambda,\mu\in X(T)$, it is shown in \cite[Prop. 11.18]{Jan} that $$(Q_\chi(\lambda):Q_\chi^I(\mu))=[Z_\chi(\mu):L_\chi(\lambda)].$$ Furthermore, \cite[\S 11.18]{Jan} shows that $Q_\chi^I(\lambda)$ has a filtration with $\left\vert W_I\cdot( \lambda+pX)\right\vert$ successive quotients, all of which are isomorphic to $Z_\chi(\lambda)$. For later use, let us denote $N_I(\lambda)\coloneqq\left\vert W_I\cdot( \lambda+pX)\right\vert$.

\subsection{Duality on $\sC_\chi$}\label{ss2.3}

	Given $M\in \sC_\chi$, the dual space $M^{*}$ can be equipped with a $\g$-module structure by defining $$(x\cdot f)(m)=f(-x\cdot m)\quad \mbox{for all  }x\in\g,f\in M^{*}\mbox{ and } m\in M.$$ With this action, $M$ in fact becomes a $U_{-\chi}(\g)$-module. Furthermore, $M$ can be equipped with an $X(T)/\bZ I$-grading by setting $$(M^{*})_{\lambda+\bZ I}=\{f\in M^{*}\mid f(M_{\nu+\bZ I})=0\mbox{ for all }\nu+\bZ I\in X(T)/\bZ I\mbox{ with }\nu+\bZ I\neq -\lambda+\bZ I\}$$ for $\lambda+\bZ I\in X(T)/\bZ I$. With this structure, it can be checked that $M$ lies in $\sC_{-\chi}$ (which we haven't formally defined but is defined as $\sC_\chi$ is except with $U_{-\chi}(\g)$ in place of $U_\chi(\g)$). 
	
	In \cite{Jan5}, Jantzen introduces an automorphism $\tau:G\to G$ (depending on $I$) such that $\tau$ restricts to an automorphism $T\to T$ (and hence the corresponding automorphism $\tau:\g\to\g$ restricts to an automorphism $\h\to\h$) with the properties that $\chi\circ\tau^{-1}=-\chi$ and $\lambda\circ\tau=-w_I\lambda$ for all $\lambda\in X(T)$.
	
	Given $N\in\sC_{-\chi}$ we define $\,^\tau N$ to be the $\g$-module with the same underlying $\bK$-vector space as $N$ and with $\g$-module structure given by $$x\cdot n=\tau^{-1}(x)n\quad \mbox{for all   } x\in\g \mbox{ and }n\in N,$$ where the $\g$-action on the right-hand side is that of $N$. Furthermore, $\,^{\tau}N$ can be equipped with an $X(T)/\bZ I$-grading by setting $$(\,^{\tau}N)_{\lambda+\bZ I}=N_{-\lambda+\bZ I}$$ for each $\lambda+\bZ I\in X(T)/\bZ I$. It can be checked (see \cite[\S 11.15]{Jan}) that $\,^{\tau}N\in\sC_\chi$.
	
	Combining these two operations, we may define a functor $\bD:\sC_\chi\to\sC_\chi$ by setting $$\bD(M)=\,^{\tau}(M^{*}).$$  It is straightforward to see that we can also define a quasi-inverse to $\bD$, denoted  $\overline{\bD}:\sC_\chi\to\sC_\chi$, by defining $$\overline{\bD}(M)=(\,^{\tau^{-1}}M)^{*}$$ (where $\,^{\tau^{-1}}M$ is defined in an analogous way to $\,^\tau M$). The map $\bD$ is thus an anti-equivalence of categories. More details on $\bD$ can be found in \cite{Jan}.
	
	In particular, it is shown in \cite[\S 11.16]{Jan} that $\bD(L_\chi(\lambda))\cong L_\chi(\lambda)$ for all $\lambda\in X(T)$. Therefore, $\bD$ (and, similarly, $\overline{\bD}$) is a duality (i.e. an anti-equivalence of categories $\sC_\chi\to\sC_\chi$ which fixes isomorphism classes of irreducible modules). In particular, this means that $\bD(Q_\chi(\lambda))\cong Q_\chi(\lambda)$ for $\lambda\in X(T)$ since projective covers and injective hulls coincide in $\sC_\chi$ (see \cite[Lem. 2.7]{Jan4}).

\section{Preliminaries on Highest Weight Theory}\label{s3}

In this section, we summarise the necessary background on highest weight categories and the applications relevant to this paper. We largely follow \cite{BS} for this exposition; other relevant sources are \cite{BD,CPS,Dlab,M}. The application of this theory to the category of $\sC_\chi$ is developed in \cite{GW}.

\subsection{Highest Weight Theory Generalities}\label{ss3.1}

A non-unital associative $\bK$-algebra $A$ is called {\bf locally unital} if it is equipped with a family of idempotents $\{e_j\mid j\in J\}$ with $e_je_k=\delta_{jk}e_j$ such that $$A=\bigoplus_{j,k\in J} e_jAe_k.$$ For such $A$, an $A$-module consists of a module $V$ in the usual sense which decomposes as $V=\bigoplus_{j\in J} e_j V$. We call $A$ {\bf essentially finite-dimensional} if $\dim(e_jA)$ and $\dim(Ae_j)$ are finite for all $j\in J$. Writing $\Mod(A)$ for the category of all finite-dimensional $A$-modules, we say that a $\bK$-linear category $\sA$ is {\bf essentially finite Abelian} if it is equivalent to $\Mod(A)$ for some essentially finite-dimensional locally unital algebra $A$. An equivalent definition for an essentially small category $\sA$ is that $\sA$ is Abelian, all objects in $\sA$ have finite length, $\dim(\Hom_{\sA}(M,N))<\infty$ for all $M,N\in\sA$, and $\sA$ has enough projectives and injectives (see \cite[Cor. 2.20]{BS}).

Given an essentially finite Abelian category $\sA$, let $\Lambda$ be an indexing set for the (isomorphism classes of) irreducible objects $\{L(\lambda)\mid \lambda\in\Lambda\}$ in $\sA$, and suppose that $\Lambda$ is equipped with a partial order $\preceq$. We assume throughout that $\preceq$ is interval finite, i.e. that for all $\lambda,\mu\in\Lambda$ there exist only finitely many $\gamma\in\Lambda$ such that $\mu\preceq\gamma\preceq\lambda$. 

We call a subset $\Omega$ of $\Lambda$ an {\bf upper set} if $\lambda\in\Omega$, $\lambda\preceq \mu$ implies $\mu\in\Omega$, and call $\Omega$ a {\bf lower set} if $\lambda\in\Omega$, $\mu\preceq\lambda$ implies $\mu\in\Omega$. Given any lower set $\Omega\subseteq\Lambda$, let us write $\sA_\Omega$ for the Serre subcategory of $\sA$ generated by the $L(\mu)$ with $\mu\in \Omega$ and $\sA^{\Lambda\setminus\Omega}$ for the Serre quotient $\sA/\sA_{\Omega}$. By \cite[Lem. 2.24]{BS} both $\sA_\Omega$ and $\sA^{\Lambda\setminus\Omega}$ are essentially finite Abelian categories.

Let us denote by $i_\Omega:\sA_\Omega\to\sA$ the canonical inclusion and by $j_{\Lambda\setminus\Omega}:\sA\to \sA^{\Lambda\setminus\Omega}$ the canonical projection; we may apply the recollement formalism to this setting (see \cite[\S 2.5]{BS} for discussion of recollement in this setting and \cite{BBD,CPS} for more general discussion of recollement). Using this, the map $i_{\Omega}$ has a left adjoint $i_{\Omega}^{\ast}:\sA\to\sA_\Omega$ and a right adjoint $i_{\Omega}^{!}:\sA\to\sA_{\Omega}$. Furthermore, the map $j_{\Lambda\setminus\Omega}$ has a left adjoint $j_{\Lambda\setminus\Omega}^{\ast}:\sA^{\Lambda\setminus\Omega}\to\sA$ and a right adjoint $j_{\Lambda\setminus\Omega}^{!}:\sA^{\Lambda\setminus\Omega}\to\sA$. As in \cite[Lem. 2.3, 2.4]{BS}, the units or counits of adjunctions (as appropriate) give $$i_{\Omega}^{\ast}\circ i_{\Omega}\simeq \Id_{\sA_\Omega}\simeq i_{\Omega}^{!}\circ i_{\Omega}$$ and
$$j_{\Lambda\setminus\Omega}\circ j_{\Lambda\setminus\Omega}^{\ast}\simeq \Id_{\sA^{\Lambda\setminus\Omega}}\simeq j_{\Lambda\setminus\Omega}\circ j_{\Lambda\setminus\Omega}^{!}.$$

Given $\lambda\in \Lambda$, we write $\Lambda^{\preceq\lambda}=\{\mu\in\Lambda\mid \mu\preceq\lambda\}$, $\Lambda^{\prec\lambda}=\{\mu\in\Lambda\mid \mu\prec\lambda\}$, $\Lambda^{\nsucceq\lambda}=\{\mu\in\Lambda\mid \mu\nsucceq\lambda\}$ and $\Lambda^{\nsucc\lambda}=\{\mu\in\Lambda\mid \mu\nsucc\lambda\}$; these are all lower sets of $\Lambda$. We then abbreviate the notation $$\sA_{\preceq\lambda}=\sA_{\Lambda^{\preceq\lambda}},\quad \sA_{\prec\lambda}=\sA_{\Lambda^{\prec\lambda}},\quad \sA_{\nsucceq\lambda}=\sA_{\Lambda^{\nsucceq\lambda}},\quad \sA_{\nsucc\lambda}=\sA_{\Lambda^{\nsucc\lambda}} $$ and 
$$\sA^{\npreceq\lambda}=\sA^{\Lambda^{\npreceq\lambda}},\quad \sA^{\nprec\lambda}=\sA^{\Lambda^{\nprec\lambda}},\quad \sA^{\succeq\lambda}=\sA^{\Lambda^{\succeq\lambda}},\quad \sA^{\succ\lambda}=\sA^{\Lambda^{\succ\lambda}}.$$

We will denote by $i_{\preceq \lambda}=i_{\Lambda^{\preceq\lambda}}:\sA_{\preceq\lambda} \to \sA$ the canonical inclusion and by $j_{\npreceq \lambda}=j_{\Lambda^{\npreceq}}:\sA\to \sA^{\npreceq \lambda}$ the canonical projection. Note that the category $\sA_{\preceq\lambda}$ has irreducible objects $\{L^{\preceq\lambda}(\mu)\mid\mu\in\Lambda^{\preceq \lambda}\}$, where $L^{\preceq\lambda}(\mu)\coloneqq L(\mu)$ (and so $i_{\preceq\lambda}(L^{\preceq\lambda}(\mu))=L(\mu)$), and the latter has irreducible objects $\{L^{\npreceq\lambda}(\mu)\mid\mu\in\Lambda\setminus\Lambda^{\preceq \lambda}\}$, where $L^{\npreceq\lambda}(\mu)\coloneqq j_{\npreceq \lambda}(L(\mu))$. 

Note that $\sA_{\prec\lambda}$ is also the Serre subcategory of $\sA_{\preceq\lambda}$ generated by the $L(\mu)$ with $\mu\prec\lambda$, and we denote by $\sA_\lambda$ the Serre quotient $\sA_{\preceq\lambda}/\sA_{\prec\lambda}$. In this setting we denote by $i_{\prec\lambda, \preceq\lambda}:\sA_{\prec\lambda}\to\sA_{\preceq\lambda}$ the canonical inclusion and by $j_\lambda:\sA_{\preceq\lambda}\to\sA^{\lambda}$ the canonical projection. As noted before, these are essentially finite Abelian categories, $\sA_{\prec\lambda}$ has simple objects $L^{\prec\lambda}(\mu)\coloneqq L(\mu)$ for $\mu\prec\lambda$, and $\sA^\lambda$ has unique simple object $L_\lambda(\lambda)\coloneqq j_\lambda(L^{\preceq\lambda}(\lambda))$.

As above, the maps $i_{\preceq \lambda}$ and $i_{\prec\lambda, \preceq\lambda}$ have left adjoints $i_{\preceq\lambda}^{\ast}:\sA\to\sA_{\preceq\lambda}$ and $i_{\prec\lambda,\preceq\lambda}^{\ast}:\sA_{\preceq\lambda}\to\sA_{\prec\lambda}$ and right adjoints $i_{\preceq\lambda}^{!}:\sA\to\sA_{\preceq\lambda}$ and $i_{\prec\lambda,\preceq\lambda}^{!}:\sA_{\preceq\lambda}\to\sA_{\prec\lambda}$. Furthermore, the maps $j_{\npreceq\lambda}$ and $j_\lambda$ have left adjoints $j_{\npreceq\lambda}^{\ast}:\sA_{\npreceq\lambda}\to\sA$ and $j_{\lambda}^{\ast}:\sA^{\lambda}\to\sA_{\preceq\lambda}$ and right adjoints $j_{\npreceq\lambda}^{!}:\sA_{\npreceq\lambda}\to\sA$ and $j_{\lambda}^{!}:\sA^{\lambda}\to\sA_{\preceq\lambda}$. They, of course, satisfy the same isomorphisms noted above for $\Omega$.


Given $\lambda\in\Lambda$, we define $P_\lambda(\lambda)$ (resp. $I_\lambda(\lambda)$) to be the projective cover (resp. injective hull) of $L_\lambda(\lambda)$ in $\sA^\lambda$ and $P(\lambda)$ (resp. $I(\lambda)$) to be the projective cover (resp. injective hull) of $L(\lambda)$ in $\sA$. We then, following \cite{BS}, make the following definitions:
$$\Delta(\lambda)\coloneqq i_{\preceq\lambda}\circ j_\lambda^{!}(P_\lambda(\lambda))\quad\mbox{is called the {\bf standard object} in }\sA\mbox{ corresponding to }\lambda\in \Lambda.$$
$$\nabla(\lambda)\coloneqq i_{\preceq\lambda}\circ j_\lambda^{\ast}(I_\lambda(\lambda))\quad\mbox{is called the {\bf costandard object} in }\sA\mbox{ corresponding to }\lambda\in \Lambda.$$
$$\overline{\Delta}(\lambda)\coloneqq i_{\preceq\lambda}\circ j_\lambda^{!}(L_\lambda(\lambda))\quad\mbox{is called the {\bf proper standard object} in }\sA\mbox{ corresponding to }\lambda\in \Lambda.$$
$$\overline{\nabla}(\lambda)\coloneqq i_{\preceq\lambda}\circ j_\lambda^{\ast}(L_\lambda(\lambda))\quad\mbox{is called the {\bf proper costandard object} in }\sA\mbox{ corresponding to }\lambda\in \Lambda.$$
Note that there exist maps $\overline{\Delta}(\lambda)\twoheadrightarrow L(\lambda)\hookrightarrow \overline{\nabla}(\lambda)$, such that $L(\lambda)$ is the irreducible socle of $\overline{\nabla}(\lambda)$ and the irreducible head of $\overline{\Delta}(\lambda)$. We may alternatively define the standard object of $\sA$ corresponding to $\lambda\in\Lambda$ to be the largest quotient of $P(\lambda)$ which lies in $\sA^{\preceq \lambda}$, and the corresponding proper standard object to be the largest quotient of $\Delta(\lambda)$ with the property that all of its composition factors lie in $\sA^{\prec\lambda}$ except for its irreducible head. Similarly, we may define the corresponding costandard object of $\sA$ corresponding to $\lambda\in\Lambda$ to be the largest subobject of $I(\lambda)$ which lies in $\sA^{\preceq \lambda}$, and the corresponding proper standard object to be the largest subobject of $\nabla(\lambda)$ with the property that all of its composition factors lie in $\sA^{\prec\lambda}$ except for its irreducible socle. Note that if $\sA$ is equipped with a duality then $\nabla(\lambda)=\bD(\Delta(\lambda))$ and $\overline{\nabla}(\lambda)=\bD(\overline{\Delta}(\lambda))$ for all $\lambda\in\Lambda$.

Given a sign function $\varepsilon:\Lambda\to\{\pm\}$, we define $$\Delta_\varepsilon(\lambda)\coloneqq\twopartdef{\Delta(\lambda)}{\varepsilon(\lambda)=+,}{\overline{\Delta}(\lambda)}{\varepsilon(\lambda)=-,}\quad\mbox{and}\quad \nabla_\varepsilon(\lambda)\coloneqq\twopartdef{\overline{\nabla}(\lambda)}{\varepsilon(\lambda)=+,}{\nabla(\lambda)}{\varepsilon(\lambda)=-.}$$ We say that $\sA$ satisfies condition $(P\Delta_\varepsilon)$ if, for each $\lambda\in\Lambda$, there exists a projective object $P_\lambda\in \sA$ with a filtration $$0=P_0\subseteq P_1\subseteq P_2\subseteq \cdots \subseteq P_{n-1}\subseteq P_n=P_\lambda$$ such that $P_n/P_{n-1}\cong \Delta_\varepsilon(\lambda)$ and, for each $1\leq i\leq n-1$, there exists $\mu_i\in \Lambda$ with $\mu_i\succ \lambda$ such that $P_i/P_{i-1}\cong \Delta_\varepsilon(\mu_i)$. Similarly, we say that $\sA$ satisfies condition $(I\nabla_\varepsilon)$ if, for each $\lambda\in\Lambda$, there exists an injective object $I_\lambda\in \sA$ with a filtration $$0=I_0\subseteq I_1\subseteq I_2\subseteq \cdots \subseteq I_{n-1}\subseteq I_n=I_\lambda$$ such that $I_1\cong \nabla_\varepsilon(\lambda)$ and, for each $2\leq i\leq n$, there exists $\mu_i\in \Lambda$ with $\mu_i\succ \lambda$ such that $I_i/I_{i-1}\cong \nabla_\varepsilon(\mu_i)$. It is shown in \cite[Thm 3.5]{BS} that these two conditions are equivalent.

A essentially finite Abelian category $\sA$ is called an {\bf essentially finite fibered highest weight category} if $(P\Delta_\varepsilon)$ (or, equivalently, $(I\nabla_\varepsilon)$) holds for all sign functions $\varepsilon:\Lambda\to\{\pm\}$. If $\Omega\subseteq\Lambda$ is a lower set, then $\sA_\Omega$ and $\sA^{\Lambda\setminus\Omega}$ are also essentially finite fibered highest weight categories if $\sA$ is (see Theorems 3.17 and 3.18 in \cite{BS}).

From now on, let $\sA$ be an essentially finite fibered highest weight category and fix a sign function $\varepsilon:\Lambda\to\{\pm\}$. An object $M\in\sA$ is said to have a {\bf $\Delta_\varepsilon$-filtration} (resp. {\bf $\nabla_\varepsilon$-filtration}) if it has a filtration \begin{equation}\label{filt}
	0=M_0\subseteq M_1\subseteq \cdots \subseteq M_{n-1}\subseteq M_n= M \tag{$\dagger$}
\end{equation} such that, for each $1\leq i\leq n$, there exists $\lambda_i\in\Lambda$ with $M_i/M_{i-1}\cong\Delta_\varepsilon(\lambda_i)$ (resp. $M_i/M_{i-1}\cong \nabla_\varepsilon(\lambda_i)$). If $M$ has a $\Delta_\varepsilon$-filtration (resp. $\nabla_\varepsilon$-filtration), we write $(M:\Delta_\varepsilon(\lambda))$ (resp. $(M:\nabla_\varepsilon(\lambda))$) for the number of appearances of $\Delta_\varepsilon(\lambda)$ (resp. $\nabla_\varepsilon(\lambda)$) in a $\Delta_\varepsilon$-filtration (resp. $\nabla_\varepsilon$-filtration) of $M$. It is shown in \cite[Thm 3.14]{BS} that if $M$ has a $\Delta_\varepsilon$-filtration then $(M:\Delta_\varepsilon(\lambda))=\dim\Hom_{\sA}(M,\nabla_\varepsilon(\lambda))$ and if $M$ has a $\nabla_\varepsilon$-filtration then $(M:\nabla_\varepsilon(\lambda))=\dim\Hom_{\sA}(\Delta_\varepsilon(\lambda),M)$; this, in particular, tells us that the number of appearances of $\Delta_\varepsilon(\lambda)$ (resp. $\nabla_\varepsilon(\lambda)$) in a $\Delta_\varepsilon$-filtration (resp. $\nabla_\varepsilon$-filtration) of $M$ is independent of the filtration.

We may also, at times, use the following nomenclature: we say that $M\in\sA$ has a {\bf standard filtration} (resp. {\bf proper standard filtration}, resp. {\bf costandard filtration}, resp. {\bf proper costandard filtration}) if it has a filtration \begin{equation*}
	0=M_0\subseteq M_1\subseteq \cdots \subseteq M_{n-1}\subseteq M_n=M 
\end{equation*} such that, for each $1\leq i\leq n$, there exists $\lambda_i\in\Lambda$ such that $M_i/M_{i-1}\cong\Delta(\lambda_i)$ (resp. $\overline{\Delta}(\lambda_i)$, resp. $\nabla(\lambda_i)$, resp. $\overline{\nabla}(\lambda_i)$).

It is worth at this point mentioning a couple of properties that standard and costandard modules have that will be useful later. Firstly, Lemma 3.48 in \cite{BS} shows that \begin{equation}\label{dimHom}\dim\Ext_{\sA}^n(\Delta_\varepsilon(\lambda),\nabla_\varepsilon(\mu))=\twopartdef{1}{n=0 \mbox{ and } \lambda=\mu,}{0}{\mbox{not};}
\end{equation}
 as a corollary, this means that $$\Hom_{\sA}(\Delta(\lambda),\nabla(\mu))=0$$ for all $\lambda\neq \mu$. Another important corollary of this result is that whenever $M$ has a standard filtration and $N$ has a proper costandard filtration, (or $M$ has a proper standard filtration and $N$ has a costandard filtration) we have, for $n\geq 1$: $$\Ext_{\sC}^n(M,N)=0.$$ We furthermore have by Lemma 3.44 in \cite{BS}, for $\lambda,\mu\in \Lambda$, $$\Ext_{\sA}^1(\Delta(\lambda),\Delta(\mu))\neq 0 \implies \lambda\preceq \mu$$ and $$\Ext_{\sA}^1(\nabla(\mu),\nabla(\lambda))\neq 0 \implies \lambda\preceq \mu.$$

 Finally, we note that $$[\Delta(\lambda):L(\mu)]\neq 0\implies \mu\preceq\lambda$$ for $\lambda,\mu\in\Lambda$ (and the same result for $\nabla(\lambda)$).

An object $M\in\sA$ is called an {\bf $\varepsilon$-tilting object} if it has a $\Delta_\varepsilon$-filtration and a $\nabla_\varepsilon$-filtration, and we denote by $\Tilt_\varepsilon(\sA)$ the full subcategory of $\sA$ consisting of $\varepsilon$-tilting objects in $\sA$. In this paper, we will assume that all essentially finite fibered highest weight categories are {\bf tilting-bounded}, see \cite[Def. 4.20]{BS} for this definition. Under this assumption, we may label the indecomposable $\varepsilon$-tilting objects in $\sA$ by the elements $\lambda\in \Lambda$ (the indexing set of the irreducible objects of $\sA$) and write them as $T_\varepsilon(\lambda)$; every tilting object in $\sA$ decomposes as a direct sum of such indecomposable tilting objects. The object $T_\varepsilon(\lambda)$ can be uniquely characterised (amongst the indecomposable tilting objects) by three equivalent conditions (see \cite[Thm 4.2]{BS}, and the discussion following \cite[Lem. 4.21]{BS}):
\begin{enumerate}
	\item $T_\varepsilon(\lambda)$ has a $\Delta_\varepsilon$-filtration as in (\ref{filt}) such that $M_1\cong \Delta_\varepsilon(\lambda)$;
	\item $T_\varepsilon(\lambda)$ has a $\nabla_\varepsilon$-filtration as in (\ref{filt}) such that $M_n/M_{n-1}\cong \nabla_\varepsilon(\lambda)$; or
	\item $T_\varepsilon(\lambda)\in \sA^{\preceq\lambda}$ and $$j_\lambda(T_\varepsilon(\lambda))\cong\twopartdef{P_\lambda(\lambda)}{\varepsilon(\lambda)=+,}{I_\lambda(\lambda)}{\varepsilon(\lambda)=-.}$$
\end{enumerate}

Note that from this characterisation, we see that when $\varepsilon(\lambda)=+$ we have $$(T(\lambda):\Delta(\lambda))=1$$ and when $\varepsilon(\lambda)=-$ we have 
$$(T_\epsilon(\lambda):\nabla(\lambda))=1.$$ Furthermore, by construction, $$(T_\varepsilon(\lambda):\nabla(\mu))\neq 0\quad \implies\quad \mu\preceq \lambda$$ for $\lambda,\mu\in\Lambda$ (and the analogous result for $\Delta(\mu)$). This also implies that when $\varepsilon(\lambda)=+$ we have $$[T(\lambda):L(\lambda)]=[\Delta(\lambda): L(\lambda)]$$ and when $\varepsilon(\lambda)=-$ we have $$[T(\lambda):L(\lambda)]=[\nabla(\lambda): L(\lambda)].$$

\subsection{Highest Weight Theory of $\sC_\chi$.}\label{ss3.2}

The character group $X(T)$ can be equipped with two partial orders which are of relevance to this paper. The first we write $\leq$ and is defined as $$\lambda\leq \mu \quad \iff \quad \mu-\lambda=\sum_{\alpha\in\Phi_s} k_\alpha \alpha \,\,\mbox{ for }\,\, k_\alpha\in\bZ_{\geq 0}. $$ The second is usually written as  $\uparrow$ (see, for example, \cite[\S 6.4]{RAGS}), although in this paper we will often write it as $\preceq$ in order to more easily distinguish between $\preceq$ and $\prec$. We define it as the transitive closure of the following relation: 
$$s_{\alpha,mp}\cdot\lambda\uparrow\lambda \quad \iff\quad  \langle\lambda+\rho,\alpha^\vee\rangle\geq mp,$$ where $\alpha\in\Phi^{+}$, $m\in\bZ$ and $\lambda\in X(T)$. Note that $\mu\uparrow\lambda$ implies $\mu\leq\lambda$ and $\mu\in W_p\cdot\lambda$. 

We want the partial order $\preceq$ (or $\uparrow$) on $X(T)$ to be a partial order on $\Lambda$, the indexing set of irreducible objects in $\sC_\chi$. At the moment, $\Lambda$ is defined to be the $W_{I,p}$-orbits on $X(T)$ under the dot-action; in order to partially order it via $\preceq$ we need to identify it with a subset of $X(T)$. This leads us into a little bit of alcove geometry, so we recall the basics below. 

Given $\alpha\in \Phi$ and $n\in\bZ$, we can define a hyperplane $H_{\alpha,n}$ of $X(T)\otimes_{\bZ}\bR$ by the equation $\langle \lambda+\rho,\alpha^\vee\rangle=np$. The connected components of $$(X(T)\otimes_\bZ \bR) \setminus \bigcup_{\substack{\alpha\in\Phi^{+}\\ n\in\bZ}} H_{\alpha,n}$$ are called {\bf alcoves}. The wall of an alcove $A$ are the interiors of the closed sets $\overline{A}\cap H_{\alpha,n}$ for any such $\alpha,n$ for which this is non-empty. More details can be found in \cite[II.6.1]{RAGS}. We saw earlier that $W_p$ acts on $X(T)\otimes_\bZ \bR$ via the dot-action; this action induces a simply-transitive action on the alcoves.

Define $$C=\{\lambda\in X(T)\otimes_{\bZ}\bR\mid 0< \langle \lambda+\rho,\alpha\rangle< p \mbox{ for all }\alpha\in\Phi^{+} \}$$ and $$\overline{C}=\{\lambda\in X(T)\otimes_{\bZ}\bR\mid 0\leq \langle \lambda+\rho,\alpha\rangle\leq p \mbox{ for all }\alpha\in\Phi^{+} \}.$$ We call $C$ the {\bf fundamental alcove}. Note that $\overline{C}$ is a fundamental domain for the dot-action of $W_p$ on $X(T)\otimes_{\bZ}\bR$. We define by $S_p\subseteq W_p$ the set of reflections in the walls of $C$; then $(W_p,S_p)$ is a Coxeter system (see, for example, \cite[II.6.3]{RAGS} and references therein) and we call elements of $S_p$ the {\bf simple reflections} in $W_p$.

\begin{rmk}\label{Cox}
	When $\Phi$ is an indecomposable root system, define $h=\max_{\beta\in\Phi^{+}}\{\langle \rho,\beta^\vee\rangle +1\}$; in general, define $h$ to be the maximum of these values as we run over the indecomposable components of $\Phi$. When $\Phi$ is indecomposable, this is the Coxeter number of $\Phi$. Note that $C\cap X(T)\neq \emptyset$ if and only if $p\geq h$. From now on, we therefore always assume that $p\geq h$.
\end{rmk}

We furthermore define $$X(T)_+=\{\lambda\in X(T) \mid \langle \lambda+\rho,\alpha\rangle\geq 0 \mbox{ for all }\alpha\in\Phi^{+} \}.$$

For a subset $I\subseteq \Phi_s$, we define $$C_I=\{\lambda\in X(T)\otimes_{\bZ}\bR\mid 0< \langle \lambda+\rho,\alpha\rangle< p \mbox{ for all }\alpha\in\Phi^{+}\cap\bZ I \}$$ and  $$\overline{C}_I=\{\lambda\in X(T)\otimes_{\bZ}\bR\mid 0\leq \langle \lambda+\rho,\alpha\rangle\leq p \mbox{ for all }\alpha\in\Phi^{+}\cap\bZ I \}.$$ The closed set $\overline{C}_{I}$ is a fundamental domain for the action of $W_{I,p}$ on $X(T)\otimes_{\bZ}\bR$.

The first partial order that we defined, $\leq$, descends to a partial order on $X(T)/\bZ I$ as
$$\lambda +\bZ I\leq \mu +\bZ I \quad \iff \quad \mu-\lambda +\bZ I=\sum_{\alpha\in\Phi_s} k_\alpha \alpha +\bZ I\,\,\mbox{ for }\,\, k_\alpha\in\bZ_{\geq 0}. $$ The second partial order we defined, $\preceq$ (also denoted $\uparrow$), meanwhile, restricts to a partial order on $\overline{C}_I\cap X(T)$. Note that since $\overline{C}_I$ is a fundamental domain for the dot-action of $W_{I,p}$ on $X(T)\otimes_\bZ \bR$, there is a bijection between $\overline{C}_I \cap X(T)$ and the set of $W_{I,p}$-orbits on $X(T)$ (under the dot-action). We therefore identify these two sets, and throughout this paper we usually denote both of them by $\Lambda_I$. As just discussed, $\preceq$ restricts to a partial order on $\Lambda_I$.

Having established these generalities, let us return to the main setting of this paper - the category $\sC_\chi$. The category $\sC_\chi$ is an essentially finite fibered highest weight category by \cite{GW}. Here, the irreducible objects are indexed by the partially-ordered set $(\overline{C}_I\cap X(T),\uparrow)$. The irreducible, standard, proper standard, costandard and proper costandard objects in $\sC_\chi$ are then $$L(\lambda)=L_\chi(\lambda),$$
$$\Delta(\lambda)=Q_\chi^I(\lambda),$$
$$\overline{\Delta}(\lambda)=Z_\chi(\lambda),$$
$$\nabla(\lambda)=\bD(Q_\chi^I(\lambda)),$$
$$\overline{\nabla}(\lambda)=\bD(Z_\chi(\lambda))$$ where $\bD$ is the duality on $\sC_\chi$ discussed in Subsection~\ref{ss2.3}. (In fact, $\bD(Z_\chi(\lambda))$ and $\bD(Q_\chi^I(\lambda))$ can be described more explicitly - see \cite{Jan,GW} for more details - but we don't need that for this paper so we omit it here.) We also have $$P(\lambda)=Q_\chi(\lambda)=I(\lambda).$$ Note also that these descriptions mean that, in  $\sC_\chi$, standard filtrations are the same thing as $Q$-filtrations and proper standard filtrations are the same things as $Z$-filtration. We shall use this terminology interchangeably.

It is also worth noting here that, as discussed above, $\Delta(\lambda)=Q_\chi^I(\lambda)$ has a proper standard filtration consisting of $N_I(\lambda)$ successive quotients, all of which are isomorphic to $\overline{\Delta}(\lambda)=Z_\chi(\lambda)$ (and the dual statement holds for $\nabla(\lambda)$). This means in particular that for any $M,N\in\sC_\chi$ where $M$ has a standard filtration and $N$ has a costandard filtration, we have $$\Ext_{\sC_\chi}^n(M,N)=0$$ for all $n\geq 1$.

Furthermore, $\sC_\chi$ is tilting-bounded and we may therefore define the indecomposable $\varepsilon$-tilting objects $T_\varepsilon(\lambda)$. In this case, we have by \cite{GW}\footnote{It was also independently shown in \cite{Li} that tilting modules and projective modules coincide in $\sC_\chi$.} that $T_\varepsilon(\lambda)=T_{-\varepsilon}(\lambda)$ and we denote both by $T_\chi(\lambda)$ (for this reason, we also from now on refer to tilting modules rather then $\varepsilon$-tilting modules). Note that there exists a bijection $f:\overline{C}_I\cap X(T)\to \overline{C}_I\cap X(T)$ such that $L_{-\chi}(\lambda+2(p-1)\rho)^{*}\cong L_\chi(f(\lambda))$. It is shown in \cite{GW} that $$T_\chi(\lambda)\cong Q_\chi(f(\lambda))$$ for all $\lambda\in \overline{C}_I\cap X(T)$.

Note that Remark 4.3 in \cite{BS} shows that $$\dim\Hom_{\sC_\chi}(T_\chi(\lambda), \overline{\nabla}(\lambda))=(T_\chi(\lambda):\Delta(\lambda))=1$$ and $$\dim\Hom_{\sC_\chi}(\overline{\Delta}(\lambda),T_\chi(\lambda))=(T_\chi(\lambda):\nabla(\lambda))=1. $$

Furthermore, there exist surjections $T_\chi(\lambda)\twoheadrightarrow \nabla(\lambda)\twoheadrightarrow \overline{\nabla}(\lambda)$ arising out of the $\nabla$-filtration of $T_\chi(\lambda)$ and the $\overline{\nabla}$-filtration of $\nabla(\lambda)$ we have already discussed. We can therefore see that all non-zero homomorphisms $T_\chi(\lambda)\to \overline{\nabla}(\lambda)$ are surjective. Similarly, there exist injections $\overline{\Delta}(\lambda)\hookrightarrow \Delta(\lambda)\hookrightarrow T_\chi(\lambda)$, and thus all non-zero homomorphisms $\overline{\Delta}(\lambda)\to T_\chi(\lambda)$ are injective. 

We fix these injections and surjections once-and-for-all, and we give them the following notation. For each $\lambda\in\Lambda_I$, we fix inclusions $$\iota_\lambda:\Delta(\lambda)\hookrightarrow T_\chi(\lambda),\quad\quad\widehat{\iota}_\lambda:\overline{\Delta}(\lambda)\hookrightarrow \Delta(\lambda),\quad\mbox{and}\quad\overline{\iota}_\lambda\coloneqq\iota_\lambda\circ\widehat{\iota}_\lambda:\overline{\Delta}(\lambda)\to T_\chi(\lambda).$$ Furthermore, we define the projections $$\pi_\lambda\coloneqq\bD(\iota_\lambda):T_\chi(\lambda)\to \nabla(\lambda),\quad\quad\widehat{\pi}_\lambda\coloneqq\bD(\widehat{\iota}_\lambda):\nabla(\lambda)\twoheadrightarrow \overline{\nabla}(\lambda),\quad\mbox{and} \quad\overline{\pi}_\lambda\coloneqq \widehat{\pi}_\lambda\circ\pi_\lambda=\bD(\overline{\iota}_\lambda):T_\chi(\lambda)\twoheadrightarrow \overline{\nabla}(\lambda).$$ (For the latter, we have to note that $\bD(T_\chi(\lambda))=T_\chi(\lambda)$, which holds as we observed above that the analogous statement holds for projective modules.)


Since $(\Delta(\lambda),\overline{\Delta}(\lambda))=(\nabla(\lambda),\overline{\nabla}(\lambda))=N_I(\lambda)$, we may also conclude that $\dim\Hom_{\sC_\chi}(T_\chi(\lambda), \nabla(\lambda))=(T_\chi(\lambda):\overline{\Delta}(\lambda))=N_I(\lambda)$ and $\dim\Hom_{\sC_\chi}(\Delta(\lambda),T_\chi(\lambda))=(T_\chi(\lambda):\overline{\nabla}(\lambda))=N_I(\lambda)$.

\subsection{Tilting Modules in $\sC_\chi$}\label{ss3.3}

In order to get a better understanding of the homomorphisms between tilting modules in $\sC_\chi$, it will be helpful to establish some preliminary results. We do so in this subsection; it should be compared with the standard results on tilting modules in highest weight categories (see, for example, \cite[\S 2]{RW1}).

\begin{prop}\label{TiltBas}
Let $M\in\sC_\chi$ have a costandard filtration and let $\lambda\in\Lambda_I$. Then there exists a collection  $\overline{X}_\lambda$ of morphisms $T_\chi(\lambda)\to M$ such that $\{x\circ\overline{\iota}_\lambda\mid x\in \overline{X}_\lambda\}$ forms a basis of $\Hom_{\sC_\chi}(\overline{\Delta}(\lambda),M)$.
\end{prop}

\begin{proof}
	We first consider the case where $M=\nabla(\mu)$ for some $\mu\in \Lambda_I$. The result is immediate when $\mu\neq \lambda$, since we have in this case $$\Hom_{\sC_\chi}(\overline{\Delta}(\lambda),\nabla(\mu))=0.$$ For $\lambda=\mu$, we consider the exact sequence $$ \Hom_{\sC_\chi}(T_\chi(\lambda)/\overline{\Delta}(\lambda),\nabla(\lambda))\to\Hom_{\sC_\chi}(T_\chi(\lambda),\nabla(\lambda))\to\Hom_{\sC_\chi}(\overline{\Delta}(\lambda),\nabla(\lambda))\to \Ext^1_{\sC_\chi}(T_\chi(\lambda)/\overline{\Delta}(\lambda),\nabla(\lambda)).$$ The result follows in this case since $\Ext_{\sC_\chi}^1(\overline{\Delta}(\kappa),\nabla(\lambda))=0$ for all $\kappa\in \Lambda_I$ (see (\ref{dimHom}) in Subsection~\ref{ss3.1}).
	
	Suppose $M$ has costandard filtration $$0=M_0\subseteq M_1\subseteq \cdots \subseteq M_{n-1}\subseteq M_n=M$$ where $M/M_{n-1}\cong \nabla(\mu)$. By induction, we may assume that $$\Hom_{\sC_\chi}(T_\chi(\lambda),M_{n-1})\to \Hom_{\sC_\chi}(\overline{\Delta}(\lambda),M_{n-1})$$ is surjective. We then get the commutative diagram (with exact rows)
	\begin{eqnarray*}
		\begin{array}{c}\xymatrix{
				\Hom_{\sC_\chi}(T_\chi(\lambda),M_{n-1}) \ar@{->}[r] \ar@{->>}[d] & \Hom_{\sC_\chi}(T_\chi(\lambda),M) \ar@{->}[r] \ar@{->}[d] & \Hom_{\sC_\chi}(T_\chi(\lambda),\nabla(\mu)) \ar@{->}[r] \ar@{->>}[d]  & \Ext_{\sC_\chi}^1(T_\chi(\lambda),M_{n-1})=0 \ar@{^{(}->}[d]\\
				\Hom_{\sC_\chi}(\overline{\Delta}(\lambda),M_{n-1}) \ar@{->}[r] & \Hom_{\sC_\chi}(\overline{\Delta}(\lambda),M) \ar@{->}[r] & \Hom_{\sC_\chi}(\overline{\Delta}(\lambda),\nabla(\mu))  \ar@{->}[r] & \Ext_{\sC_\chi}^1(\overline{\Delta}(\lambda),M_{n-1})=0
		}\end{array}
	\end{eqnarray*}
where we recall that $\Ext_{\sC_\chi}^1(A,B)=0$ whenever $A$ has a proper standard filtration and $B$ has a costandard filtration. The four-lemma then implies that $$\Hom_{\sC_\chi}(T_\chi(\lambda),M)\to \Hom_{\sC_\chi}(\overline{\Delta}(\lambda),M)$$ is surjective, as required.
\end{proof}

An almost identical argument shows:

\begin{prop}\label{TiltBas2}
	Let $M\in\sC_\chi$ have a costandard filtration and let $\lambda\in\Lambda_I$. Then there exists a collection  $X_\lambda$ of morphisms $T_\chi(\lambda)\to M$ such that $\{x\circ\iota_\lambda\mid x\in X_\lambda\}$ forms a basis of $\Hom_{\sC_\chi}(\Delta(\lambda),M)$.
\end{prop}

Using the duality $\bD$, we have the dual statements.

\begin{cor}\label{TiltBasCor}
	Let $M\in\sC_\chi$ have a standard filtration and let $\lambda\in\Lambda_I$. Then there exists a collection  $Y_\lambda$ of morphisms $M\to T_\chi(\lambda)$ such that $\{\pi_\lambda\circ y \mid y\in Y_\lambda\}$ forms a basis of $\Hom_{\sC_\chi}(M,\nabla(\lambda))$ and a collection  $\overline{Y}_\lambda$ of morphisms $M\to T_\chi(\lambda)$ such that $\{\overline{\pi}_\lambda\circ y \mid y\in \overline{Y}_\lambda\}$ forms a basis of $\Hom_{\sC_\chi}(M,\overline{\nabla}(\lambda))$.
\end{cor}

\begin{prop}
	Let $g:T_\chi(\lambda)\to\nabla(\lambda)$ be a surjective morphism in $\sC_\chi$. Then the composition $\overline{\Delta}(\lambda)\xhookrightarrow{\overline{\iota}_\lambda} T_\chi(\lambda)\xrightarrow{g} \nabla(\lambda)$ is non-zero.
\end{prop}

\begin{proof}
	Since $[\overline{\Delta}(\lambda):L_\chi(\lambda)]=1$, we have $[\overline{\iota}_\lambda(\overline{\Delta}(\lambda)): L_\chi(\lambda)]=1$. If $g\circ\overline{\iota}_\lambda=0$ then $\overline{\iota}_\lambda(\overline{\Delta}(\lambda))\subseteq \ker(g)$. Since $[T_\chi(\lambda):L_\chi(\lambda)]=[\ker(g):L_\chi(\lambda)]+[g(T_\chi(\lambda)):L_\chi(\lambda)]=[\ker(g):L_\chi(\lambda)]+[\nabla(\lambda):L_\chi(\lambda)]$, we therefore get that $$[\nabla(\lambda):L_\chi(\lambda)]<[T_\chi(\lambda):L_\chi(\lambda)].$$ However, as observed at the end of Subsection~\ref{ss3.2}, we must have $$[\nabla(\lambda):L_\chi(\lambda)]=[T_\chi(\lambda):L_\chi(\lambda)].$$ Hence, $g\circ\iota\neq 0$.
\end{proof}

\begin{prop}\label{NZeroSurj}
	Every homomorphism $g:T_\chi(\lambda)\to \nabla(\lambda)$ such that the composition $\overline{\Delta}(\lambda)\xhookrightarrow{\overline{\iota}_\lambda} T_\chi(\lambda)\xrightarrow{g} \nabla(\lambda)$ is non-zero is surjective.
\end{prop}

\begin{proof}
 For each $\mu\in\Lambda_I$ let us fix a set $\overline{X}_\mu$ of morphisms $T_\chi(\mu)\to \nabla(\lambda)$ such that $$\{x\circ \overline{\iota}_\mu\mid x\in \overline{X}_\mu\}\subseteq \Hom_{\sC_\chi}(\overline{\Delta}(\mu),\nabla(\lambda))$$ forms a basis of $\Hom_{\sC_\chi}(\overline{\Delta}(\mu),\nabla(\lambda))$ and a set $Y_\mu$ of morphisms $T_\chi(\lambda)\to T_\chi(\mu)$ such that $$\{\pi_\mu\circ y \mid y\in Y_\mu\}\subseteq \Hom_{\sC_\chi}(T_\chi(\lambda),\nabla(\mu))$$ forms a basis of $\Hom_{\sC_\chi}(T_\chi(\lambda),\nabla(\mu))$ (we can do this by Proposition~\ref{TiltBas2} and Corollary~\ref{TiltBasCor}). Note that $$\dim\Hom_{\sC_\chi}(\overline{\Delta}(\mu),\nabla(\lambda))=(\nabla(\lambda):\nabla(\mu))=\twopartdef{1}{\lambda=\mu,}{0}{\lambda\neq \mu,}$$ and thus $\overline{X}_\mu=\emptyset$ for $\mu\neq \lambda$ and $\left\vert \overline{X}_\lambda\right\vert=1$. In particular, since $\overline{\Delta}(\lambda)\xhookrightarrow{\overline{\iota}_\lambda} T_\chi(\lambda)\xrightarrow{g} \nabla(\lambda)$ is non-zero, we may choose $\overline{X}_\lambda=\{g\}$. Theorem 4.43 in \cite{BS} then shows that $$\{g\circ y \mid y\in Y_\lambda\}\subseteq \Hom_{\sC_\chi}(T_\chi(\lambda),\nabla(\lambda))$$ is a basis of $\Hom_{\sC_\chi}(T_\chi(\lambda),\nabla(\lambda))$. If $g$ were not surjective, this would imply that there are no surjective morphisms $T_\chi(\lambda)\twoheadrightarrow\nabla(\lambda)$, which we have already seen is false. Thus, $g$ is surjective.
\end{proof}

We also have the dual statement.

\begin{cor}
	Let $g:\Delta(\lambda)\to T_\chi(\lambda)$ be a morphism in $\sC_\chi$. Then $g$ is injective if and only if the composition $\Delta(\lambda)\xrightarrow{g} T_\chi(\lambda)\overset{\overline{\pi}_\lambda}{\twoheadrightarrow} \overline{\nabla}(\lambda)$ is non-zero.
\end{cor}

Understanding endomorphisms of indecomposable tilting modules is in general a difficult task. If we pass to a certain Serre quotient category, however, the problem becomes more tractable, as the following results show.

\begin{prop}\label{Trunc}
	Let $M\in\sC_\chi$ and $\lambda\in \Lambda_I$. Then the map $$\Hom_{\sC_\chi}(\Delta(\lambda),M)\to \Hom_{\sC_\chi^{\succeq \lambda}}(j_{\succeq\lambda}(\Delta(\lambda)),j_{\succeq\lambda}(M)),\quad \phi\mapsto j_{\succeq\lambda}(\phi)$$ is an isomorphism.
\end{prop}

\begin{proof}
	By \cite[Thm. 3.18]{BS} there exists an isomorphism $\Psi:j_{\succeq\lambda}^{!} \Delta^{\succeq \lambda}(\lambda)\xrightarrow{\sim} \Delta(\lambda)$, where $\Delta^{\succeq \lambda}(\lambda)$ denotes the standard object corresponding to $\lambda$ in $\sC_\chi^{\succeq\lambda}$ (recall that we observed earlier that $\sC_\chi^{\succeq\lambda}$ is an essentially finite fibered highest weight category since $\sC_\chi$ is). Let us write $$\adj:\Hom_{\sC_\chi}(j_{\succeq \lambda}^{!} (\Delta^{\succeq \lambda}(\lambda)),M)\xrightarrow{\sim} \Hom_{\sC_\chi^{\succeq \lambda}}(\Delta^{\succeq \lambda}(\lambda),j_{\succeq \lambda}(M))$$ for the adjunction of $(j_{\succeq\lambda},j^{!}_{\succeq\lambda})$. We then have that $\adj(\Psi)$ is the composition $$\Delta^{\succeq \lambda}(\lambda)\to j_{\succeq \lambda} j_{\succeq \lambda}^{!} (\Delta^{\succeq \lambda}(\lambda))\xrightarrow{j_{\succeq \lambda}(\Psi)} j_{\succeq \lambda}(\Delta(\lambda)),$$ where the first map is the unit of the adjunction. By \cite[Lem. 2.24]{BS}, this is an isomorphism. We then get linear bijections 
	\begin{equation*}
		\begin{split}
			\Hom_{\sC_\chi}(\Delta(\lambda),M) & \cong \Hom_{\sC_\chi}(j_{\succeq \lambda}^{!} (\Delta^{\succeq \lambda}(\lambda)),M) \\ & \cong \Hom_{\sC_\chi^{\succeq \lambda}}(\Delta^{\succeq \lambda}(\lambda), j_{\succeq \lambda} (M)) \\ & \cong \Hom_{\sC_\chi^{\succeq \lambda}}(j_{\succeq \lambda} (\Delta(\lambda)), j_{\succeq \lambda} (M))
		\end{split}
	\end{equation*}
	such that $\phi\in \Hom_{\sC_\chi}(\Delta(\lambda),M)$ maps to $\adj(\phi\circ\Psi)\circ\adj(\Psi)^{-1}$. Using general properties of adjunctions, we have $\adj(\phi\circ\Psi)=j_{\succeq\lambda}(\phi)\circ \adj(\Psi)$. Thus, the isomorphism just described is precisely the map $\phi\mapsto j_{\succeq\lambda} (\phi)$.
\end{proof}

Note that, for ease of notation, we often simply write $M\in\sC_\chi^{\succeq\lambda}$ rather than $j_{\succeq\lambda}(M)$ for $M\in\sC_\chi$ going forward.

To make the statement of the following proposition, we need to introduce some notation. We have already seen that $W_I$ acts on $\h^{*}$ (here we wish to use the usual action, not the dot-action); using this we can induce an action by algebra automorphisms on $S(\h)$, as this can be identified with the polynomial ring on $\h^{*}$. We therefore can define $S(\h)^{W_I}$ to be the algebra of fixed points under this action (with similar notation for any subgroups of $W_I$). Writing $\epsilon:S(\h)\to\bK$ for the counit, we then define $S(\h)^{W_I}_{+}$ to be the two-sided ideal of $S(\h)$ generated by the elements $z\in S(\h)^{W_I}$ such that $\epsilon(z)=0$.

\begin{prop}\label{EndStand}
	Let $\lambda\in \Lambda_I$ and write $W_{I,\lambda}=\Sta_{W_I}(\lambda)$. Then $$\End_{\sC_\chi^{\succeq \lambda}}(T_\chi(\lambda))\cong S(\h)^{W_{I,\lambda}}/S(\h)^{W_I}_+$$ as $\bK$-algebras.
\end{prop}

\begin{proof}
	We prove the following chain of isomorphisms:
	\begin{equation*}
		\begin{split}
			\End_{\sC_\chi^{\succeq \lambda}}(T_\chi(\lambda)) & \overset{(1)}{\simeq} \End_{\sC_\chi^{\succeq \lambda}}(Q_\chi^I(\lambda)) \\ &  \overset{(2)}{\simeq} \End_{\sC_\chi}(Q_\chi^I(\lambda)) \\ & \overset{(3)}{\simeq} \End_{\sC_\chi^I}(Q_{\chi,I}(\lambda)) \\ & \overset{(4)}{\simeq} \End_{U_\chi(\g_I)}(Q_{\chi,I}(d\lambda)) \\ & \overset{(5)}{\simeq} S(\h)^{W_{I,\lambda}}/(S(\h)^{W_I}_+)
		\end{split}
	\end{equation*}
	
	Here $d\lambda\in\h^{*}$ is the derivative of $\lambda\in X(T)$. We justify each of these isomorphisms individually.
	
	{\bf Isomorphism (1):} Since $j_{\succeq \lambda}$ is exact, the short exact sequence $$0\to Q_\chi^I(\lambda)\to T_\chi(\lambda)\to T_\chi(\lambda)/Q_\chi^I(\lambda)\to 0$$ induces an exact sequence $$0\to j_{\succeq \lambda}(Q_\chi^I(\lambda))\to j_{\succeq \lambda}(T_\chi(\lambda))\to j_{\succeq \lambda}(T_\chi(\lambda)/Q_\chi^I(\lambda))=0\to 0$$ and therefore an isomorphism $Q_\chi^I(\lambda)\xrightarrow{\sim} T_\chi(\lambda)$ in $\sC_\chi^{\succeq \lambda}$.
	
	{\bf Isomorphism (2):}  This is Proposition~\ref{Trunc} with $M=Q_\chi^I(\lambda)$, noting that the map $\phi\mapsto j_{\succeq\lambda}(\phi)$ is clearly an algebra homomorphism in this setting.
	
	{\bf Isomorphism (3):} This follows from Frobenius reciprocity applied to $\Gamma_\chi$ (or, alternatively, from the fact that morphisms from $Q_\chi^I(\lambda)$ are entirely determined by their effect on $Q_\chi^I(\lambda)_{\lambda+\bZ I}$). It is straightforward to check that this gives an isomorphism of algebras.
	
	{\bf Isomorphism (4):} This is clear from the fact that $Q_{\chi,I}(\lambda)$ lies entirely in grade $\lambda+\bZ I$.
	
	{\bf Isomorphism (5):} This is Corollary 14 in \cite{MR}; see also \cite[Prop. 10.12]{Jan}.

\end{proof}

\section{Translation Functors}\label{s4}

In this section we explore translation functors on the category $\sC_\chi$. We begin by recalling the general theory, which can be found in \cite[11.20]{Jan}. We then proceed to determining the effect of the translation functors on certain modules in the principal block, mainly the standard modules $Q_\chi^I(w\cdot\lambda)$ and the tilting modules $T_\chi(w\cdot\lambda)$.

\subsection{Generalities on Translation Functors}\label{ss4.1}

Given $\lambda\in X(T)$, let us denote by $\sC_\chi(\lambda)$ the Serre subcategory of $\sC_\chi$ generated by the simple objects $L_\chi(w\cdot \lambda)$ as $w$ runs over the elements of $W_p$. It can be easily deduced from \cite[11.11, Prop. 11.18]{Jan} that $L_\chi(w\cdot\lambda)$, $Z_\chi(w\cdot \lambda)$, $Q_\chi^I(w\cdot\lambda)$ and $Q_\chi(w\cdot\lambda)$ all lie inside $\sC_\chi(\lambda)$. Furthermore, there is a direct sum decomposition (see \cite[11.19]{Jan}) $$\sC_\chi=\bigoplus_{\lambda\in \overline{C}\cap X(T)} \sC_\chi(\lambda)$$ since there are no non-split extensions between modules in different $\sC_\chi(\lambda)$ (see \cite[\S 4.7]{Jan4}). For each $M\in\sC_\chi$, let us write $\pr_\lambda(M)\in\sC_\chi(\lambda)$ for the appropriate direct summand of $M$ in this decomposition, so that $M=\bigoplus_{\lambda\in \overline{C}\cap X(T)}\pr_\lambda(M)$.

Let $M\in\sC_\chi$ and $E\in\sC_0$ (the latter category corresponding to $I=\emptyset$). Then $E\otimes M$ is a $\g$-module via the action $$x\cdot (e\otimes m)=(x\cdot e)\otimes m + e\otimes (x\cdot m)\quad\mbox{for}\,\, x\in\g,\,\, e\in E,\,\,\mbox{and}\,\, m\in M.$$ It is straightforward to check that $E\otimes M$ has $p$-character $\chi$ and thus becomes a $U_\chi(\g)$-module. Furthermore, $E\otimes M$ is $X(T)/\bZ I$-graded, where we set $$(E\otimes M)_{\lambda+\bZ I}=\bigoplus_{\mu+\bZ I\in X(T)/\bZ I} E_{\mu+\bZ I}\otimes M_{\lambda-\mu+\bZ I}$$ for each $\lambda+\bZ I\in X(T)/\bZ I$. Here, $E_{\mu+\bZ I}$ is defined as $\bigoplus_{\gamma\in\mu+\bZ I}E_\gamma$, since $E$ is $X(T)$-graded by definition. The $X(T)/\bZ I$-grading on $E\otimes M$ is compatible with the $U_\chi(\g)$-module structure; in fact, $E\otimes M\in\sC_\chi$.

Given $\lambda,\mu\in \overline{C}\cap X(T)$, let us denote by $E(\lambda,\mu)$ the simple $G$-module with highest weight in $W(\mu-\lambda)$. Viewing this as an $X(T)$-graded $\g$-module, it is straightforward to check that it lies inside $\sC_0$. We may define the functor (see \cite[11.20]{Jan}) $$T_\lambda^\mu:\sC_\chi(\lambda)\to\sC_\chi(\mu),\quad T_\lambda^\mu(M)\coloneqq \pr_\mu(E(\lambda,\mu)\otimes M),$$ which we call the {\bf translation functor} from $\sC_\chi(\lambda)$ to $\sC_\chi(\mu)$.

Note that we may also form $E\otimes M$ for $E\in \sC_{0}'$ (resp. $\sC_{0}''$, $\sC_{0}^I$, $\sC_{0}^{I,+}$) and $M\in \sC_{\chi}'$ (resp. $\sC_{\chi}''$, $\sC_{\chi}^I$, $\sC_{\chi}^{I,+}$), giving $E\otimes M\in \sC_{\chi}'$ (resp. $\sC_{\chi}''$, $\sC_{\chi}^I$, $\sC_{\chi}^{I,+}$). Note here that $\sC_{0}^I$ and $\sC_{0}^{I,+}$ are defined for $I$, not for $\emptyset$, and as such we have not yet given their definition in this paper. Nonetheless, it is clear what this definition should be, so we don't delve into it here.

Let $M\in \sC_{\chi}'$, $E\in \sC_{0}$ (which we can view as a subcategory of $\sC_{0}'$), and $N\in\sC_{\chi}$. By \cite[1.12(1)]{Jan4}, there exists an isomorphism in $\sC_\chi$ $$Z_{\chi}(E\otimes M)\cong E\otimes Z_{\chi}(M),$$ which follows by considering the universal properties of both sides. Similarly, for $M\in \sC_{\chi}^{I,+}$, $E\in \sC_{0}$ (which we can view as a subcategory of $\sC_{0}^{I,+}$), and $N\in\sC_{\chi}$, there exists an isomorphism in $\sC_\chi$ $$\Gamma_{\chi}(E\otimes M)\cong E\otimes \Gamma_{\chi}(M).$$ These identifications will be useful for the following proposition.

\begin{prop}\label{TensQFilt}
	Let $M\in\sC_\chi$ have a $Q$-filtration (i.e. a standard filtration) and let $E\in\sC_{0}$. Then $E\otimes M\in\sC_\chi$ has a $Q$-filtration (i.e. a standard filtration).
\end{prop}

\begin{proof}
	Since $M\mapsto E\otimes M$ is an exact functor, any $M\in\sC_{\chi}$ with a $Q$-filtration $$0=M_0\subseteq M_1\subseteq\cdots\subseteq M_{r-1}\subseteq M_r=M$$ such that $M_i/M_{i-1}\cong Q_{\chi}^I(\lambda_i)$, for $1\leq i\leq r$, maps to a filtration $$0=E\otimes M_0\subseteq E\otimes M_1\subseteq\cdots\subseteq E\otimes M_{r-1}\subseteq E\otimes M_r=E\otimes M$$ such that $(E\otimes M_i)/(E\otimes M_{i-1})\cong E\otimes Q_{\chi}^I(\lambda_i)$ for each $1\leq i \leq r$. We claim that, for $\lambda\in X(T)$, $E\otimes Q_{\chi}^I(\lambda)\in\sC_{\chi}$ has a $Q$-filtration. From above, we have $$E\otimes Q_{\chi}^I(\lambda)\cong\Gamma_{\chi}(E\otimes Q_{I,\chi}(\lambda)),$$ where we view $Q_{I,\chi}(\lambda)\in\sC_{\chi}^I$ as an object in $\sC_{\chi}^{I,+}$ by trivial extension and, on the right hand side, we view $E$ as an object in $\sC_\chi^{I,+}$.
	
	{\bf Note:} One can make $E\otimes Q_{I,\chi}(\lambda)\in\sC_\chi^I$ (formed by viewing $E$ as an element of $\sC_\chi^I$) into an object in $\sC_{\chi}^{I,+}$ by trivial extension. On the other hand, one can make $Q_{I,\chi}(\lambda)\in\sC_\chi^I$ into an object in $\sC_{\chi}^{I,+}$ by trivial extension and then tensor with $E\in\sC_\chi^{I,+}$. Our usual notation would denote the resulting objects from these processes in the same way, even though they will not in general be equal. So, for this proof only, if $M\in\sC_{\chi}^I$ then we denote by $\iota(M)$ the object in $\sC_{\chi}^{I,+}$ obtained by trivial extension. Hence, with this notation established, the above statement should say $$E\otimes Q_{\chi}^I(\lambda)\cong\Gamma_{\chi}(E\otimes \iota(Q_{I,\chi}(\lambda))).$$
	
	Since $\Gamma_\chi$ is exact and $\Gamma_\chi(\iota(Q_{I,\chi}(\mu)))=Q_\chi^I(\mu)$ for $\mu\in X(T)$, it is enough to show that $E\otimes \iota(Q_{I,\chi}(\lambda))\in \sC_\chi^{I,+}$ has a filtration with successive quotients of the form $\iota(Q_{I,\chi}(\mu))$ for $\mu\in X(T)$.
	
	Let $(\sigma_1+\bZ I,\ldots,\sigma_n+\bZ I)\in (X(T)/\bZ I)^n$ be such that $E_{\sigma_j+\bZ I}\neq 0$ for $1\leq j\leq n$, $$E=\bigoplus_{j=1}^nE_{\sigma_j+\bZ I},$$ and $i>j$ if $\sigma_i+\bZ I<\sigma_j+\bZ I$. For $1\leq i\leq n$, set $$E_i\coloneqq\bigoplus_{j=1}^i E_{\sigma_j+\bZ I}.$$ One can check that each $E_i$ lies in $\sC_{0}^{I,+}$. We may therefore define $$V_i(X)\coloneqq E_i\otimes \iota(X)\in \sC_{\chi}^{I,+}$$ for any $X\in\sC_{\chi}^{I}$.
	
	Note that $E_{\sigma_j+\bZ I}\in\sC_{0}^I$ for all $1\leq j\leq n$. If $X\in\sC_{\chi}^I$, we then get (using that $i>j$ if $\sigma_i+\bZ I<\sigma_j +\bZ I$) $$V_i(X)/V_{i-1}(X)\cong \iota(E_{\sigma_i+\bZ I}\otimes X)\in \sC_{\chi}^{I,+}.$$ Since $Q_{I,\chi}(\lambda)\in\sC_{\chi}^I$ is projective, $E_{\sigma_i+\bZ I}\otimes Q_{I,\chi}(\lambda)\in\sC_\chi^I$ is also projective. Hence, we have $$E_{\sigma_i+\bZ I}\otimes Q_{I,\chi}(\lambda)=\bigoplus_{\mu\in \overline{C}_I\cap X(T)}  Q_{I,\chi}(\mu)^{m(\mu,i)}\in\sC_\chi^I$$ for some $m(\mu,i)\in\bZ_{\geq 0}$. Indeed, since $\sC_\chi^I$ is defined in exactly the same way as $\sC_\chi$ except using $U_\chi(\g_I)$ in place of $U_\chi(\g)$, everything we know about $\sC_\chi$ remains true for $\sC_\chi^I$; in particular, the simple objects are indexed by orbits of $W_{I,p}$ on $X(T)$ under the dot-action, and by construction the $Q_{\chi,I}(\mu)$ are the projective covers of those simple modules.
	
	Hence, $$\iota(E_{\sigma_i+\bZ I}\otimes Q_{I,\chi}(\lambda))=\bigoplus_{\mu\in \overline{C}_I\cap X(T)}  \iota(Q_{I,\chi}(\mu))^{m(\mu,i)}.$$ Thus, $$V_i(Q_{I,\chi}(\lambda))/V_{i-1}(Q_{I,\chi}(\lambda))=\bigoplus_{\mu\in \overline{C}_I\cap X(T)}  \iota(Q_{I,\chi}(\mu))^{m(\mu,i)}$$ in $\sC_{\chi}^{I,+}$. We therefore have $$0=V_0(Q_{I,\chi}(\lambda))\subseteq V_1(Q_{I,\chi}(\lambda))\subseteq V_2(Q_{I,\chi}(\lambda))\subseteq \cdots V_{n-1}(Q_{I,\chi}(\lambda))\subseteq V_n(Q_{I,\chi}(\lambda))=E\otimes \iota(Q_{I,\chi}(\lambda))$$ with $V_i(Q_{I,\chi}(\lambda))/V_{i-1}(Q_{I,\chi}(\lambda))\cong \bigoplus_{\mu\in \overline{C}_I\cap X(T)}  \iota(Q_{I,\chi}(\mu))^{m(\mu,i)}$. The result then follows.
%
%
\end{proof}

\begin{cor}\label{StandTransFilt}
	Let $\lambda,\mu\in \overline{C}\cap X(T)$ and let $M\in\sC_\chi$ have a standard filtration. Then $T_\lambda^\mu(M)$ has a standard filtration.
\end{cor}

The following statement is observed in \cite[\S 4.7]{Jan4}:

\begin{lemma}
	Let $E\in\sC_0$ and $\lambda\in X(T)$. Then $E\otimes Z_{\chi}(\lambda)$ has a $Z$-filtration in which $Z_{\chi}(\mu)$ appears $\sum_{\sigma\in W_{I,p}\cdot\mu} \dim E_{\sigma-\lambda}$ times for each $\mu\in \overline{C}_I\cap X(T)$. 
\end{lemma}

In particular, we get the following corollaries.

\begin{cor}
	Let $M\in\sC_\chi$ have a $Z$-filtration (i.e. a proper standard filtration) and let $E\in\sC_{0}$. Then $E\otimes M$ has a $Z$-filtration (i.e. a proper standard filtration).
\end{cor}

\begin{cor}
	Let $\lambda,\mu\in \overline{C}\cap X(T)$ and let $M\in\sC_\chi$ have a proper standard filtration. Then $T_\lambda^\mu(M)$ has a proper standard filtration.
\end{cor}

In order to get the analogous results for costandard and proper costandard filtrations, the most straightforward way is to use the duality $\bD$. Since $\bD$ is exact and fixes irreducible objects, it is straightforward to see that $\bD$ sends objects in $\sC_\chi(\gamma)$ to objects in $\sC_\chi(\gamma)$. This therefore allows us to state the following well-known proposition (whose proof we include, since we couldn't find one in the literature).

\begin{prop}\label{TransDual}
	Let $M\in\sC_\chi$ and let $\lambda,\mu\in \overline{C}\cap X(T)$. Then $\bD(T_\lambda^\mu(M))\cong T_\lambda^\mu(\bD(M))$.
\end{prop}

\begin{proof}
	
	Given a finite-dimensional $G$-module $E$, let us define $\bD(E)$ to be the $G$-module $\,^{\tau}(E^{*})$ (so $g\in G$ acts on $\bD(E)$ as $\tau^{-1}(g)$ acts on $E^{*}$). This module has an $X(T)$-grading with $$\bD(E)_\gamma=(E^{*})_{-w_I\gamma}=\{f\in E^{*}\mid f(E_\kappa)=0\mbox{ for all }\kappa \in X(T)\mbox{ such that }\kappa\neq w_I\gamma\}$$ for $\gamma\in X(T)$. Giving $\bD(E)$ an $X(T)/\bZ I$-grading by setting $\bD(E)_{\nu+\bZ I}=\bigoplus_{\gamma\in \nu+\bZ I} \bD(E)_{\gamma}$, we can view $\bD(E)$ as a module in $\sC_0$ via differentiation.
	
	
	Furthermore, we can check that for each $\nu+\bZ I\in X(T)/\bZ I$ we have  $$\bigoplus_{\gamma\in \nu+\bZ I} \bD(E)_{\gamma}=\left\{f\in E^{*}\mid f\left(\sum_{\varepsilon\in \gamma+\bZ I} E_\varepsilon\right)=0\mbox{ for all }\gamma+\bZ I\in X(T)/\bZ I\mbox{ such that }\gamma+\bZ I\neq \nu+\bZ I\right\}.$$ It is then easy to check $\bD(E(\lambda,\mu)\otimes M)\cong \bD(E(\lambda,\mu))\otimes \bD(M)$; we need to use the above equality to see that the $X(T)/\bZ I$-gradings line up.

	{\bf Claim:} $\bD(E(\lambda,\mu))=E(\lambda,\mu)$ in $\sC_0$.
	
	{\bf Proof of claim:} As $E(\lambda,\mu)$ is a simple $G$-module, there exists $\nu\in X(T)_{+}$ such that $E(\lambda,\mu)=L(\nu)$. The simplicity of $E(\lambda,\mu)$ implies the simplicity of $\bD(E(\lambda,\mu))$ as a $G$-module, so there exists $\kappa\in X(T)_{+}$ such that $\bD(E(\lambda,\mu))=L(\kappa)$. We claim $\kappa=\nu$. We know that $\dim L(\nu)_\nu=1$ and that $\dim L(\nu)_\eta\neq 0$ implies $\eta\leq \nu$ (see, for example, \cite[Prop II.2.4]{RAGS}). Furthermore, we know that $\dim L(\nu)_\eta=\dim L(\nu)_{w\eta}$ for all $w\in W$ (see, for example, \cite[II.1.19(2)]{RAGS}). By construction, we have $\dim\bD(L(\nu))_\nu=\dim (L(\nu)^{*})_{-w_I\nu}=\dim L(\nu)_{w_I\nu}=\dim L(\nu)_\nu=1$ and $0\neq \dim\bD(L(\nu))_\eta=\dim (L(\nu)^{*})_{-w_I\eta}=\dim L(\nu)_{w_I\eta}=\dim L(\nu)_\eta$ implies $\eta\leq \nu$. This implies that $\nu=\kappa$ and thus that $\bD(E(\lambda,\mu))\cong E(\lambda,\mu)$ (as $G$-modules, and hence in $\sC_0$).
	
	Putting this all together, we get that $\bD(E(\lambda,\mu)\otimes M)\cong E(\lambda,\mu)\otimes \bD M$ and thus $\pr_\mu(\bD(E(\lambda,\mu)\otimes M))\cong \pr_\mu(E(\lambda,\mu)\otimes \bD(M))=T_\lambda^\mu(\bD(M))$. On the other hand, since $\bD$ preserves each $\sC_\chi(\gamma)$ for $\gamma\in X(T)$, it is clear that $\pr_\mu(\bD(E(\lambda,\mu)\otimes M))=\bD(\pr_\mu(E(\lambda,\mu)\otimes M))=\bD(T_\lambda^\mu(M))$.
\end{proof}

\begin{cor}
	Let $\lambda,\mu\in \overline{C}\cap X(T)$ and let $w\in W_p$. Then $T_\lambda^\mu(\nabla(w\cdot\lambda))=\bD(T_\lambda^\mu(\Delta(w\cdot\lambda)))$ and $T_\lambda^\mu(\overline{\nabla}(w\cdot\lambda))=\bD(T_\lambda^\mu(\overline{\Delta}(w\cdot\lambda)))$.
\end{cor}

\begin{cor}\label{TransCostand}
	Let $\lambda,\mu\in \overline{C}\cap X(T)$ and let $M\in\sC_\chi$ have a costandard (resp. proper costandard) filtration. Then $T_\lambda^\mu(M)$ has a costandard (resp. proper costandard) filtration.
\end{cor}

\begin{cor}\label{TransTilting}
	Let $\lambda,\mu\in \overline{C}\cap X(T)$ and let $M\in\sC_\chi$ be a tilting module. Then $T_\lambda^\mu(M)$ is a tilting module.
\end{cor}

Note that by choosing $E(\mu,\lambda)=E(\lambda,\mu)^{*}$ for all $\lambda\neq\mu\in \overline{C}\cap X(T)$ we see that the translation functors $T_\lambda^\mu$ and $T_\mu^\lambda$ are left and right adjoint to each other, and so we get, for $M\in\sC_\chi(\mu)$ and $N\in\sC_\chi(\lambda)$, isomorphisms $$\Hom_{\sC_\chi}(M,T_\lambda^\mu N)\xrightarrow{\sim}\Hom_{\sC_\chi}(T_\mu^\lambda M,N)$$ and $$\Hom_{\sC_\chi}(N,T_\mu^\lambda M)\xrightarrow{\sim}\Hom_{\sC_\chi}(T_\lambda^\mu N,M)$$ which are functorial in $M$ and $N$.

\subsection{Effect of Translation Functors on Certain Modules}\label{ss4.2}

The behaviour of the translation functors $T_\lambda^\mu$ on many of the objects we care about is determined by the stabilisers of $\lambda$ and $\mu$ under the dot-action of $W_p$. For baby Verma modules, we can describe this explicitly by \cite[Lem. 4.7]{Jan4}.

\begin{lemma}
	Let $\lambda,\mu\in \overline{C}\cap X(T)$ and $w\in W_p$. Then $T_\lambda^\mu Z_\chi(w\cdot\lambda)$ has a filtration whose successive quotients are of the form  $Z_\chi(wu\cdot \mu)$ where $u$ runs over elements of $\Sta_{W_p}(\lambda)/\Sta_{W_p}(\lambda)\cap\Sta_{W_p}(\mu)$.
\end{lemma} 

\begin{cor}\label{TransZ}
	Let $\lambda,\mu\in \overline{C}\cap X(T)$ and $w\in W_p$. 
	\begin{enumerate}
		\item If $\Sta_{W_p}(\lambda)=\{1\}$ then $T_\lambda^\mu Z_\chi(w\cdot\lambda)= Z_\chi(w\cdot\mu)$.
		\item If $\Sta_{W_p}(\lambda)=\{1\}$ and $\Sta_{W_p}(\mu)=\{1,s\}$ for some reflection $s\in W_p$ then $T_\mu^\lambda(Z_\chi(w\cdot\mu))$ has a filtration whose successive quotients are one copy of $Z_\chi(w\cdot\lambda)$ and one copy of $Z_\chi(ws\cdot\lambda)$. These factors are isomorphic if and only if $wsw^{-1}\in W_{I,p}$.
	\end{enumerate}
\end{cor}

The last sentence of this corollary is worth highlighting, since we will use it implicitly throughout. Suppose $\Sta_{W_p}(\lambda)=\{1\}$. Since $Z_\chi(w\cdot\lambda)\cong Z_\chi(ws\cdot\lambda)$ if and only if $ws\cdot\lambda\in W_{I,p}\cdot(w\cdot\lambda)$, this is also equivalent to saying that $wsw^{-1}\in W_{I,p}$.

This corollary also highlights that the relationship between $w\cdot\lambda$ and $ws\cdot\lambda$ is important in understanding translation functors. Let $\lambda\in X(T)$ and $w\in W_p$. By the definition of the partial order $\preceq$ (which we recall from Subsection~\ref{ss3.2}), we then have $$ws\cdot\lambda=(wsw^{-1})w\cdot\lambda<w\cdot\lambda \quad\iff \quad ws\cdot\lambda\prec w\cdot\lambda.$$ In particular, this means that when comparing $w\cdot\lambda$ and $ws\cdot\lambda$ we can either use the notation $<$ or $\prec$ without issue; we use both in this paper, depending on context.

Note also that, viewing $ws\cdot\lambda$ as $(wsw^{-1})w\cdot\lambda$, it is clear that exactly one of the inequalities $w\cdot\lambda<ws\cdot\lambda$ and $ws\cdot\lambda< w\cdot\lambda$ holds; we will often find that objects corresponding to $w\cdot\lambda$ behave differently depending on whether $w\cdot\lambda<ws\cdot\lambda$ or $ws\cdot\lambda< w\cdot\lambda$.

Recall that we assume throughout this paper that $p\geq h$ (see Remark~\ref{Cox}). This allows us to set up the following framework, which we maintain for the rest of the paper. Indeed, if $\gamma,\kappa\in X(T)$ have the same stabiliser in $W_p$ under the dot-action then the translation functor $T_\gamma^\kappa$ is an equivalence of categories, and thus we don't lose anything by making the choices we are about to.

\begin{nota}\label{NOT}

From now on, we fix once-and-for-all a reflection $s\in W_p$ in a wall of $\overline{C}$ and two elements $\lambda,\mu\in X(T)$ satisfying the following properties:
$$\mbox{(1)}\,\,\,\lambda\in C\quad \mbox{and}\quad \mbox{(2)} \,\,\,\mu\in \overline{C}\,\,\mbox{is such that}\,\, \Sta_{W_p}(\mu)=\{1,s\}.$$ This in particular implies that $$\Sta_{W_p}(\lambda)=\{1\}.$$ Note that, given $w\in W_p$, this means that $$\Sta_{W_p}(w\cdot\lambda)=\{1\}\qquad\mbox{and}\qquad \Sta_{W_p}(w\cdot\mu)=\{1,wsw^{-1}\}.$$ We furthermore denote by $W^{I,\lambda}\subseteq W_p$ the subset of $W_p$ consisting of those $w\in W_p$ such that $w\cdot \lambda\in \overline{C}_I\cap X(T)$, and by $W^{I,\mu}\subseteq W_p$ the subset of $W_p$ consisting of those $w\in W_p$ such that $w\cdot \mu\in \overline{C}_I\cap X(T)$ and $ws\cdot\lambda<w\cdot\lambda$. Note that $w\cdot\mu\in \overline{C}_I\cap X(T)$ if $w\in W^{I,\lambda}$.
\end{nota} 

We frequently say ``Let $s,\lambda$ and $\mu$ be as in Notation~\ref{NOT}'' to import this convention into the statements of our results.

The following lemma will be used implicitly a few times in what follows, so we record it here.

\begin{prop}\label{ReflI}
	Let $s,\lambda$ and $\mu$ be as in Notation~\ref{NOT}, and let $w\in W^{I,\lambda}$. Then $ws\cdot\lambda\in C_I\cap X(T)$ if and only if $wsw^{-1}\notin W_{I,p}$.
\end{prop}

\begin{proof}
	If $wsw^{-1}\in W_{I,p}$ then $ws\cdot\lambda=(wsw^{-1})\cdot(w\cdot\lambda)\notin C_I\cap X(T)$ since $w\cdot\lambda\in C_I\cap X(T)$ and we know that $\overline{C}_I$ is a fundamental domain for the action of $W_{I,p}$ on $X(T)\otimes_{\bZ}\bR$. 
	
	Conversely, if $ws\cdot\lambda\notin C_I\cap X(T)$ then $w\cdot\mu$ lies on the wall between the alcove containing $w\cdot\lambda$ and the alcove containing $ws\cdot\lambda$; hence $w\cdot\mu\in \overline{C}_I\cap X(T)$. In particular, this means that there exists $\alpha\in \Phi^{+}\cap \bZ I$ such that $\langle w\cdot\mu+\rho,\alpha^\vee\rangle$ equals either $0$ or $p$. If $\langle w\cdot\mu+\rho,\alpha^\vee\rangle=0$, we have $$s_\alpha\cdot (w\cdot\mu)=s_\alpha(w\cdot\mu+\rho)-\rho=w\cdot\mu - \langle w\cdot\mu+\rho,\alpha^\vee\rangle \alpha=w\cdot\mu$$ and thus $s_\alpha\in\Sta_{W_p}(w\cdot\mu)=\{1,wsw^{-1}\}$; alternatively, if $\langle w\cdot\mu+\rho,\alpha^\vee\rangle=p$ we have $$s_{\alpha,p}\cdot (w\cdot\mu)=s_\alpha(w\cdot\mu+\rho)-\rho+p\alpha=w\cdot\mu - \langle w\cdot\mu+\rho,\alpha^\vee\rangle \alpha+p\alpha=w\cdot\mu$$ and thus $s_{\alpha,p}\in\Sta_{W_p}(w\cdot\mu)=\{1,wsw^{-1}\}$. Since $s_{\alpha}\neq 1$ (resp. $s_{\alpha,p}\neq 1$), we must have $s_\alpha=wsw^{-1}$ (resp. $s_{\alpha,p}=wsw^{-1}$). In either case, we then have $wsw^{-1}\in W_{I,p}$ (since $\alpha\in\Phi^{+}\cap \bZ I$). This proves the result.
\end{proof}

For $w\in W_p$ we observed earlier that $Q_{\chi}^I(w\cdot\lambda)$ has a proper standard filtration with $Z_{\chi}(w\cdot\lambda)$ appearing $N_I(w\cdot\lambda)$ times and no other baby Verma modules appearing. This implies by Corollary~\ref{TransZ} that $T_\lambda^\mu Q_{\chi}^I(w\cdot \lambda)$ has a proper standard filtration with $Z_{\chi}(w\cdot \mu)$ appearing $N_I(w\cdot\mu)$ times. It will therefore be helpful to determine $N_I(w\cdot\lambda)$ and $N_I(w\cdot\mu)$ in the situations we consider; we do so now.

Recall that $W_p=p\bZ\Phi\rtimes W$ and thus that there exists a canonical surjection $W_p\twoheadrightarrow W$. For $w\in W_p$ we write $\overline{w}$ for the image of $w$ under this surjection (and use the same notation for the surjection $W_{I,p}\twoheadrightarrow W_I$). Given $\gamma\in \bZ \Phi$, we write $t_{p\gamma}\in W_p$ for the translation $\lambda\mapsto \lambda+p\gamma$ on $X(T)$; by definition, $\overline{t_{p\gamma}}=1$ for all $\gamma\in \bZ \Phi$.

\begin{lemma}\label{StabSize}
	Let $s,\lambda$ and $\mu$ be as in Notation~\ref{NOT}, and let $w\in W_p$. The following equalities hold:
	\begin{enumerate}
		\item  $N_I(w\cdot\lambda)=\left\vert W_I\right\vert$.
		\item If $wsw^{-1}\notin W_{I,p}$, then $N_I(w\cdot\mu)=\left\vert W_I\right\vert$.
		\item If $wsw^{-1}\in W_{I,p}$, then $N_I(w\cdot\mu)=\left\vert W_I\right\vert/2$.
	\end{enumerate}
\end{lemma}

\begin{proof}
	From the Orbit-Stabiliser Theorem we have $N_I(w\cdot\lambda)=\left\vert W_I\cdot(w\cdot\lambda+pX)\right\vert=\left\vert W_I\right\vert/\left\vert \Sta_{W_I}(w\cdot\lambda+pX)\right\vert$ (and, of course, the analogous result for $w\cdot\mu+pX)$. Let $u\in \Sta_{W_I}(w\cdot\lambda+pX)$, so $uw\cdot \lambda-w\cdot\lambda\in pX$. Since $u\in W_I$ we also have $uw\cdot\lambda - w\cdot\lambda\in \bZ I$ and thus (see \cite[11.2]{Jan}, recalling that we make Jantzen's standard assumptions throughout this paper) $$uw\cdot\lambda-w\cdot\lambda\in p\bZ I.$$ There thus exists $\gamma\in\bZ I$ such that $uw\cdot\lambda-w\cdot\lambda=p\gamma$. We therefore have that $$t_{-p\gamma}uw\cdot\lambda=w\cdot\lambda,$$ and so $t_{-p\gamma}u\in\Sta_{W_p}(w\cdot\lambda)$. Since $\Sta_{W_p}(w\cdot\lambda)=\{1\}$, we have $u=t_{p\gamma}$. As $u\in W_I$, we conclude that $u=\overline{u}=\overline{t_{p\gamma}}=1$. (1) then follows.
	
	For (2) and (3), the same argument shows that $t_{-p\gamma}u\in\Sta_{W_p}(w\cdot\mu)$. We thus either have $t_{-p\gamma}u=1$ or $t_{-p\gamma}u=wsw^{-1}$. If $t_{-p\gamma}u=1$, then as above we have $u=1$. If $t_{-p\gamma}u=wsw^{-1}$, let us write $wsw^{-1}=s_{\beta,np}=t_{np\beta}s_\beta$ for some $\beta\in \Phi^{+}$ and $n\in\bZ$. Then $t_{p\gamma}wsw^{-1}=u\in W_I$ and so $t_{p\gamma}wsw^{-1}=t_{p(\gamma+n\beta)}s_\beta\in W_I\subseteq W$. We therefore have $t_{p(\gamma+n\beta)}\in W$ and so $t_{p(\gamma+n\beta)}=\overline{t_{p(\gamma+n\beta)}}=1$. Thus, $u=t_{p\gamma}wsw^{-1}=s_\beta$, and so $\overline{wsw^{-1}}=s_\beta=u\in W_{I}$.
	
	In (2), $\overline{wsw^{-1}}\notin W_I$ and so this cannot happen; hence $\Sta_{W_I}(w\cdot\mu+pX)=\{1\}$.
	
	In (3), $\overline{wsw^{-1}}\in W_I$ and so the above arguments show that $\Sta_{W_I}(w\cdot\mu+pX)\subseteq\{1,\overline{wsw^{-1}}\}$. It is easy to see that $\overline{wsw^{-1}}\in\Sta_{W_I}(w\cdot\mu+pX)$, so in this case we have $\Sta_{W_I}(w\cdot\mu+pX)=\{1,\overline{wsw^{-1}}\}$. The result follows.
\end{proof}

\begin{prop}\label{TransFunct}
	Let $s,\lambda$ and $\mu$ be as in Notation~\ref{NOT}, and let $w\in W^{I,\lambda}$. Then $T_\lambda^\mu Q_{\chi}^I(w\cdot\lambda)=Q_{\chi}^I(w\cdot\mu)$ if $wsw^{-1}\notin W_{I,p}$, and $T_\lambda^\mu Q_{\chi}^I(w\cdot\lambda)$ has a filtration consisting of two copies of $Q_{\chi}^I(w\cdot\mu)$ if $wsw^{-1}\in W_{I,p}$.
\end{prop}

\begin{proof}
	 By Corollary~\ref{StandTransFilt}, $T_\lambda^\mu Q_{\chi}^I(w\cdot\lambda)$ has a $Q$-filtration, and thus it suffices to compute $(T_\lambda^\mu(Q_\chi^I(w\cdot\lambda)):Q_\chi^I(u\cdot\mu))$ for each $u\in W_p$ such that $u\cdot\mu\in \overline{C}_I\cap X(T)$. For this computation, recall that $Q_\chi^I(\kappa)$ has a $Z$-filtration consisting precisely of $N_I(\kappa)$-copies of $Z_\chi(\kappa)$ and nothing else, and recall also that we know how the translation functors act on baby Verma modules by Corollary~\ref{TransZ}. We hence get \begin{equation*}
		\begin{split}
			(T_\lambda^\mu(Q_\chi^I(w\cdot\lambda)):Q_\chi^I(u\cdot \mu)) & = \dfrac{1}{N_I(u\cdot \mu)} (T_\lambda^\mu(Q_\chi^I(w\cdot\lambda)):Z_\chi(u\cdot \mu)) \\ & = \dfrac{N_I(w\cdot \lambda)}{N_I(u\cdot \mu)} (T_\lambda^\mu(Z_\chi(w\cdot\lambda)):Z_\chi(u\cdot \mu)) \\ & = \dfrac{N_I(w\cdot \lambda)}{N_I(u\cdot \mu)} (Z_\chi(w\cdot\mu):Z_\chi(u\cdot \mu)) \\ & =\twopartdef{0}{Z_\chi(w\cdot\mu)\ncong Z_\chi(u\cdot \mu),}{\dfrac{N_I(w\cdot \lambda)}{N_I(u\cdot \mu)}}{Z_\chi(w\cdot\mu)\cong Z_\chi(u\cdot \mu)}
			\\	& =\threepartdef{0}{Q_\chi^I(w\cdot\mu)\ncong Q_\chi^I(u\cdot \mu),}{\dfrac{\left\vert W_I\right\vert}{\left\vert W_I\right\vert}}{Q_\chi^I(w\cdot\mu)\cong Q_\chi^I(u\cdot \mu)\mbox{ and } wsw^{-1}\notin W_{I,p},}
			{\dfrac{\left\vert W_I\right\vert}{\left\vert W_I\right\vert/2}}{Q_\chi^I(w\cdot\mu)\cong Q_\chi^I(u\cdot \mu)\mbox{ and } wsw^{-1}\in W_{I,p}}
			\\ & =\threepartdef{0}{Q_\chi^I(w\cdot\mu)\ncong Q_\chi^I(u\cdot \mu),}{1}{Q_\chi^I(w\cdot\mu)\cong Q_\chi^I(u\cdot \mu)\mbox{ and } wsw^{-1}\notin W_{I,p},}
			{2}{Q_\chi^I(w\cdot\mu)\cong Q_\chi^I(u\cdot \mu)\mbox{ and } wsw^{-1}\in W_{I,p}.}
		\end{split}
	\end{equation*}

This proves the result.

\end{proof}

\begin{prop}\label{TransFunct2}
	Let $s,\lambda$ and $\mu$ be as in Notation~\ref{NOT}, and let $w\in W^{I,\lambda}$. Then $T_\mu^\lambda Q_{\chi}^I(w\cdot\mu)=Q_{\chi}^I(w\cdot\lambda)$ if $wsw^{-1}\in W_{I,p}$, and $T_\mu^\lambda Q_{\chi}^I(w\cdot\mu)$ has a filtration consisting of one copy of $Q_{\chi}^I(w\cdot\lambda)$ and one copy of $Q_{\chi}^I(ws\cdot\lambda)$ if $wsw^{-1}\notin W_{I,p}$.
\end{prop}

\begin{proof}
	As before, $T_\mu^\lambda(Q_\chi^I(w\cdot\mu))$ has a $Q$-filtration and we thus simply need to compute $(T_\mu^\lambda(Q_\chi^I(w\cdot\mu)):Q_\chi^I(u\cdot\lambda))$ for each $u\in W_p$ such that $u\cdot\lambda\in \overline{C}_I\cap X(T)$. Using Corollary~\ref{TransZ} and arguing as in the proof of Proposition~\ref{TransFunct}, we get \begin{equation*}
		\begin{split}
		(T_\mu^\lambda(Q_\chi^I(w\cdot\mu)):Q_\chi^I(u\cdot \lambda)) & = \dfrac{1}{N_I(u\cdot \lambda)} (T_\mu^\lambda(Q_\chi^I(w\cdot\mu)):Z_\chi(u\cdot \lambda)) \\ & = \dfrac{N_I(w\cdot \mu)}{N_I(u\cdot \lambda)} (T_\mu^\lambda(Z_\chi(w\cdot\mu)):Z_\chi(u\cdot \lambda)) \\ & = \dfrac{N_I(w\cdot \mu)}{N_I(u\cdot \lambda)}((Z_\chi(w\cdot\lambda):Z_\chi(u\cdot \lambda))+(Z_\chi(ws\cdot\lambda):Z_\chi(u\cdot \lambda))) \\ & =\fourpartdef{0}{Z_\chi(w\cdot\lambda)\ncong Z_\chi(u\cdot \lambda)\ncong Z_\chi(ws\cdot\lambda),}{\dfrac{N_I(w\cdot \mu)}{N_I(u\cdot \lambda)}}{Z_\chi(w\cdot\lambda)\cong Z_\chi(u\cdot \lambda)\mbox{ and }wsw^{-1}\notin W_{I,p},}{\dfrac{N_I(w\cdot \mu)}{N_I(u\cdot \lambda)}}{Z_\chi(ws\cdot\lambda)\cong Z_\chi(u\cdot \lambda)\mbox{ and }wsw^{-1}\notin W_{I,p},}
			{2\dfrac{N_I(w\cdot \mu)}{N_I(u\cdot \lambda)}}{Z_\chi(w\cdot\lambda)\cong Z_\chi(u\cdot \lambda)\mbox{ and }wsw^{-1}\in W_{I,p}}
			\\	& =\fourpartdef{0}{Q_\chi^I(w\cdot\lambda)\ncong Q_\chi^I(u\cdot \lambda)\ncong Q_\chi^I(ws\cdot\lambda),}{1}{Q_\chi^I(w\cdot\lambda)\cong Q_\chi^I(u\cdot \lambda)\mbox{ and }wsw^{-1}\notin W_{I,p}}{1}{Q_\chi^I(ws\cdot\lambda)\cong Q_\chi^I(u\cdot \lambda)\mbox{ and }wsw^{-1}\notin W_{I,p},}
			{1}{Q_\chi^I(w\cdot\lambda)\cong Q_\chi^I(u\cdot \lambda)\mbox{ and }wsw^{-1}\in W_{I,p}.}
		\end{split}
	\end{equation*}
	
	This proves the result.
\end{proof}

By using the duality $\bD$ and Proposition~\ref{TransDual}, we get the dual statements.

\begin{cor}
	Let $s,\lambda$ and $\mu$ be as in Notation~\ref{NOT}, and let $w\in W^{I,\lambda}$. Then $T_\lambda^\mu \nabla(w\cdot\lambda)=\nabla(w\cdot\mu)$ if $wsw^{-1}\notin W_{I,p}$, and $T_\lambda^\mu \nabla(w\cdot\lambda)$ has a filtration consisting of two copies of $\nabla(w\cdot\mu)$ if $wsw^{-1}\in W_{I,p}$.
\end{cor}

\begin{cor}
	Let $s,\lambda$ and $\mu$ be as in Notation~\ref{NOT}, and let $w\in W^{I,\lambda}$. Then $T_\mu^\lambda \nabla(w\cdot\mu)=\nabla(w\cdot\lambda)$ if $wsw^{-1}\in W_{I,p}$, and $T_\mu^\lambda \nabla(w\cdot\mu)$ has a filtration consisting of one copy of $\nabla(w\cdot\lambda)$ and one copy of $\nabla(ws\cdot\lambda)$ if $wsw^{-1}\notin W_{I,p}$.
\end{cor}

For completeness, let us also state the corresponding statement for proper costandard objects. This follows easily from Proposition~\ref{TransDual} and Corollary~\ref{TransZ}.

\begin{cor}
	Let $s,\lambda$ and $\mu$ be as in Notation~\ref{NOT}, and let $w\in W^{I,\lambda}$. Then $T_\lambda^\mu \overline{\nabla}(w\cdot\lambda)=\overline{\nabla}(w\cdot\mu)$ and $T_\mu^\lambda(w\cdot\mu)$ has a filtration consisting of one copy of $\overline{\nabla}(w\cdot\lambda)$ and one copy of $\overline{\nabla}(ws\cdot\lambda)$ (so two copies of $\overline{\nabla}(w\cdot\lambda)$ if $wsw^{-1}\in W_{I,p}$).
\end{cor}

For $s,\lambda$ and $\mu$ as in Notation~\ref{NOT}, we may define the {\bf wall-crossing functor} $$\Theta_s\coloneqq T_\mu^\lambda T_\lambda^\mu:\sC_{\chi}(\lambda)\to \sC_{\chi}(\lambda).$$ Combining Propositions~\ref{TransFunct} and \ref{TransFunct2} we get the following result (recalling that $Q_{\chi}^I(ws\cdot\lambda)\cong Q_{\chi}(w\cdot\lambda)$ if and only if $wsw^{-1}\in W_{I,p}$):

\begin{cor}
	Let $s,\lambda$ and $\mu$ be as in Notation~\ref{NOT}, and let $w\in W^{I,\lambda}$. If $wsw^{-1}\notin W_{I,p}$, then $\Theta_s(Q_{\chi}^I(w\cdot\lambda))$ has a filtration consisting of one copy of $Q_{\chi}^I(w\cdot\lambda)$ and one copy of $Q_{\chi}^I(ws\cdot\lambda)$; if $wsw^{-1}\in W_{I,p}$, then $\Theta_s(Q_{\chi}^I(w\cdot\lambda))$ has a filtration consisting of two copies of $Q_{\chi}^I(w\cdot\lambda)$.
\end{cor}

The dual statement also holds. 
\begin{cor}
	Let $s,\lambda$ and $\mu$ be as in Notation~\ref{NOT}, and let $w\in W^{I,\lambda}$. If $wsw^{-1}\notin W_{I,p}$, then $\Theta_s(\nabla(w\cdot\lambda))$ has a filtration consisting of one copy of $\nabla(w\cdot\lambda)$ and one copy of $\nabla(ws\cdot\lambda)$; if $wsw^{-1}\in W_{I,p}$, then $\Theta_s(\nabla(w\cdot\lambda))$ has a filtration consisting of two copies of $\nabla(w\cdot\lambda)$.
\end{cor}

Recall that the Grothendieck group of an exact category $\cA$, which we will write as $[\cA]$, is the abelian group generated by symbols $[A]$ for (isomorphism classes of) objects $A\in\cA$ subject to the relation $[B]=[A]+[C]$ whenever there exists an exact sequence $0\to A\to B\to C\to 0$ in $\cA$. We will denote $[\cA]_\bQ\coloneqq [A]\otimes_{\bZ}\bQ$, which is a $\bQ$-vector space. We will abuse notation and write $[A]\in[\cA]_{\bQ}$ for $[A]\otimes 1$ when $A\in \cA$.

The following results are standard, or deduced in the standard way (see, for example, \cite[Rmk 4.5]{Jan4} and the remarks at the end of Subsection~\ref{ss3.1}).

\begin{prop}
	Let $s,\lambda$ and $\mu$ be as in Notation~\ref{NOT}. The following are $\bQ$-linearly independent subsets of $[\sC_\chi(\lambda)]_{\bQ}$:
	\begin{enumerate}
		\item $\{[L_\chi(w\cdot\lambda)] \mid w\in W^{I,\lambda}\}$.
		\item $\{[Z_\chi(w\cdot\lambda)]\mid w\in W^{I,\lambda}\}$.
		\item $\{[\overline{\nabla}(w\cdot\lambda)] \mid w\in W^{I,\lambda}\}$.
		\item $\{[Q_\chi^I(w\cdot\lambda)]\mid w\in W^{I,\lambda}\}$.
		\item $\{[\nabla(w\cdot\lambda)] \mid w\in W^{I,\lambda}\}$.
		\item $\{[T_\chi(w\cdot\lambda)] \mid w\in W^{I,\lambda}\}$.	
	\end{enumerate}
\end{prop}

\begin{prop}\label{GrothMuBas}
	Let $s,\lambda$ and $\mu$ be as in Notation~\ref{NOT}. The following are $\bQ$-linearly independent subsets of $[\sC_\chi(\mu)]_{\bQ}$:
	\begin{enumerate}
		\item $\{[L_\chi(w\cdot\mu)] \mid w\in W^{I,\mu}\}$.
		\item $\{[Z_\chi(w\cdot\mu)]\mid w\in W^{I,\mu}\}$.
		\item $\{[\overline{\nabla}(w\cdot\mu)]\mid w\in W^{I,\mu}\}$.
		\item $\{[Q_\chi^I(w\cdot\mu)]\mid w\in W^{I,\mu}\}$.
		\item $\{[\nabla(w\cdot\mu)] \mid w\in W^{I,\mu}\}$.
		\item $\{[T_\chi(w\cdot\mu)] \mid w\in W^{I,\mu}\}$.	
	\end{enumerate}
\end{prop}

In fact, using the duality $\bD$, we have $[Q_\chi^I(w\cdot\lambda)]=[\nabla(w\cdot\lambda)]$ and $[Z_\chi^I(w\cdot\lambda)]=[\overline{\nabla}(w\cdot\lambda)]$ for all $w\in W^{I,\lambda}$ (and, of course, analogous equalities for $\mu$).

The following lemma will be useful in proving the proposition which follows it, and should be compared with \cite[Prop 7.19]{RAGS}.

\begin{prop}\label{nzhom}
	Let $s,\lambda$ and $\mu$ be as in Notation~\ref{NOT}, and let $w\in W^{I,\lambda}$ such that $wsw^{-1}\notin W_{I,p}$. Then $$\Hom_{\sC_\chi}(\Delta(ws\cdot\lambda),\Delta(w\cdot\lambda))\neq 0.$$
\end{prop}

\begin{proof}
	We know that $T_\mu^\lambda(\Delta(w\cdot\mu))$ has a standard filtration consisting of one copy of $\Delta(w\cdot\lambda)$ and one copy of $\Delta(ws\cdot\lambda)$. Since we know from \cite[Lem 3.44]{BS} that $\Ext_{\sC_\chi}^{1}(\Delta(\gamma),\Delta(\kappa))\neq 0$ implies $\gamma\preceq\kappa$, there must exist an exact sequence $$0\to \Delta(w\cdot\lambda)\xrightarrow{\theta} T_\mu^\lambda(\Delta(w\cdot\mu))\to \Delta(ws\cdot\lambda)\to 0.$$ This then induces an exact sequence $$0\to \Hom_{\sC_\chi}(\Delta(ws\cdot\lambda),\Delta(w\cdot\lambda))\to\Hom_{\sC_\chi}(T_\mu^\lambda(\Delta(w\cdot\mu)),\Delta(w\cdot\lambda))\xrightarrow{\Psi} \Hom_{\sC_\chi}(\Delta(w\cdot\lambda),\Delta(w\cdot\lambda))$$ and therefore to prove our claim it suffices to show that 
	$\Psi$ has non-zero kernel.
	
	Note now that $$\dim\Hom_{\sC_\chi}(T_\mu^\lambda(\Delta(w\cdot\mu)),\overline{\Delta}(w\cdot\lambda))=\dim\Hom_{\sC_\chi}(\Delta(w\cdot\mu),T_\lambda^\mu(\overline{\Delta}(w\cdot\lambda)))=\dim\Hom_{\sC_\chi}(\Delta(w\cdot\mu),\overline{\Delta}(w\cdot\mu))\neq 0$$ using Proposition~\ref{TransZ} and the fact that we know $\Delta(w\cdot\mu)$ has a filtration by copies of $\overline{\Delta}(w\cdot\mu)$. Let $0\neq g\in\Hom_{\sC_\chi}(T_\mu^\lambda(\Delta(w\cdot\mu)),\overline{\Delta}(w\cdot\lambda))$, and consider $0\neq\widehat{g}\coloneqq \widehat{\iota}_{w\cdot\lambda}\circ g\in \Hom_{\sC_\chi}(T_\mu^\lambda(\Delta(w\cdot\mu)),\Delta(w\cdot\lambda))$. 
	
	{\bf Claim:} $\Psi(\widehat{g})=0$, i.e. $\widehat{g}\circ\theta=0$.
	
	{\bf Proof of Claim:} Since $\overline{\Delta}(w\cdot\lambda)$ has irreducible head $L_\chi(w\cdot\lambda)$ we either have $\im(\widehat{g})=\overline{\Delta}(w\cdot\lambda)$ or $\overline{\Delta}(w\cdot\lambda)/\im(\widehat{g})\twoheadrightarrow L_\chi(w\cdot\lambda)$. Noting that $$\Hom_{\sC_\chi}(T_\mu^\lambda(\Delta(w\cdot\mu)),L_\chi(w\cdot\lambda))=\Hom_{\sC_\chi}(\Delta(w\cdot\mu),T_\lambda^\mu(L_\chi(w\cdot\lambda)))=\Hom_{\sC_\chi}(\Delta(w\cdot\mu),0)=0$$ by \cite[4.9(b)]{Jan4} we can't have $\im(\widehat{g})=\overline{\Delta}(w\cdot\lambda)$, and so we must have $\overline{\Delta}(w\cdot\lambda)/\im(\widehat{g})\twoheadrightarrow L_\chi(w\cdot\lambda)$. In particular, $[\im(\widehat{g}): L_\chi(w\cdot\lambda)]=0$. 
	
	This therefore implies that the composition $$\Delta(w\cdot\lambda)\xrightarrow{\theta} T_\mu^\lambda(\Delta(w\cdot\mu))\xrightarrow{\widehat{g}} \overline{\Delta}(w\cdot\lambda)\twoheadrightarrow L_\chi(w\cdot\lambda)$$ is zero, and hence that the composition  $$\Delta(w\cdot\lambda)_{w\cdot\lambda+\bZ I}\xrightarrow{\theta_{w\cdot\lambda+\bZ I}} T_\mu^\lambda(\Delta(w\cdot\mu))_{w\cdot\lambda+\bZ I}\xrightarrow{\widehat{g}_{w\cdot\lambda+\bZ I}} \overline{\Delta}(w\cdot\lambda)_{w\cdot\lambda+\bZ I}\xrightarrow{\sim} L_\chi(w\cdot\lambda)_{w\cdot\lambda+\bZ I}$$ is zero (using \cite{Jan4} for the observation that the last map becomes an isomorphism on the $(w\cdot\lambda+\bZ I)$-graded part). In other words, we have that $(\widehat{g}\circ\theta)_{w\cdot\lambda+\bZ I}=0$. Since $\widehat{g}\circ\theta$ is a module homomorphism $Q_\chi^I(w\cdot\lambda)=\Delta(w\cdot\lambda)\to \overline{\Delta}(w\cdot\lambda)$ and $Q_\chi^I(w\cdot\lambda)$ is generated as a $U_\chi(\g)$-module by $Q_\chi^I(w\cdot\lambda)_{w\cdot\lambda+\bZ I}$, we must therefore have that $\widehat{g}\circ \theta=0$, i.e. that $\Psi(\widehat{g})=0$ as required.
	\end{proof}

The proof of the following proposition should be compared with \cite[II.E.11]{RAGS}.
\begin{prop}\label{TransTilt}
	Let $s,\lambda$ and $\mu$ be as in Notation~\ref{NOT}, and let $w\in W^{I,\lambda}$.
	\begin{enumerate}
		\item If $wsw^{-1}\notin W_{I,p}$ and $ws\cdot\lambda<w\cdot\lambda$ then $$T_\mu^\lambda(T_\chi(w\cdot\mu))=T_\chi(w\cdot\lambda),$$  $$T_\lambda^\mu(T_\chi(w\cdot\lambda))=T_\chi(w\cdot\mu)\oplus T_\chi(w\cdot\mu)$$ and $$T_\lambda^\mu(T_\chi(ws\cdot\lambda))=T_\chi(w\cdot\mu)\oplus T_{\prec w\cdot\mu}$$ where $T_{\prec w\cdot\mu}$ is a tilting module which decomposes as a direct sum of indecomposable tilting modules indexed by $\kappa\in \Lambda_I$ with $\kappa\prec w\cdot \mu$.
		\item If $wsw^{-1}\in W_{I,p}$ then $$T_\lambda^\mu(T_\chi(
		w\cdot\lambda))=T_\chi(w\cdot\mu)\oplus T_\chi(w\cdot\mu)$$ and $$T_\mu^\lambda(T_\chi(w\cdot\mu))=T_\chi(w\cdot\lambda).$$
	\end{enumerate}
\end{prop}

\begin{proof}
	Note that, in (1), $ws\cdot\lambda\in C_I\cap X(T)$ by Proposition~\ref{ReflI}. With that in mind, we prove both statements at the same time. We do so in 5 steps.
	
	{\bf Step 1:} $T_\lambda^\mu T_\mu^\lambda(T_\chi(w\cdot\mu))=T_\chi(w\cdot\mu)\oplus T_\chi(w\cdot\mu)$.

	{\bf Proof of Step 1:} Since translation functors are exact, they induce linear maps on the level of the Grothendieck groups (which we write as $[T_\lambda^\mu]$, $[T_\mu^\lambda]$). By Propositions~\ref{TransFunct} and \ref{TransFunct2} we have $$[T_\lambda^\mu][T_\mu^\lambda]([Q_\chi^I(u\cdot\mu)])=2[Q_\chi^I(u\cdot\mu)]$$ for all $u\in W^{I,\mu}$. We know (see the end of Subsection~\ref{ss3.1}) that we can write $$[T_\chi(w\cdot\mu)]=[Q_\chi^I(w\cdot\mu)]+\sum_{u\cdot\mu\prec w\cdot\mu} m_{u\cdot\mu} [Q_\chi^I(u\cdot\mu)]$$ for some $m_{u\cdot\mu}\in\bZ_{\geq 0}$, where the sum is over all $u\in W^{I,\mu}$ with $u\cdot\mu\prec w\cdot\mu$ (as in the equation above, we usually abbreviate this to $u\cdot\mu\prec w\cdot\mu$ in subscripts to prevent clutter). We therefore have $$[T_\lambda^\mu][T_\mu^\lambda]([T_\chi(w\cdot\mu)])=2[Q_\chi^I(w\cdot\mu)]+\sum_{ u\cdot\mu\prec w\cdot\mu} 2m_{u\cdot\mu} [Q_\chi^I(u\cdot\mu)]=2[T_\chi(w\cdot\mu)].$$ On the other hand, since translation functors send tilting modules to tilting modules (Corollary~\ref{TransTilting}), there exist some $n_{u\cdot\mu}\in\bZ_{\geq 0}$ such that $$T_\lambda^\mu T_\mu^\lambda(T_{\chi}(w\cdot\mu))=\bigoplus_{u\in W^{I,\mu}} T_\chi(u\cdot\mu)^{n_{u\cdot\mu}}$$ and thus $$[T_\lambda^\mu][T_\mu^\lambda]([T_\chi(w\cdot\mu)])=\sum_{u\in W^{I,\mu}} n_{u\cdot \mu} [T_\chi(u\cdot\mu)].$$ Putting this all together, we get
	$$2[T_\chi(w\cdot\mu)]=\sum_{u\in W^{I,\mu}} n_{u\cdot \mu} [T_\chi(u\cdot\mu)]$$ and therefore Proposition~\ref{GrothMuBas} implies that 
	$$T_\lambda^\mu T_\mu^\lambda(T_\chi(w\cdot\mu))=T_\chi(w\cdot\mu)\oplus T_\chi(w\cdot\mu).$$

	{\bf Step 2:} $T_\chi(w\cdot\lambda)$ is a direct summand of $T_\mu^\lambda(T_\chi(w\cdot\mu))$.

	{\bf Proof of Step 2:} Using what we know about how each $[T_\chi(\kappa)]$ can be written in terms of $[\nabla(\gamma)]$ in the Grothendieck group of $\sC_\chi$, and the fact that $T_\mu^\lambda(T_\chi(w\cdot\mu))$ can be written as a direct sum of indecomposable tilting modules, it is straightforward to check that Step 2 follows once we show that $$(T_\mu^\lambda(T_\chi(w\cdot\mu)):\nabla(w\cdot\lambda))=1,\quad \mbox{and}\quad (T_\mu^\lambda(T_\chi(w\cdot\mu)):\nabla(u\cdot\lambda))\neq 0 \implies w\cdot\lambda\nprec u\cdot\lambda.$$

	 Let $u\in W^{I,\lambda}$. We then have
	\begin{equation*}
		\begin{split}
			(T_\mu^\lambda(T_\chi(w\cdot\mu)):\nabla(u\cdot\lambda)) & = \dim\Hom_{\sC_\chi}(\overline{\Delta}(u\cdot\lambda),T_\mu^\lambda(T_\chi(w\cdot\mu))) \\
			& = \dim\Hom_{\sC_\chi}(T_\lambda^\mu(\overline{\Delta}(u\cdot\lambda)),T_\chi(w\cdot\mu)) \\
			& = \dim\Hom_{\sC_\chi}(\overline{\Delta}(u\cdot\mu),T_\chi(w\cdot\mu)) \\ 
			& = (T_\chi(w\cdot\mu):\nabla(u\cdot\mu)).
		\end{split}
	\end{equation*}
	From this, it follows immediately that  $$(T_\mu^\lambda(T_\chi(w\cdot\mu)):\nabla(w\cdot\lambda))=1.$$ 
	Suppose that $w\cdot\lambda\prec u\cdot\lambda$ but $(T_\mu^\lambda(T_\chi(w\cdot\mu)):\nabla(w\cdot\lambda))\neq 0$. By \cite[\S II.6.5]{RAGS}, $$w\cdot\lambda\prec u\cdot\lambda \implies w\cdot\overline{C}\prec u\cdot\overline{C}\implies w\cdot\mu\preceq u\cdot\mu,$$ (where the partial order on alcoves is as given in \cite[ II.6.5]{RAGS}). On the other hand, the above calculation shows that $$(T_\mu^\lambda(T_\chi(w\cdot\mu)):\nabla(w\cdot\lambda))\neq 0\implies u\cdot\mu\preceq w\cdot\mu.$$ We thus get $u\cdot\mu=w\cdot\mu$.
	
	If $wsw^{-1}\in W_{I,p}$ this forces $u=w$, while if $wsw^{-1}\notin W_{I,p}$ we have either $u=w$ or $u=ws$. Since, by assumption, $ws\cdot\lambda < w\cdot\lambda\prec u\cdot\lambda$ the latter cannot happen. We must therefore have $u=w$. But this also cannot happen, since $w\cdot\lambda\prec u\cdot\lambda$. This proves Step 2.

	{\bf Step 3:} $T_\lambda^\mu(T_\chi(w\cdot\lambda))=T_\chi(w\cdot\mu)\oplus T_\chi(w\cdot\mu)$.

	{\bf Proof of Step 3:} By Step 2, $T_\chi(w\cdot\lambda)$ is a direct summand of $T_\mu^\lambda(T_\chi(w\cdot\mu))$ and hence $T_\lambda^\mu (T_\chi(w\cdot\lambda))$ is a direct summand of $T_\lambda^\mu(T_\mu^\lambda(T_\chi(w\cdot\mu)))=T_\chi(w\cdot\mu)\oplus T_\chi(w\cdot\mu)$. Therefore, either $T_\lambda^\mu (T_\chi(w\cdot\lambda))=T_\chi(w\cdot\mu)$ or $T_\lambda^\mu (T_\chi(w\cdot\lambda))=T_\chi(w\cdot\mu)\oplus T_\chi(w\cdot\mu)$. To determine which, we simply need to compute $(T_\lambda^\mu(T_\chi(w\cdot\lambda)):\nabla(w\cdot\mu))$. This gives us:
	\begin{equation*}
	\begin{split}
		(T_\lambda^\mu(T_\chi(w\cdot\lambda)):\nabla(w\cdot\mu)) & = \dim\Hom_{\sC_\chi}(\overline{\Delta}(w\cdot\mu),T_\lambda^\mu(T_\chi(w\cdot\lambda))) \\
		& = \dim\Hom_{\sC_\chi}(T_\mu^\lambda(\overline{\Delta}(w\cdot\mu)),T_\chi(w\cdot\lambda)) \\
		& = \dim\Hom_{\sC_\chi}(\overline{\Delta}(w\cdot\lambda),T_\chi(w\cdot\lambda))+\dim\Hom_{\sC_\chi}(\overline{\Delta}(ws\cdot\lambda),T_\chi(w\cdot\lambda)) \\ & = (T_\chi(w\cdot\lambda):\nabla(w\cdot\lambda))+(T_\chi(w\cdot\lambda):\nabla(ws\cdot\lambda)).
	\end{split}
\end{equation*}
 For the third equality, we have used the fact that $T_\mu^\lambda(\overline{\Delta}(w\cdot\mu))$ lies in an exact sequence $$0\to M\to T_\mu^\lambda(\overline{\Delta}(w\cdot\mu))\to N \to 0$$ where one of $M$ or $N$ is $\overline{\Delta}(w\cdot\lambda)$ and the other is $\overline{\Delta}(ws\cdot\lambda)$, and thus we get the equality from the exact sequence $$0\to\Hom_{\sC_\chi}(N, T_\chi(w\cdot\lambda))\to\Hom_{\sC_\chi}(T_\mu^\lambda(\overline{\Delta}(w\cdot\mu)), T_\chi(w\cdot\lambda))\to\Hom_{\sC_\chi}(M, T_\chi(w\cdot\lambda))\to\Ext^1_{\sC_\chi}(N, T_\chi(w\cdot\lambda))=0.$$ Here we have used the fact that there are no non-split extensions of a module with a costandard filtration by a module with a proper standard filtration, as we observed earlier. 
 
 If $wsw^{-1}\in W_{I,p}$ then we have $\nabla(ws\cdot\lambda)\cong \nabla(w\cdot\lambda)$ and thus $(T_\chi(w\cdot\lambda):\nabla(w\cdot\lambda))+(T_\chi(w\cdot\lambda):\nabla(ws\cdot\lambda))=2$, which implies that $T_\lambda^\mu (T_\chi(w\cdot\lambda))=T_\chi(w\cdot\mu)\oplus T_\chi(w\cdot\mu)$ as required. If $wsw^{-1}\notin W_{I,p}$ then we have $(T_\chi(w\cdot\lambda):\nabla(w\cdot\lambda))=1$ and we need to compute $(T_\chi(w\cdot\lambda):\nabla(ws\cdot\lambda))$.

What we have already seen implies that $(T_\chi(w\cdot\lambda):\nabla(ws\cdot\lambda))\leq 1$ and thus to conclude Step 3 it suffices to show that it is non-zero. This will follow if we can show that $\dim\Hom_{\sC_\chi}(\overline{\Delta}(ws\cdot\lambda),T_\chi(w\cdot\lambda))\neq 0$, which will follow if $\dim\Hom_{\sC_\chi}(\overline{\Delta}(ws\cdot\lambda),\Delta(w\cdot\lambda))\neq 0$ (using the inclusion $\Delta(w\cdot\lambda)\hookrightarrow T_\chi(w\cdot\lambda)$). This follows easily from Proposition~\ref{nzhom}. 

We thus conclude that $(T_\lambda^\mu(T_\chi(w\cdot\lambda)):\nabla(w\cdot\mu))=(T_\chi(w\cdot\lambda):\nabla(w\cdot\lambda))+(T_\chi(w\cdot\lambda):\nabla(ws\cdot\lambda))=2$ and hence $T_\lambda^\mu (T_\chi(w\cdot\lambda))=T_\chi(w\cdot\mu)\oplus T_\chi(w\cdot\mu)$.

{\bf Step 4:}  $T_\mu^\lambda(T_\chi(w\cdot\mu))=T_\chi(w\cdot\lambda)$.

{\bf Proof of Step 4:} By Step 2, we know that there exists a tilting module $T\in\sC_\chi$ such that $T_\mu^\lambda(T_\chi(w\cdot\mu))=T_\chi(w\cdot\lambda)\oplus T$. Furthermore, by Step 3 we know that $$T_\chi(w\cdot\mu)\oplus T_\chi(w\cdot\mu)=T_\lambda^\mu(T_\mu^\lambda(T_\chi(w\cdot\mu)))=T_\lambda^\mu(T_\chi(w\cdot\lambda))\oplus T_\lambda^\mu(T)=T_\chi(w\cdot\mu)\oplus T_\chi(w\cdot\mu)\oplus T_\lambda^\mu(T).$$ Therefore, $T_\lambda^\mu(T)=0$. Since $T$ has a standard filtration it has a proper standard filtration and thus, unless $T=0$, there exists $u\in W^{I,\mu}$ such that $(T:\overline{\Delta}(w\cdot\mu))\neq 0$. In this case, $T_\mu^\lambda(\overline{\Delta}(w\cdot\mu))$ appears in a filtration of $T_\mu^\lambda(T)=0$. Since $T_\mu^\lambda(\overline{\Delta}(w\cdot\mu))\neq 0$ this is impossible, and hence we must have $T=0$.

{\bf Step 5:} $T_\lambda^\mu(T_\chi(ws\cdot\lambda))=T_\chi(w\cdot\mu)\oplus T_{\prec w\cdot\mu}$, where $T_{\prec w\cdot\mu}$ is a tilting module which decomposes as a direct sum of indecomposable tilting modules indexed by $\kappa\in X(T)$ with $\kappa\prec w\cdot \mu$.

{\bf Proof of Step 5:}

By similar arguments to those we have already seen, it is enough to show that $(T_\lambda^\mu(T_\chi(ws\cdot\lambda)):\nabla(w\cdot\mu))=1$ and that $u\cdot\mu\prec  w\cdot\mu$ whenever $(T_\lambda^\mu(T_\chi(ws\cdot\lambda)):\nabla(u\cdot\mu))\neq 0$ and $\nabla(u\cdot\mu)\ncong \nabla(w\cdot\mu)$. We can compute
\begin{equation*}
	\begin{split}
		(T_\lambda^\mu(T_\chi(ws\cdot\lambda)):\nabla(u\cdot\mu)) & = \dim\Hom_{\sC_\chi}(\overline{\Delta}(u\cdot\mu),T_\lambda^\mu(T_\chi(ws\cdot\lambda)) \\ & 
		= \dim\Hom_{\sC_\chi}(T_\mu^\lambda(\overline{\Delta}(u\cdot\mu)),T_\chi(ws\cdot\lambda)) \\ & = \dim\Hom_{\sC_\chi}(\overline{\Delta}(u\cdot\lambda),T_\chi(ws\cdot\lambda)) + \dim\Hom_{\sC_\chi}(\overline{\Delta}(us\cdot\lambda),T_\chi(ws\cdot\lambda)) \\ & = 
		(T_\chi(ws\cdot\lambda):\nabla(u\cdot\lambda))+ (T_\chi(ws\cdot\lambda):\nabla(us\cdot\lambda)).
	\end{split}
\end{equation*}

If $u\cdot\mu=w\cdot\mu$ then we either have $u=w$ or $u=ws$. In either case we have $(T_\lambda^\mu(T_\chi(ws\cdot\lambda):\nabla(u\cdot\mu)))=(T_\chi(ws\cdot\lambda):\nabla(w\cdot\lambda))+ (T_\chi(ws\cdot\lambda):\nabla(ws\cdot\lambda)) =1$ since $ws\cdot\lambda < w\cdot\lambda$. 

Furthermore, if $(T_\chi(ws\cdot\lambda):\nabla(u\cdot\lambda))+ (T_\chi(ws\cdot\lambda):\nabla(us\cdot\lambda))\neq 0$ then either $u\cdot\lambda\preceq ws\cdot\lambda $ or $us\cdot\lambda \preceq ws\cdot\lambda$ (or both). These then imply that either $u\cdot \overline{C}\preceq ws\cdot\overline{C} $ or $us\cdot\overline{C} \preceq ws\cdot\overline{C}$ (or both), and thus that either $u\cdot \mu\preceq ws\cdot\mu $ or $us\cdot\mu \preceq ws\cdot\mu$. Both of these inequalities are precisely $u\cdot \mu \preceq w\cdot \mu$.

\end{proof}

\begin{cor}\label{TransTilt2}
	Let $s,\lambda$ and $\mu$ be as in Notation~\ref{NOT}, and let $w\in W^{I,\lambda}$. If $wsw^{-1}\notin W_{I,p}$ and $w\cdot\lambda<ws\cdot\lambda$ then $$T_\mu^\lambda(T_\chi(w\cdot\mu))=T_\chi(ws\cdot\lambda),$$  $$T_\lambda^\mu(T_\chi(ws\cdot\lambda))=T_\chi(w\cdot\mu)\oplus T_\chi(w\cdot\mu)$$ and $$T_\lambda^\mu(T_\chi(w\cdot\lambda))=T_\chi(w\cdot\mu)\oplus T_{\prec w\cdot\mu},$$ where $T_{\prec w\cdot\mu}$ is a tilting module which decomposes as a direct sum of indecomposable tilting modules indexed by $\kappa\in \Lambda_I$ with $\kappa\prec w\cdot \mu$.
\end{cor}

\begin{cor}
	Let $s,\lambda$ and $\mu$ be as in Notation~\ref{NOT}, and let $w\in W^{I,\lambda}$.
	\begin{enumerate}
		\item If $wsw^{-1}\notin W_{I,p}$ and $ws\cdot\lambda<w\cdot\lambda$ then $$\Theta_s(T_\chi(w\cdot\lambda))=T_\chi(w\cdot\lambda)\oplus T_\chi(w\cdot\lambda)$$ and $$\Theta_s(T_\chi(ws\cdot\lambda))=T_\chi(w\cdot\lambda)\oplus T_\mu^\lambda(T_{\prec w\cdot\mu})$$  where $T_{\prec w\cdot\mu}$ is a tilting module which decomposes as a direct sum of indecomposable tilting modules indexed by $\kappa\in \Lambda_I$ with $\kappa\prec w\cdot \mu$.
		\item If $wsw^{-1}\in W_{I,p}$ then $$\Theta_s(T_\chi(
		w\cdot\lambda))=T_\chi(w\cdot\lambda)\oplus T_\chi(w\cdot\lambda).$$
	\end{enumerate}
\end{cor}

It will be useful to refine this corollary a little, using ideas that can be found in \cite[II.6.6]{RAGS}. Let $w\in W_p$. There exists, for each $\alpha\in \Phi^{+}$, an integer $n_\alpha$ such that $$n_\alpha p< \langle w\cdot\lambda+\rho,\alpha^\vee\rangle < (n_\alpha+1)p.$$ We then define $$d(w\cdot\lambda)\coloneqq\sum_{\alpha\in\Phi^{+}} n_\alpha.$$ It is shown in \cite[Lem. II.6.6]{RAGS} that for any reflection $t$ in $W_p$ we have $$d(tw\cdot\lambda) < d(w\cdot\lambda) \quad \iff \quad tw\cdot\lambda \prec w\cdot\lambda.$$ We particularly note the following:
\begin{lemma}\label{ReflOrd}
	Let $w\in W_p$ and $\lambda\in C\cap X(T)$, and let $s_{\alpha,np}$ be a reflection in a wall of the alcove containing $w\cdot\lambda$. If $s_{\alpha,np}w\cdot\lambda\prec w\cdot\lambda$ then $d(s_{\alpha,np}w\cdot\lambda)=d(w\cdot\lambda)-1$.
\end{lemma}

\begin{proof}
	This is immediate since there is a single reflection hyperplane separating $s_{\alpha,np}w\cdot\lambda$ and $w\cdot\lambda$. See also \cite[Lem. II.6.7]{RAGS}.
\end{proof}

The following statement will also be useful.

\begin{lemma}\label{PrecML}
	Let $s,\lambda$ and $\mu$ be as in Notation~\ref{NOT}. Let $w,v\in W_p$ such that $w\cdot\lambda\prec ws\cdot\lambda$. Then $$w\cdot\mu\prec v\cdot\mu \quad \implies \quad w\cdot\lambda\prec v\cdot\lambda.$$
\end{lemma}

\begin{proof}
	We first show that if $t\in W_p$ is a reflection then $$tv\cdot\mu\prec v\cdot\mu \quad \implies \quad tv\cdot\lambda \prec v \cdot\lambda.$$ For each $\alpha\in\Phi^{+}$ we set $n_\alpha\in\bZ $ to be such that $$n_\alpha p < \langle v\cdot\lambda +\rho,\alpha^\vee\rangle < (n_\alpha+1)p.$$ In particular, there exists a unique root $\beta\in\Phi^{+}$ such that either ({\bf Case 1}) $$\langle v\cdot\mu+\rho,\beta^\vee\rangle=n_\beta p$$ or ({\bf Case 2}) $$\langle v\cdot\mu+\rho,\beta^\vee\rangle=(n_\beta+1) p.$$ In the former case we have $vsv^{-1}=s_{\beta,n_\beta p}$, while in the latter case we have $vsv^{-1}=s_{\beta,(n_\beta+1)p}$. In each case, we have $$n_\alpha p < \langle v\cdot\mu +\rho,\alpha^\vee\rangle < (n_\alpha+1)p$$ for $\alpha\neq\beta$. 
	
	Let us write $t=s_{\alpha,mp}$. Then $s_{\alpha,mp}v\cdot\mu\prec v\cdot\mu$ implies that $s_{\alpha,mp}\neq vsv^{-1}$ and $\langle v\cdot\mu+\rho,\alpha^\vee\rangle \geq mp$. If $\alpha\neq \beta$ this implies that $m\leq n_\alpha$ and thus that $\langle v\cdot\lambda+\rho,\alpha^\vee\rangle \geq n_\alpha p \geq mp$. Hence $tv\cdot\lambda\prec v\cdot\lambda$. On the other hand, if $\alpha=\beta$ then we consider {\bf Case 1} and {\bf Case 2} separately. In {\bf Case 1} we have  $\langle v\cdot\mu+\rho,\alpha^\vee\rangle\geq mp$ implies $m\leq n_\beta$ and thus $\langle v\cdot\lambda+\rho,\beta^\vee\rangle \geq n_\beta p \geq mp$. Hence $tv\cdot\lambda\prec v\cdot\lambda$. In {\bf Case 2} we have $\langle v\cdot\mu+\rho,\alpha^\vee\rangle\geq mp$ implies $m\leq n_\beta+1$. If $m\leq n_\beta$ then $\langle v\cdot\lambda+\rho,\beta^\vee\rangle \geq n_\beta p \geq mp$, hence $tv\cdot\lambda\prec v\cdot\lambda$, while if $m=n_\beta+1$ we have $t=s_{\beta,(n_\beta+1)p}=vsv^{-1}$, a contradiction.
	
	We therefore see that $$tv\cdot\mu\prec v\cdot\mu \quad \implies \quad tv\cdot\lambda \prec v \cdot\lambda,$$ as required.
	
	Now, suppose that $$w\cdot\mu\prec v\cdot\mu$$ for $w,v\in W_p$ as in the statement of the lemma. This means that there exist $\alpha_1,\ldots,\alpha_k\in\Phi^{+}$ and $n_1,\ldots,n_k\in\bZ$ such that $$w\cdot\mu=s_{\alpha_1,n_1p}\cdots s_{\alpha_k,n_kp}v\cdot\mu \prec s_{\alpha_2,n_2p}\cdots s_{\alpha_k,n_kp}v\cdot\mu\prec\cdots \prec s_{\alpha_k,n_kp}v\cdot\mu\prec v\cdot\mu.$$ Applying the above argument, we thus get $$s_{\alpha_1,n_1p}\cdots s_{\alpha_k,n_kp}v\cdot\lambda \prec s_{\alpha_2,n_2p}\cdots s_{\alpha_k,n_kp}v\cdot\lambda\prec\cdots \prec s_{\alpha_k,n_kp}v\cdot\lambda\prec v\cdot\lambda.$$ Since $s_{\alpha_1,n_1p}\cdots s_{\alpha_k,n_kp}v\cdot\mu=w\cdot\mu$, we have $s_{\alpha_1,n_1p}\cdots s_{\alpha_k,n_kp}v\cdot\lambda$ equals either $w\cdot\lambda$ or $ws\cdot\lambda$. As we know that $w\cdot\lambda\prec ws\cdot\lambda$, either way we get $$w\cdot\lambda \prec v\cdot\lambda$$ as required.
\end{proof}

The reader should note that this proof works equally well if we replace $\lambda$ with $s\cdot\lambda$, in which case we get the following corollary.

\begin{cor}\label{PrecML2}
	Let $s,\lambda$ and $\mu$ be as in Notation~\ref{NOT}. Let $w,v\in W_p$ such that $w\cdot\lambda\prec ws\cdot\lambda$. Then $$w\cdot\mu\prec v\cdot\mu \quad \implies \quad w\cdot\lambda\prec vs\cdot\lambda.$$
\end{cor}

We know by Corollary~\ref{TransTilting} that translation and wall-crossing functors send tilting module to tilting modules, and that each tilting module decomposes as a direct sum of indecomposable tilting modules. In particular, in the set-up of Proposition~\ref{TransTilt} with $wsw^{-1}\notin W_{I,p}$ and $ws\cdot\lambda < w\cdot\lambda$, we have $$T_\lambda^\mu(T_\chi(ws\cdot\lambda))=T_\chi(w\cdot\mu)\oplus\bigoplus_{\substack{ u\in W^{I,\mu} \\ u\cdot\mu\prec w\cdot\mu }} T_\chi(u\cdot\mu)^{\oplus m_u}$$ for some $m_u\geq 0$. Applying Proposition~\ref{TransTilt}(1) we also get 
$$\Theta_s(T_\chi(ws\cdot\lambda))=T_\chi(w\cdot\lambda)\oplus\bigoplus_{\substack{ u\in W^{I,\mu} \\ u\cdot\mu\prec w\cdot\mu }} T_\chi(u\cdot\lambda)^{\oplus m_u}.$$

Note that, in this sum, $u\in W^{I,\mu}$ and $u\cdot\mu\prec w\cdot\mu$ imply that $us\cdot\lambda\prec w\cdot\lambda$ and $us\cdot\lambda\prec ws\cdot\lambda$ (by Lemma~\ref{PrecML} and Corollary~\ref{PrecML2}). In particular, this means that $d(u\cdot\lambda)=d(us\cdot\lambda)+1<d(ws\cdot\lambda)+1=d(w\cdot\lambda)$. Hence we can alternatively write $$\Theta_s(T_\chi(ws\cdot\lambda))=T_\chi(w\cdot\lambda)\oplus\bigoplus_{\substack{ u\in W_p \\ d(u\cdot\lambda)< d(w\cdot\lambda) }} T_\chi(u\cdot\lambda)^{\oplus m_u}$$ (where we define $m_u$ to be zero for any $u\in W_p\setminus W^{I,\mu}$ or any $u\in W^{I,\mu}$ such that $d(u\cdot\lambda)<d(w\cdot\lambda)$ but $u\cdot\mu\nprec w\cdot\mu$). We summarise this in the following corollary.

\begin{cor}\label{RefTransWall}
	Let $s,\lambda$ and $\mu$ be as in Notation~\ref{NOT}, and let $w\in W^{I,\lambda}$.
	\begin{enumerate}
		\item If $wsw^{-1}\notin W_{I,p}$ and $ws\cdot\lambda<w\cdot\lambda$ then $$\Theta_s(T_\chi(w\cdot\lambda))=T_\chi(w\cdot\lambda)\oplus T_\chi(w\cdot\lambda)$$ and $$\Theta_s(T_\chi(ws\cdot\lambda))=T_\chi(w\cdot\lambda)\oplus T_{< d(w\cdot\lambda)}$$  where $T_{< d(w\cdot\lambda)}$ is a tilting module which decomposes as a direct sum of indecomposable tilting modules indexed by $u\cdot\lambda$ for $u\in W^{I,\lambda}$ with $d(u\cdot\lambda)< d(w\cdot \lambda)$.
		\item If $wsw^{-1}\in W_{I,p}$ then $$\Theta_s(T_\chi(
		w\cdot\lambda))=T_\chi(w\cdot\lambda)\oplus T_\chi(w\cdot\lambda).$$
	\end{enumerate}
\end{cor}

This refined version of Corollary~\ref{TransTilt2} can be more useful to work with in practice, as the following propositions show.

\begin{prop}\label{DomExp}
	Let $s,\lambda$ and $\mu$ be as in Notation~\ref{NOT} and let $w\in W^{I,\lambda}$. If $w\cdot\lambda\in X(T)_+$ then  $\lambda\preceq w\cdot\lambda$. Furthermore, if $w\cdot\lambda\in X(T)_+$ then there exist $s_1,\ldots,s_n\in S_p$ such that $w=s_1 s_2\cdots s_{n-1}s_n$, $$s_1 s_{2}\cdots s_i\cdot \lambda \prec s_1 s_{2}\cdots s_{i}s_{i+1}\cdot \lambda$$ and $$s_1 s_{2} \cdots s_{i-1} s_{i} s_{i-1} \cdots s_{2} s_1 \notin W_{I,p}$$ for all $i=0, \ldots, n$.
\end{prop}

\begin{proof}
	First, we show that for $w\in W_p$ we have $\lambda\preceq w\cdot\lambda$ whenever $w\cdot\lambda\in X(T)_+$. We use the same notation $n_\alpha$, for $\alpha\in\Phi^{+}$, as above. Note that if $w\cdot\lambda\in X(T)_{+}$ then $d(w\cdot\lambda)$ equals the number of hyperplanes in $X(T)\otimes_{\bZ}\bR$ (of the form $\langle \gamma+\rho,\beta^\vee\rangle=kp$ for $\beta\in\Phi^{+}$ and $k\in\bZ$) between $\lambda$ and $w\cdot\lambda$. We proceed by induction on $d(w\cdot\lambda)$; the claim is clearly true when $d(w\cdot\lambda)=0$.
	
	Let $\beta\in\Phi^{+}$ be a positive root such that $n_\beta\neq 0$ and such that the hyperplane $\langle \gamma+\rho,\beta^\vee\rangle = n_\beta p$ contains a wall of the alcove containing $w\cdot\lambda$ (we can always do this so long as $w\cdot\lambda\neq \lambda$). Then $d(s_{\beta,n_{\beta}p}w\cdot\lambda)<d(w\cdot\lambda)$ and thus by \cite[Lem. II.6.6]{RAGS} we have $s_{\beta,n_\beta p}w\cdot\lambda\preceq w\cdot\lambda$. Furthermore, by construction, $s_{\beta,n_\beta p}w\cdot\lambda \in X(T)_{+}$ and thus by induction we have $\lambda\preceq s_{\beta,n_\beta p}w\cdot\lambda$. Putting these together, we get that $\lambda\preceq w\cdot\lambda$, as required. 
	
	We now claim that if $w\cdot\lambda\in C_I\cap X(T)_{+}$ then there exist positive roots $\beta_1,\ldots,\beta_k$ and integers $m_1,\ldots,m_k$ such that $$w\cdot\lambda=s_{\beta_1,m_1 p} s_{\beta_2,m_2 p} \cdots s_{\beta_k,m_k p}\cdot\lambda,$$ each $s_{\beta_i,m_i p}$ is a reflection in a wall of the alcove containing $s_{\beta_{i+1},m_{i+1} p} s_{\beta_{i+2},m_{i+2} p} \cdots s_{\beta_k,m_k p}\cdot\lambda$, and $$s_{\beta_i,m_i p}\notin W_{I,p}\qquad \mbox{for } i=1\ldots,k.$$
	The statement is immediate when $I=\Phi_s$ (since then $w=1$), and so we assume $I\neq \Phi_s$.
	
	As before, we proceed by induction on $d(w\cdot\lambda)$. The statement is clear when $d(w\cdot\lambda)=0$.
	
	Given $w\cdot\lambda\in C_I\cap X(T)_{+}$ with $d(w\cdot\lambda)>0$, there exists $\beta\in \Phi^{+}\setminus \bZ I$ such that $n_\beta\neq 0$ and such that the hyperplane $\langle \gamma+\rho,\beta^\vee\rangle = n_\beta p$ contains a wall of the alcove containing $w\cdot\lambda$. Indeed, we noted above that there exists $\beta\in \Phi^{+}$ with this property. By definition $n_\gamma=0$ for all $\gamma\in \Phi^{+}\cap \bZ I$, and thus $\beta\notin \Phi^{+}\cap\bZ I$. 
	
	Now, $d(s_{\beta,n_\beta p}w\cdot\lambda)<d(w\cdot\lambda)$, and thus by induction we can find the desired $\overline{\beta}_1, \ldots,\overline{\beta}_k$ and $\overline{m}_1,\ldots,\overline{m}_k$ for $s_{\beta,n_\beta p}w\cdot\lambda$. We thus get that $$w\cdot\lambda=s_{\beta,n_\beta p}s_{\overline{\beta}_1,\overline{m}_1 p}\cdots s_{\overline{\beta}_k,\overline{m}_k p}\cdot\lambda,$$ where $s_{\overline{\beta}_i,\overline{m}_i p}$ is reflection in a wall of the alcove containing $s_{\overline{\beta}_{i+1},\overline{m}_{i+1} p}\cdots s_{\overline{\beta}_k,\overline{m}_k p}\cdot\lambda$ for each $i=1,\ldots,k$ and $s_{\overline{\beta}_i,\overline{m}_i p}\notin W_{I,p}$. Furthermore, $s_{\beta,n_\beta p}$ is a reflection in a wall of the alcove containing $s_{\overline{\beta}_1,\overline{m}_1 p}\cdots s_{\overline{\beta}_k,\overline{m}_k p}\cdot\lambda=s_{\beta,n_\beta p}w\cdot\lambda$, and $s_{\beta,n_\beta p}\notin W_{I,p}$ since $\beta\notin \Phi^{+}\cap\bZ I$. This proves the claim. Furthermore, from our construction it is clear that $$s_{\overline{\beta}_i,\overline{m}_i p}\cdots s_{\overline{\beta}_k,\overline{m}_k p}\cdot\lambda\prec s_{\overline{\beta}_{i-1},\overline{m}_{i-1} p}\cdots s_{\overline{\beta}_k,\overline{m}_k p}\cdot\lambda$$ for all $i$.
	
	Now let $(\beta_1,m_1),\ldots,(\beta_k,m_k)$ be the appropriate roots and integers for $w\cdot\lambda\in X(T)_{+}$.	Set $n=k$ and define $$s_{i}=s_{\beta_k,m_k p}s_{\beta_{k-1},m_{k-1} p}\cdots s_{\beta_{k-i+2},m_{k-i+2} p}s_{\beta_{k-i+1},m_{k-i+1} p} s_{\beta_{k-i+2},m_{k-i+2} p}\cdots s_{\beta_{k-1},m_{k-1}}s_{\beta_k,m_k p}.$$ We can then see easily that $$s_1 s_{2}\cdots	s_{k-1} s_k\cdot\lambda=w\cdot\lambda.$$ Furthermore, the inequality 
	$$s_{\beta_{k-i+1},m_{k-i+1} p} s_{\beta_{k-i+2},m_{k-i+2} p}\cdots s_{\beta_k,m_k p}\cdot\lambda\prec s_{\beta_{k-i},m_{k-i} p} s_{\beta_{k-i+1},m_{k-i+1} p}\cdots s_{\beta_k,m_k p}\cdot\lambda$$ becomes 
	$$s_1 s_{2}\cdots s_{i}\cdot\lambda\prec s_{1}\cdots s_{i+1}\cdot\lambda$$
	for all $i$. We also have $$s_1 s_{2}\cdots s_{i-1}s_{i}s_{i-1}\cdots s_{2}s_{1}=s_{\beta_{k-i+1},m_{k-i+1} p}\notin W_{I,p}$$ for all $i$. Finally, we note that $s_{i}\in S_p$ for all $i$. This proves the result.
	
\end{proof}

For the rest of this subsection, suppose $wsw^{-1}\notin W_{I,p}$ and $w\cdot\lambda < ws\cdot\lambda$.\footnote{Throughout the rest of this paper we generally prefer the convention that $ws\cdot\lambda < w\cdot\lambda$, but for what we do here it makes more sense to flip this convention.} Corollary~\ref{RefTransWall} then says that
$$\Theta_s(T_\chi(w\cdot\lambda))=T_\chi(ws\cdot\lambda)\oplus\bigoplus_{\substack{ u\in W_p \\ d(u\cdot\lambda)< d(ws\cdot\lambda) }} T_\chi(u\cdot\lambda)^{\oplus m_u}$$ for some $m_u\geq 0$.

Now let $t\in W_p$ be a reflection in a wall of $\overline{C}$. We may form the translation functor $\Theta_t:\sC_\chi(\lambda)\to\sC_\chi(\lambda)$; this involves a choice of an element in $\overline{C}\cap X(T)$ with suitable stabiliser, but from now on we assume that, for each simple reflection in $W_p$, such an element has been fixed. Then, if $wstsw^{-1}\notin W_{I,p}$ and $ws\cdot\lambda<wst\cdot\lambda$, we get from  Corollary~\ref{RefTransWall} that 
\begin{equation*}
	\begin{split}
\Theta_t\Theta_s(T_\chi(w\cdot\lambda)) = & T_\chi(wst\cdot\lambda)\oplus\bigoplus_{\substack{ v\in W_p \\ d(v\cdot\lambda)< d(wst\cdot\lambda) }} T_\chi(v\cdot\lambda)^{\oplus n_v}\oplus \bigoplus_{\substack{ u\in W_p \\ d(u\cdot\lambda)< d(ws\cdot\lambda) \\ utu^{-1}\in W_{I,p} }} T_\chi(u\cdot\lambda)^{\oplus 2m_u} \\ &  \oplus \bigoplus_{\substack{ u\in W_p \\ d(u\cdot\lambda)< d(ws\cdot\lambda) \\ utu^{-1}\notin W_{I,p},\,\, ut\cdot\lambda<u\cdot\lambda }} T_\chi(u\cdot\lambda)^{\oplus 2m_u}  \\ & \oplus \bigoplus_{\substack{ u\in W_p \\ d(u\cdot\lambda)< d(ws\cdot\lambda) \\ utu^{-1}\notin W_{I,p},\,\, u\cdot\lambda<ut\cdot\lambda }} \left( T_\chi(ut\cdot\lambda)\oplus \bigoplus_{\substack{ x\in W_p \\ d(x\cdot\lambda)< d(ut\cdot\lambda)}} T_\chi(x\cdot\lambda)^{\oplus k_x}\right) ^{\oplus m_u}
	\end{split}
\end{equation*}
for some $n_v,k_x\geq 0$. Labelling the large direct sums as (I) through (V), we note that:
\begin{itemize}
	
\item In sum (I), we have $d(v\cdot\lambda)<d(wst\cdot\lambda)$ immediately.

\item In sums (II) and (III), we have $d(u\cdot\lambda)<d(ws\cdot\lambda)$ and $d(ws\cdot\lambda)<d(wst\cdot\lambda)$, since $ws\cdot\lambda < wst\cdot\lambda$, and thus $d(u\cdot\lambda)<d(wst\cdot\lambda)$.

\item In sum (IV), we have the following: Since $u\cdot\lambda < ut\cdot\lambda$, Lemma~\ref{ReflOrd} says that $d(ut\cdot\lambda)=d(u\cdot\lambda)+1$; by construction, $d(u\cdot\lambda)<d(ws\cdot\lambda)$; and since $ws\cdot\lambda < wst\cdot\lambda$ we get $d(ws\cdot\lambda)=d(wst\cdot\lambda)-1$. Putting all this together, we get $$d(ut\cdot\lambda)=d(u\cdot\lambda)+1< d(ws\cdot\lambda)+1=d(wst\cdot\lambda).$$

\item In sum (V), the same argument as in (IV) shows that $d(ut\cdot\lambda)<d(wst\cdot\lambda)$ which, when combined with $d(x\cdot\lambda)<d(ut\cdot\lambda)$, shows that $d(x\cdot\lambda)<d(wst\cdot\lambda)$.

\end{itemize}

Combining all of this, we get 
$$\Theta_t \Theta_s(T_\chi(w\cdot\lambda))=T_\chi(wst\cdot\lambda)\oplus\bigoplus_{\substack{u\in W_p \\ d(u\cdot\lambda)<d(wst\cdot\lambda)}} T_\chi(u\cdot\lambda)^{\oplus q_u}$$ for some $q_u\geq 0$. We may iterate this process to get the following.

\begin{prop}\label{TiltSummand}
	Let $\lambda$ be as in Notation~\ref{NOT}, and let $w\in W^{I,\lambda}$. Suppose that there exist $s_1,\ldots,s_n\in S_p$ such that $$w=s_1 s_{2}\ldots s_{n-1} s_n,$$ $$s_1 s_{2}\cdots s_i\cdot \lambda \prec s_1 s_{2}\cdots s_i s_{i+1}\cdot \lambda$$ and $$s_1 s_{2} \cdots s_{i-1} s_{i} s_{i-1} \cdots s_{2} s_1 \notin W_{I,p}$$ for all $i=0, \ldots, n$. Then $$\Theta_{s_n}\Theta_{s_{n-1}}\cdots \Theta_{s_{2}}\Theta_{s_1}(T_\chi(\lambda))=T_\chi(w\cdot\lambda)\oplus \bigoplus_{\substack{u\in W_p \\ d(u\cdot\lambda)<d(w\cdot\lambda)}} T_\chi(u\cdot\lambda)^{\oplus m_u}$$ for some $m_u\geq 0$.
\end{prop}

\section{Canonical $\nabla$-flags and Sections of $\nabla$-flags}\label{s5}

In \cite{RW1}, a key technique is the theory of canonical $\nabla$-flags and sections thereof. In this section, we extend this theory to the essentially-finite fibered highest weight category $\sC_\chi$. The content of Subsection~\ref{ss5.1} should be compared with \cite[\S\S 2.2-2.3]{RW1}, Subsection~\ref{ss5.2} should be compared with \cite[\S 3.3]{RW1} and  Subsection~\ref{ss5.3} should be compared with \cite[\S 3.4]{RW1}. We see that most of the theory works the same way, but more care needs to be taken due to the distinctions between standard/costandard modules and proper standard/proper costandard modules which only arise in this setting. 

\subsection{Flags and Sections of Flags}\label{ss5.1}

 Recall that in $\sC_\chi$ the simple modules are indexed by the partially-ordered indexing set $(\Lambda_I,\preceq)=(\overline{C}_I\cap X(T),\uparrow)$; as usual, we shall switch between these two notations as is convenient. 


A {\bf canonical $\nabla$-flag} of an object $M\in \sC_\chi$ admitting a costandard filtration (cf. \cite[\S 2.2]{RW1}) is an assignment $\Gamma_\bullet(M)$ to each lower set $\Omega\subseteq \Lambda_I$ of a subobject $\Gamma_\Omega(M)\subseteq M$ satisfying the following conditions:
\begin{itemize}
	\item $\Gamma_{\Lambda_I}(M)=M$ and $\Gamma_{\emptyset}(M)=0$.
	\item $\Gamma_\Omega(M)\in  (\sC_\chi)_\Omega$ and $\Gamma_\Omega(M)$ and $M/\Gamma_\Omega(M)$ have costandard filtrations for all lower sets $\Omega\subseteq \Lambda_I$. Furthermore $(M/\Gamma_\Omega(M):\nabla(\kappa))=0$ for all $\kappa\in\Omega$.
	\item If $\Omega\subseteq\Omega'$ then $\Gamma_{\Omega}(M)\subseteq\Gamma_{\Omega'}(M)$.
	\item If $\Omega'$ is a lower set of $\Lambda_I$ and $\xi\in\Omega'$ is such that $\Omega=\Omega'\setminus\{\xi\}$ is also a lower set, then $\Gamma_{\Omega'}(M)/\Gamma_{\Omega}(M)$ is isomorphic to a direct sum of copies of $\nabla(\xi)$.
\end{itemize}

Similarly, a {\bf canonical $\Delta$-flag} of an object $M\in \sC_\chi$ admitting a standard filtration is an assignment $\Gamma^\bullet(M)$ to each lower set $\Omega\subseteq\Lambda_I$ of a submodule $\Gamma^\Omega(M)\subseteq M$ satisfying the following conditions:
\begin{itemize}
	\item $\Gamma^{\Lambda_I}(M)=0$ and $\Gamma^{\emptyset}(M)=M$.
	\item $M/\Gamma^\Omega(M)\in (\sC_\chi)_\Omega$ and $\Gamma^\Omega(M)$ and $M/\Gamma^\Omega(M)$ have standard filtrations for all lower sets $\Omega\subseteq \Lambda_I$. Furthermore $(\Gamma^\Omega(M):\Delta(\kappa))=0$ for all $\kappa\in\Omega$.
	\item If $\Omega\subseteq\Omega'$ then $\Gamma^{\Omega'}(M)\subseteq\Gamma^{\Omega}(M)$.
	\item If $\Omega'$ is a lower set of $\Lambda_I$ and $\xi\in\Omega'$ is such that $\Omega=\Omega'\setminus\{\xi\}$ is also a lower set, then $\Gamma^{\Omega}(M)/\Gamma^{\Omega'}(M)$ is isomorphic to a direct sum of copies of $\Delta(\xi)$.
\end{itemize}

\begin{prop}
	Suppose $M\in\sC_\chi$ has a standard filtration. There is a bijection between canonical $\Delta$-flags of $M$ and canonical $\nabla$-flags of $\bD M$. 
\end{prop}

\begin{proof}
	Let $\Gamma^\bullet(M)$ be a canonical $\Delta$-flag of $M$. For a lower set $\Omega\subseteq\Lambda_I$, define $$\Gamma_\Omega(\bD M)=\ker\left(\bD\left(\Gamma^\Omega(M)\hookrightarrow M\right)\right)=\ker\left(\bD M\twoheadrightarrow \bD \Gamma^\Omega(M) \right).$$ We claim that this gives a canonical $\nabla$-flag of $\bD M$.
	
	First, note that $$\Gamma_{\Lambda_I}(\bD M)=\ker\left(\bD M\twoheadrightarrow \bD M\right)=0$$ and $$\Gamma_\emptyset(\bD M)=\ker\left(\bD M \twoheadrightarrow 0\right)=\bD M.$$ Furthermore, if $\Omega\subseteq\Omega'$ then we have $\Gamma^{\Omega'} (M)\subseteq \Gamma^\Omega (M)\subseteq M$ and thus $$\Gamma_\Omega(\bD M)=\ker(\bD M\twoheadrightarrow \bD \Gamma^\Omega (M))\subseteq \ker (\bD M\twoheadrightarrow \bD \Gamma^{\Omega'} (M))=\Gamma_{\Omega'}(\bD M).$$ Suppose now that $\Omega'$ is a lower set of $\Lambda_I$, $\xi\in\Omega'$ and $\Omega\coloneqq\Omega'\setminus\{\xi\}$ is a lower set of $\Lambda_I$. Note that the short exact sequence $$0\to \Gamma^\Omega(M) \to M \to M/\Gamma^\Omega(M)\to 0$$ induces the short exact sequence $$0\to\bD(M/\Gamma^\Omega(M))\to \bD M \to \bD \Gamma^\Omega(M)\to 0,$$ which implies that $\Gamma_\Omega(\bD M)=\bD(M/\Gamma^\Omega(M))$. Since $M/\Gamma^\Omega(M)\in (\sC_\chi)_\Omega$ and $\bD(L(\eta))\cong L(\eta)$ for all $\eta\in \Lambda_I$, this shows that $\Gamma_\Omega(\bD M)\in (\sC_\chi)_\Omega$ and (since $\bD(\Delta(\eta))=\nabla(\eta)$) that it has a costandard filtration. We see further that $\bD M/\Gamma_\Omega(\bD M)\cong \bD\Gamma^\Omega(M)$ has a $\nabla$-filtration, and that $$(\bD M/\Gamma_\Omega(\bD M):\nabla(\kappa))=(\bD\Gamma^\Omega(M):\nabla(\kappa))=(\Gamma^\Omega(M):\Delta(\kappa))=0$$ for all $\kappa\in\Omega$. It also shows that $$\Gamma_{\Omega'}(\bD M)/\Gamma_\Omega(\bD M)=\frac{\bD(M/\Gamma^{\Omega'}(M))}{\bD(M/\Gamma^{\Omega}(M))}.$$ On the other hand, the short exact sequence $$0\to \frac{\Gamma^\Omega(M)}{\Gamma^{\Omega'}(M)}\to \frac{M}{\Gamma^{\Omega'}(M)}\to \frac{M}{\Gamma^\Omega(M)}\to 0$$ induces the short exact sequence $$0\to\bD\left(\frac{M}{\Gamma^\Omega(M)}\right)\to\bD\left(\frac{M}{\Gamma^{\Omega'}(M)}\right)\to \bD\left(\frac{\Gamma^\Omega(M)}{\Gamma^{\Omega'}(M)}\right)\to 0,$$ which implies that $$\Gamma_{\Omega'}(\bD M)/\Gamma_\Omega(\bD M)=\frac{\bD(M/\Gamma^{\Omega'}(M))}{\bD(M/\Gamma^{\Omega}(M))}\cong \bD\left(\frac{\Gamma^\Omega(M)}{\Gamma^{\Omega'}(M)}\right)\cong \bD\left(\bigoplus \Delta(\xi)\right)\cong \bigoplus \nabla(\xi).$$ This thus proves that $\Gamma_\bullet(\bD M)$ is a canonical $\nabla$-flag of $\bD M$. We thus have a map from the set of canonical $\Delta$-flags of $M$ to the set of canonical $\nabla$-flags of $\bD M$. Using $\overline{\bD}$, it is straightforward to construct an inverse to this map, thus proving the proposition.
\end{proof}

Since $\bD$ is a duality, this proposition tells us that any results on $\Delta$-flags can be turned into results on $\nabla$-flags, and vice versa. In this paper, our convention is to work with $\nabla$-flags.


\begin{lemma}
	Every $M\in\sC_\chi$ with a costandard filtration has a unique canonical $\nabla$-flag.
\end{lemma}

\begin{proof}
	
	Let us proceed by induction on the length $n$ of a costandard filtration of $M$. 
	
	{\bf Base case:} When $n=1$ we have $M=\nabla(\xi)$ for some $\xi\in\Lambda_I$. We then define $$\Gamma_\Omega(M)=\twopartdef{M}{\xi\in \Omega,}{0}{\xi\notin \Omega,}$$ which it is straightforward to check is a canonical $\nabla$-flag of $M$.
	
	{\bf Induction step:} Let $$0=M_0\subseteq M_1\subseteq \cdots \subseteq M_{n-1}\subseteq M_n=M$$ be a costandard filtration of $M$, and let $\kappa_i\in \Lambda_I$ be such that $M_i/M_{i-1}\cong \nabla(\kappa_i)$ for each $1\leq i\leq n$. Then $M/M_{1}\in\sC_\chi$ has a costandard filtration and by induction there exists a canonical $\nabla$-flag of $M/M_{1}$, which we denote $\Gamma_\bullet(M/M_1)$. As we have already observed (see also \cite[Lem. 3.44]{BS}), if $\xi\npreceq \kappa$ then $\Ext_{\sC_\chi}^1(\nabla(\kappa),\nabla(\xi))=0$; we may thus assume that $\kappa_i\nprec \kappa_1$ for $i\geq 1$. 
	
	Let us define a canonical $\nabla$-flag of $M$ as follows. There exists $\widehat{\Gamma}_\Omega(M)\subseteq M$ containing $M_1$ such that $\widehat{\Gamma}_\Omega(M)/M_1=\Gamma_\Omega(M)$. If $\kappa_1\in\Omega$, define $\Gamma_\Omega(M)=\widehat{\Gamma}_\Omega(M)$. If $\kappa_1\notin\Omega$ then our assumption that $\kappa_i\nprec \kappa_1$ for all $i\geq 1$ shows that we have $\Ext_{\sC_\chi}^1(\Gamma_\Omega(M/M_1),M_1)=0$, since $\Gamma_\Omega(M/M_1)\in (\sC_\chi)_\Omega$ and has a costandard filtration. Thus, the map $\widehat{\Gamma}_\Omega(M)\twoheadrightarrow \Gamma_\Omega(M/M_1)$ splits and there exists a section $\gamma_\Omega :\Gamma_\Omega(M/M_1)\to \widehat{\Gamma}_\Omega(M)$. 
	
	Define $$\Theta\coloneqq\{\zeta\in\Lambda_I\mid \kappa_1\npreceq \zeta\},$$ which is a lower set of $\Lambda_I$ not containing $\kappa_1$. Choose arbitrarily a section $\gamma:\Gamma_{\Theta}(M)\to \widehat{\Gamma}_{\Theta}(M)$. For an arbitrary lower set $\Omega\subseteq\Lambda_I$ with $\kappa_1\notin\Omega$ we note that $\Omega\subseteq\Theta$ and so $\Gamma_{\Omega}(M/M_1)\subseteq \Gamma_{\Theta}(M/M_1)$. It is straightforward to check that $\gamma$ restricts to a section $\Gamma_\Omega(M/M_1)\to \widehat{\Gamma}_\Omega(M)$. We will assume that all such sections are obtained in this way, and we set $\Gamma_\Omega(M)$ to be $\gamma(\Gamma_\Omega(M/M_1))$ (so that $\widehat{\Gamma}_\Omega(M)=\Gamma_\Omega(M)\oplus M_1$, and the projection induces an isomorphism $\Gamma_\Omega(M)\xrightarrow{\sim} \Gamma_\Omega(M/M_1)$).
	
	We show that this is a canonical $\nabla$-flag of $M$. First, note that $\Gamma_{\Lambda_I}(M)=\widehat\Gamma_{\Lambda_I}(M)=M$, since $\Gamma_{\Lambda_I}(M/M_1)=M/M_1$, and that $\Gamma_\emptyset(M)=\gamma_\emptyset(\Gamma_\emptyset(M/M_1))=0$.
	
	Second, note that if $\kappa_1\notin \Omega$ then $\Gamma_\Omega(M)\cong \Gamma_\Omega(M/M_1)\in(\sC_\chi)_\Omega$ and thus has a costandard filtration. Furthermore, we see that  $M/\Gamma_\Omega(M)$ has a costandard filtration by considering the exact sequence $$0\to M_1\to M/\Gamma_\Omega(M)\to (M/M_1)/\Gamma_\Omega(M/M_1)\to 0.$$ We also see that $$(M/\Gamma_\Omega(M):\nabla(\kappa))=(M_1:\nabla(\kappa))+((M/M_1)/\Gamma_\Omega(M/M_1):\nabla(\kappa))=0$$ for all $\kappa\in \Omega.$ Now suppose that $\kappa_1\in\Omega$. Then we have the short exact sequence $$0\to M_1\to\Gamma_\Omega(M)\to \Gamma_\Omega(M/M_1)\to 0,$$ which shows that $\Gamma_\Omega(M)\in (\sC_\chi)_\Omega$ and that it has a costandard filtration. Furthermore, $M/\Gamma_\Omega(M)\cong (M/M_1)/(\Gamma_\Omega(M)/M_1)=(M/M_1)/\Gamma_\Omega(M/M_1)$ and thus has a costandard filtration such that $(M/\Gamma_\Omega(M):\nabla(\kappa))=0$ for all $\kappa\in\Omega$.
	
	Now, suppose $\Omega\subseteq\Omega'$. We have $\Gamma_\Omega(M/M_1)\subseteq \Gamma_{\Omega'}(M/M_1)$ and so $\widehat{\Gamma}_\Omega(M)\subseteq \widehat{\Gamma}_{\Omega'}(M)$. When $\kappa_1\in\Omega$ this just says that $\Gamma_\Omega(M)\subseteq \Gamma_{\Omega'}(M)$. If $\kappa\in \Omega'\setminus\Omega$ then we have $\Gamma_\Omega(M)\subseteq \widehat{\Gamma}_\Omega(M)\subseteq \widehat{\Gamma}_{\Omega'}(M)=\Gamma_{\Omega'}(M)$. If $\kappa\notin\Omega'$, then we have $\Gamma_\Omega(M)=\gamma(\Gamma_\Omega(M/M_1))\subseteq \gamma(\Gamma_{\Omega'}(M/M_1))=\Gamma_{\Omega'}(M)$ (since we assumed all sections $\gamma$ are obtained by restriction from the section corresponding to $\Theta$).
	
	Finally, suppose $\Omega$ and $\Omega'$ are lower sets in $\Lambda_I$ such that $\Omega=\Omega'\setminus\{\xi\}$ for some $\xi\in \Lambda_I$. If $\kappa_1\in\Omega$ then $$\frac{\Gamma_{\Omega'}(M)}{\Gamma_{\Omega}(M)}=\frac{\widehat{\Gamma}_{\Omega'}(M)}{\widehat{\Gamma}_\Omega(M)}\cong\frac{\widehat{\Gamma}_{\Omega'}(M)/M_1}{\widehat{\Gamma}_\Omega(M)/M_1}\cong \frac{\Gamma_{\Omega'}(M/M_1)}{\Gamma_{\Omega}(M/M_1)}\cong \bigoplus \nabla(\xi).$$
	
	If $\kappa_1\notin\Omega'$ then $$\frac{\Gamma_{\Omega'}(M)}{\Gamma_{\Omega}(M)}=\frac{\gamma(\Gamma_{\Omega'}(M/M_1))}{\gamma(\Gamma_{\Omega}(M/M_1))}\cong \frac{\Gamma_{\Omega'}(M/M_1)}{\Gamma_{\Omega}(M/M_1)}\cong \bigoplus \nabla(\xi)$$ since $\gamma$ is an isomorphism onto its image.
	
	If $\kappa_1=\xi$ then we have a short exact sequence $$0\to M_1=\nabla(\xi)\to \frac{\Gamma_{\Omega'}(M)}{\Gamma_{\Omega}(M)}=\frac{\widehat{\Gamma}_{\Omega'}(M)}{\gamma(\Gamma_{\Omega}(M/M_1))}\to \frac{\Gamma_{\Omega'}(M/M_1)}{\Gamma_{\Omega}(M/M_1)}\cong \bigoplus \nabla(\xi)\to 0.$$ Note now that $\Ext_{\sC_\chi}^1(\nabla(\xi),\nabla(\xi))=\Ext_{\sC_\chi}^1(\Delta(\xi),\Delta(\xi))=\Ext_{\sC_\chi}^1(Q_\chi^I(\xi),Q_\chi^I(\xi))=0$ by \cite[Cor. 8.13]{West} (setting $A=\bK$ and $\pi$ to be the counit in that result). This short exact sequence thus splits, and we get that $\Gamma_{\Omega'}(M)/\Gamma_{\Omega}(M)$ is a direct sum of copies of $\nabla(\xi)$. This proves existence.
	
	{\bf Uniqueness:} Let $\Upsilon=\{\xi\in\Lambda_I\mid (M:\nabla(\xi))\neq 0\}$, and let $\Omega$ be a lower set of $\Lambda_I$. First, we claim that if $\Omega\cap\Upsilon=\emptyset$, then we must have $\Gamma_\Omega(M)=0$. Since $\Gamma_\Omega(M)\in(\sC_\chi)_\Omega$, it is enough to show that $(\Gamma_\Omega(M):\nabla(\xi))=0$ for all $\xi\in \Omega$. But this is clearly true, since $0=(M:\nabla(\xi))=(\Gamma_\Omega(M):\nabla(\xi))+(M/\Gamma_\Omega(M):\nabla(\xi))$ for all $\xi\in\Omega$. This thus proves the claim.
	
	If $\Omega\cap\Upsilon\neq\emptyset$ then there must exist $\xi\in\Omega\cap\Upsilon$ such that $\kappa\nprec\xi$ for all $\kappa\in\Upsilon$. Since $\Ext_{\sC_\chi}^1(\nabla(\xi),\nabla(\kappa))=0$ whenever $\kappa\npreceq\xi$, we may therefore assume that the filtration $$0=M_0\subseteq M_1\subseteq\cdots \subseteq M_{n-1}\subseteq M_n=M$$ has $M_1=\nabla(\xi)$.
	
	Let us prove uniqueness by induction on the length of a costandard filtration of $M$. It is straightforward to see that the canonical $\nabla$-flag of $\nabla(\xi)$ is unique. Let $\Omega\subseteq \Lambda_I$ be a lower set; we will show that $\Gamma_\Omega(M)$ is unique. As above, we may assume that the filtration is such that $M_1=\nabla(\xi)$ with $\xi\in\Omega$ such that $\kappa\nprec\xi$ for all $\kappa\in\Upsilon$.
	
	By induction, we may assume that the canonical $\nabla$-flag of $M/M_1$ is unique, and thus uniquely define $\Gamma_\Omega(M/M_1)$. In particular, there exists a unique $\widehat{\Gamma}_\Omega(M)\subseteq M$ such that $\widehat{\Gamma}_\Omega(M)/M_1=\Gamma_\Omega(M/M_1)$. Let $\Gamma_\Omega(M)$ be part of a canonical $\nabla$-flag of $M$. 
	
	First, we claim that $\Gamma_\Omega(M)\subseteq \widehat{\Gamma}_\Omega(M)$. For this, it is sufficient to show that $\Hom_{\sC_\chi}(\Gamma_\Omega(M),M/\widehat{\Gamma}_\Omega(M))=0$. Note that $\Gamma_\Omega(M)$ has a $\nabla$-filtration in which all the $\nabla$ are indexed by elements of $\Omega$, and that $M/\widehat{\Gamma}_\Omega(M)\cong (M/M_1)/(\widehat{\Gamma}_\Omega(M)/M_1)=M/\Gamma_\Omega(M/M_1))$ has a $\nabla$-filtration in which all the $\nabla$ are indexed by elements of $\Lambda_I\setminus\Omega$. Recalling that in $\sC_\chi$ each $\nabla(\zeta)$ has a $\overline{\nabla}$-filtration whose successive quotients are all $\overline{\nabla}(\zeta)$, the question reduces to showing that $\Hom_{\sC_\chi}(\nabla(\omega),\overline{\nabla}(\zeta))=0$ whenever $\omega\in\Omega$ and $\zeta\in \Lambda_I\setminus \Omega$. If $0\neq g\in\Hom_{\sC_\chi}(\nabla(\omega),\overline{\nabla}(\zeta))$ then we must have that $L(\zeta)\subseteq g(\nabla(\omega))$ (as $L(\zeta)$ is the irreducible socle of $\overline{\nabla}(\zeta)$), and thus that $[\nabla(\omega):L(\zeta)]\neq 0$. By \cite[Prop. 4.5]{Jan4}, this implies that $\zeta\preceq \omega$. Since $\omega\in\Omega$ and $\Omega$ is a lower set, this implies that $\zeta\in\Omega$, a contradiction. This thus shows that $\Gamma_\Omega(M)\subseteq \widehat{\Gamma}_\Omega(M)$.
	
	We will now show that $\Gamma_\Omega(M)= \widehat{\Gamma}_\Omega(M)$. To show this, it is enough to see that $(\Gamma_\Omega(M):\nabla(\kappa))=(\widehat{\Gamma}_\Omega(M):\nabla(\kappa))$ for all $\kappa\in \Omega$. For this we observe that if $\kappa\neq \xi$ then we have $(\Gamma_\Omega(M):\nabla(\kappa))=(M:\nabla(\kappa))=(M/M_1:\nabla(\kappa))=(\widehat{\Gamma}_\Omega(M):\nabla(\kappa))$, and we also have $$(\Gamma_\Omega(M):\nabla(\xi))=(M:\nabla(\xi))=1+(M/M_1:\nabla(\xi))=(\widehat{\Gamma}_\Omega(M):\nabla(\xi)).$$ 
	
	This thus proves the uniqueness of the $\nabla$-flag.
\end{proof}

\begin{cor}
	Every $M\in\sC_\chi$ with a standard filtration has a unique canonical $\Delta$-flag.
\end{cor}

\begin{dfn}
	Let $M\in\sC_\chi$ have a costandard filtration. A {\bf $\overline{\Delta}$-section of the $\nabla$-flag} of $M$ consists of 
	\begin{itemize}
		\item A finite set $\Pi$.
		\item A map $e:\Pi\to\Lambda_I$.
		\item For each $\pi\in\Pi$, a morphism $\phi_\pi^M:T_\chi(e(\pi))\to M$ such that, for each $\xi\in\Lambda_I$, the set $$\{\overline{\Delta}(\xi)\xhookrightarrow{\overline{\iota}_\xi} T_\chi(\xi)\xrightarrow{\phi_\pi^M} M\mid e(\pi)=\xi\}$$ forms a basis of $\Hom_{\sC_\chi}(\overline{\Delta}(\xi),M)$.
	\end{itemize}
\end{dfn}

Using the inclusions $\iota_\xi$ and the projections $\overline{\pi}_\xi$ and $\pi_\xi$, respectively, we may similarly define a {\bf $\Delta$-section of the $\nabla$-flag} of $M$ (for $M\in\sC_\chi$ admitting a costandard filtration), and a {\bf $\overline{\nabla}$-section of the $\Delta$-flag} of $M$ and a {\bf $\nabla$-section of the $\Delta$-flag} of $M$ (for $M\in\sC_\chi$ admitting a standard filtration).

The following results are proved in more-or-less the same way as Lemmas 2.9 and 2.10 in \cite{RW1}, so we omit the proofs.

\begin{lemma}\label{SubSect}
	Let $M\in\sC_\chi$ have a costandard filtration, and let $(\Pi,e,(\phi_\pi^M)_{\pi\in\Pi})$ be a $\overline{\Delta}$-section of the $\nabla$-flag of $M$. Let $\Omega$ be a lower set of $\Lambda_I$. For each $\pi\in\Pi$ such that $e(\pi)\in\Omega$ there exists a unique map $\phi_\pi^{\Gamma_\Omega(M)}:T_\chi(e(\pi))\to \Gamma_\Omega(M)$ such that the following diagram commutes:
	\begin{eqnarray*}
		\begin{array}{c}\xymatrix{
				T_\chi(e(\pi))\ar@{->}[rr]^{\phi_\pi^M} \ar@{->}[drr]^{\phi_\pi^{\Gamma_\Omega(M)}} & & M  \ar@{<-^{)}}[d] \\ & & \Gamma_\Omega(M).
		}\end{array}
	\end{eqnarray*}
	
	For each lower set $\Omega\subseteq\Lambda_I$ there exists a $\overline{\Delta}$-section of the $\nabla$-flag of $\Gamma_\Omega(M)$ defined as follows:
	\begin{itemize}
		\item $\Pi_\Omega=\{\pi\in\Pi\mid e(\pi)\in\Omega\}$.
		\item $e_\Omega=e\vert_{\Pi_\Omega}:\Pi_\Omega\to \Omega$.
		\item For each $\pi\in\Pi_\Omega$, the map $\phi_\pi^{\Gamma_\Omega(M)}:T_\chi(e(\pi))\to \Gamma_\Omega(M)$ is as described above.
	\end{itemize}
\end{lemma}

\begin{lemma}\label{QuotSect}
	Let $M\in\sC_\chi$ have a costandard filtration, and let $(\Pi,e,(\phi_\pi^M)_{\pi\in\Pi})$ be a $\overline{\Delta}$-section of the $\nabla$-flag of $M$. For each lower set $\Omega\subseteq\Lambda_I$ there exists a $\overline{\Delta}$-section of the $\nabla$-flag of $M/\Gamma_\Omega(M)$ defined as follows:
	\begin{itemize}
		\item $\Pi^\Omega=\{\pi\in\Pi\mid e(\pi)\notin\Omega\}$.
		\item $e_\Omega=e\vert_{\Pi^\Omega}:\Pi_\Omega\to \Lambda_I\setminus\Omega$.
		\item For each $\pi\in\Pi_\Omega$, the map $\phi_\pi^{M/\Gamma_\Omega(M)}:T_\chi(e(\pi))\to M/\Gamma_\Omega(M)$ is the composition $$T_\chi(e(\pi))\xrightarrow{\phi_\pi^M} M\twoheadrightarrow M/\Gamma_\Omega(M).$$
	\end{itemize}
\end{lemma}

It is straightforward to check that if $(\Pi,e,(\phi_\pi^M)_{\pi\in\Pi})$ is a $\overline{\Delta}$-section of the $\nabla$-flag of some $M\in\sC_\chi$ with a costandard filtration, then $(\Pi,e,(\bD(\phi_\pi^M))_{\pi\in\Pi})$ is a $\overline{\nabla}$-section of the $\Delta$-flag of $\bD M$. The above results can therefore easily be modified to give the analogous results for $\overline{\nabla}$-sections of $\Delta$-flags; details are left to the reader.

Let $M\in\sC_\chi$ have a costandard filtration, and suppose that $(\Pi,e,(\phi_\pi^M)_{\pi\in\Pi})$ is a $\Delta$-section of the $\nabla$-flag of $M$. For each $\xi\in\Lambda_I$ we then have that $(\phi_\pi^M)_{\pi\in e^{-1}(\xi)}\subseteq \Hom_{\sC_\chi}(T_\chi(\xi),M)$ are such that the maps $(\phi_\pi^M\circ \iota_\xi)_{\pi\in e^{-1}(\xi)}$ form a basis of $\Hom_{\sC_\chi}(\Delta(\xi),M)$. There is an exact sequence $$\Hom_{\sC_\chi}(\Delta(\xi),M)\to \Hom_{\sC_\chi}(\overline{\Delta}(\xi),M)\to \Ext_{\sC_\chi}^1(\Delta(\xi)/\widehat{\iota}_\xi(\overline{\Delta}(\xi)),M)$$ and $\Ext_{\sC_\chi}^1(\Delta(\xi)/\widehat{\iota}_\xi(\overline{\Delta}(\xi)),M)=0$ since $\Ext_{\sC_\chi}^1(\overline{\Delta}(\gamma),\nabla(\kappa))=0$ for all $\gamma,\kappa\in\Lambda_I$. Therefore the map $$\Hom_{\sC_\chi}(\Delta(\xi),M)\to \Hom_{\sC_\chi}(\overline{\Delta}(\xi),M),\qquad f\mapsto f\circ\widehat{\iota}_\xi$$ is surjective and the maps $(\phi_\pi^M\circ \overline{\iota}_\xi)\subseteq \Hom_{\sC_\chi}(\overline{\Delta}(\xi),M)$ span $\Hom_{\sC_\chi}(\overline{\Delta}(\xi),M)$. In particular, we can find a set $\Pi_\xi\subseteq e^{-1}(\xi)$ such that the maps $(\phi_\pi^M\circ\overline{\iota}_\xi)_{\pi\in\Pi_\xi}$ form a basis of $\Hom_{\sC_\chi}(\overline{\Delta}(\xi),M)$. Therefore, setting $\overline{\Pi}=\bigcup_{\xi\in\Lambda_I}\Pi_\xi$ and $\overline{e}=e\vert_{\overline{\Pi}}$, we get the following lemma.

\begin{lemma}\label{Stand2PropStand}
	Let $M\in\sC_\chi$ have a costandard filtration and suppose that $(\Pi,e,(\phi_\pi^M)_{\pi\in\Pi})$ is a $\Delta$-section of the $\nabla$-flag of $M$. Then, in the notation above, $(\overline{\Pi},\overline{e},(\phi_\pi^M)_{\pi\in\overline{\Pi}})$ is a $\overline{\Delta}$-section of the $\nabla$-flag of $M$.
\end{lemma}

On the other hand, let $M\in\sC_\chi$ have a costandard filtration and suppose that $(\overline{\Pi},\overline{e},(\phi_\pi^M)_{\pi\in\overline{\Pi}})$ is a $\overline{\Delta}$-section of the $\nabla$-flag of $M$. Let $\xi\in\Lambda_I$; by definition the $(\phi_\pi^M\circ\overline{\iota}_\xi)_{\pi\in \overline{e}^{-1}(\xi)}$ form a basis of $\Hom_{\sC_\chi}(\overline{\Delta}(\xi),M)$. Furthermore, by Corollary~\ref{TiltBasCor} there exists a collection of maps $\sigma_1^\xi,\ldots,\sigma_{N_\xi}^\xi\in\Hom_{\sC_\chi}(T_\chi(\xi),T_\chi(\xi))$ such that $\pi_\xi\circ\sigma_1^\xi,\ldots,\pi_\xi\circ\sigma_{N_\xi}^\xi$ form a basis of $\Hom_{\sC_\chi}(T_\chi(\xi),\nabla(\xi))$.

Consider now the diagram 

\begin{eqnarray*}
	\begin{array}{c}\xymatrix{
			& & \Delta(\xi) \ar@{->}[rr] \ar@{^{(}->}[d]^{\iota_\xi} & & \nabla(\xi) \ar@{<<-}[d]^{\pi_\xi} & & \\ 
			\Delta(\xi) \ar@{->}[rr]^{\iota_\xi} \ar@{->}[drr] & & T_\chi(\xi) \ar@{->}[rr]^{\sigma_i^{\xi}} \ar@{->>}[d]^{\overline{\pi}_\xi} & & T_\chi(\xi) \ar@{->}[rr]^{\phi_\pi^M} \ar@{<-^{)}}[d]^{\overline{\iota}_\xi} & & M \ar@{<-}[dll] \\
			& & \overline{\nabla}(\xi) & & \overline{\Delta}(\xi) & & 
			.}\end{array}
\end{eqnarray*}

It is straightforward to check that the maps $\pi_\xi\circ \sigma_1^\xi\circ \iota_\xi,\ldots,\pi_\xi\circ \sigma_{N_\xi}^\xi\circ \iota_\xi$ form a basis of $\Hom_{\sC_\chi}(\Delta(\xi),\nabla(\xi))$. Recalling that $\Hom_{\sC_\chi}(\Delta(\xi),\overline{\nabla}(\kappa))=\Hom_{\sC_\chi}(\Delta(\xi),\nabla(\kappa))=0$ whenever $\kappa\neq\xi$, we therefore obtain by \cite[Thm 4.45]{BS} that the maps $\phi_\pi^M\circ \sigma_i^\xi\circ\iota_\xi$, for $1\leq i\leq N_\xi$ and $\pi\in \overline{e}^{-1}(\xi)$, form a basis of $\Hom_{\sC_\chi}(\Delta(\xi),M)$. In particular, let us set $\Pi^\xi=\overline{e}^{-1}(\xi)\times\{1,\ldots,N_\xi\}$, $\Pi=\bigcup_{\xi\in\Lambda_I}\Pi^\xi$, and set $e(\pi,i)=\overline{e}(\pi)$ for all $\pi\in\Pi^{\overline{e}(\pi)}$ and $i\in\{1,\ldots, N_{\overline{e}(\pi)}\}$. We then get the following result
\begin{lemma}\label{PropStand2Stand}
	Let $M\in\sC_\chi$ have a costandard filtration and suppose that $(\overline{\Pi},\overline{e},(\phi_\pi^M)_{\pi\in\overline{\Pi}})$ is a $\overline{\Delta}$-section of the $\nabla$-flag of $M$. Using the above notation, $(\Pi,e,(\phi_\pi^M\circ \sigma_i^{\overline{e}(\pi)})_{(\pi,i)\in\Pi})$ is a $\Delta$-section of the $\nabla$-flag of $M$.
\end{lemma}

\subsection{Sections of $\nabla$-Flags: Translation onto a Wall}\label{ss5.2}
For this subsection, let $s,\lambda$ and $\mu$ be as in Notation~\ref{NOT}. Let $M\in\sC_\chi(\lambda)$ have a costandard filtration and let $(\Pi,e,(\phi_\pi^M)_{\pi\in\Pi})$ be a $\overline{\Delta}$-section of the $\nabla$-flag of $M$. In this subsection, we will construct a $\overline{\Delta}$-section of the $\nabla$-flag of $T_\lambda^\mu(M)$.

Define
$$\Pi_1\coloneqq\{\pi\in\Pi\mid e(\pi)=w\cdot\lambda\mbox{ for }w\in W^{I,\lambda}\mbox{ such that } wsw^{-1}\notin W_{I,p}\}$$ and $$\Pi_2\coloneqq\{\pi\in\Pi\mid e(\pi)=w\cdot\lambda\mbox{ for }w\in W^{I,\lambda}\mbox{ such that } wsw^{-1}\in W_{I,p}\}.$$ We then set $$\Pi'\coloneqq\Pi_1\cup (\Pi_2\times\{1,2\})$$ and define $e':\Pi'\to \Lambda_I$ as follows. If $\pi\in \Pi_1$ then we define $$e'(\pi)=w\cdot\mu.$$ Alternatively, if $\pi\in \Pi_2$ then we define $$e'(\pi,1)=e'(\pi,2)=w\cdot\mu.$$

For each $w\in W^{I,\mu}$, we define a family $(\phi_\pi^{T_\lambda^\mu(M)})_{\pi\in (e')^{-1}(w\cdot\mu)}$ of maps $T_\chi(w\cdot\mu)\to T_\lambda^\mu(M)$. We will distinguish between two cases.

{\bf Case 1:} $wsw^{-1}\in W_{I,p}$.

In this case, we need to define a map $T_\chi(w\cdot\mu)\to T_\lambda^\mu(M)$ for each $(\pi,i)\in \Pi_2\times\{1,2\}$ with $e(\pi)=w\cdot\lambda$. We define such maps $\phi_{(\pi,i)}^{T_\lambda^\mu(M)}:T_\chi(w\cdot\mu)\to T_\lambda^\mu(M)$ via $$\phi_{(\pi,i)}^{T_\lambda^\mu(M)}:T_\chi(w\cdot\mu)\xhookrightarrow{\iota_i} T_\chi(w\cdot\mu)\oplus T_\chi(w\cdot\mu)\cong T_\lambda^\mu(T_\chi(w\cdot\lambda))\xrightarrow{T_\lambda^\mu(\phi_{\pi}^M)} T_\lambda^\mu(M) $$ for $i=1,2$ and $\pi\in \Pi$ such that $e(\pi)=w\cdot\lambda$. Here $\iota_i$ is the natural inclusion of $T_\chi(w\cdot\mu)$ into the $i$th direct summand of $T_\chi(w\cdot\mu)\oplus T_\chi(w\cdot\mu)$. 

{\bf Case 2:} $wsw^{-1}\notin W_{I,p}$.

We need to define a map $T_\chi(w\cdot\mu)\to T_\lambda^\mu(M)$ for each $\pi\in\Pi_1$ with $e(\pi)=w\cdot\lambda$ or $e(\pi)=ws\cdot\lambda$. Recalling that $ws\cdot\lambda < w\cdot\lambda$ (from the definition of $W^{I,\mu}$) we make the following definitions:

When $\pi\in \Pi$ is such that $e(\pi)=w\cdot\lambda$, we define  $\phi_\pi^{T_\lambda^\mu(M)}:T_\chi(w\cdot\mu)\to T_\lambda^\mu(M)$ to be the image of the map $\phi_\pi^M$ under the isomorphism $$\Hom_{\sC_\chi}(T_\chi(w\cdot\lambda),M)\cong \Hom_{\sC_\chi}(T_\mu^\lambda(T_\chi(w\cdot\mu)),M)\cong \Hom_{\sC_\chi}(T_\chi(w\cdot\mu),T_\lambda^\mu(M)).$$ 

On the other hand, when $\pi\in \Pi$ is such that $e(\pi)=ws\cdot\lambda$ we define  $\phi_\pi^{T_\lambda^\mu(M)}:T_\chi(w\cdot\mu)\to T_\lambda^\mu(M)$ to be the composition $$\phi_\pi^{T_\lambda^\mu(M)}:T_\chi(w\cdot\mu)\hookrightarrow T_\chi(w\cdot\mu)\oplus T_{\prec w\cdot\mu}\cong T_\lambda^\mu(T_\chi(ws\cdot\lambda))\xrightarrow{T_\lambda^\mu(\phi_\pi^M)} T_\lambda^\mu(M).$$ Here, $T_{\prec w\cdot\mu}$ is the module with the same name from Corollary~\ref{TransTilt2}.

We show in Theorem~\ref{LMSect} that this construction gives a $\overline{\Delta}$-section of the $\nabla$-flag of $T_\lambda^\mu(M)$. To prove this, we first tackle the case where $M$ is a costandard module.

\begin{prop}\label{LMNabla}
	Let $M=\nabla(u\cdot\lambda)$ for some $u\in W^{I,\lambda}$. Then the collection $(\Pi',e',(\phi_\pi^{T_\lambda^\mu(M)})_{\pi\in \Pi'})$ is a $\overline{\Delta}$-section of the $\nabla$-flag of $T_\lambda^\mu(M)$.
\end{prop}

\begin{proof}
	By Equality~(\ref{dimHom}), a $\overline{\Delta}$-section of the $\nabla$-flag of $M={\nabla(u\cdot\lambda)}$ consists of $\Pi=\{\ast\}$, $e:\{\ast\}\to \Lambda_I$ such that $e(\ast)=u\cdot\lambda$, and a map $\phi_{\ast}^{\nabla(u\cdot\lambda)}:T_\chi(u\cdot\lambda)\to \nabla(u\cdot\lambda)$ such that $\phi_\ast^{\nabla(u\cdot\lambda)}\circ\overline{\iota}_{u\cdot\lambda}\neq 0$.
	
	Let $w\in W^{I,\mu}$. We have \begin{equation*}
		\begin{split}
			\dim\Hom_{\sC_\chi}(\overline{\Delta}(w\cdot\mu),T_\lambda^\mu(\nabla(u\cdot\lambda))) & =\dim\Hom_{\sC_\chi}(T_\mu^\lambda(\overline{\Delta}(w\cdot\mu)),\nabla(u\cdot\lambda)) \\ & =\dim\Hom_{\sC_\chi}(\overline{\Delta}(w\cdot\lambda),\nabla(u\cdot\lambda)) + \dim\Hom_{\sC_\chi}(\overline{\Delta}(ws\cdot\lambda),\nabla(u\cdot\lambda))	\\ & 
			=(\nabla(u\cdot\lambda):\nabla(w\cdot\lambda))+(\nabla(u\cdot\lambda):\nabla(ws\cdot\lambda)) \\ & 
			=\threepartdef{0}{w\cdot\lambda\neq u\cdot\lambda\neq ws\cdot\lambda,}{1}{w\cdot\lambda=u\cdot\lambda\mbox{ or }ws\cdot\lambda=u\cdot\lambda,\mbox{ and }usu^{-1}=wsw^{-1}\notin W_{I,p},}{2}{w\cdot\lambda=u\cdot\lambda\mbox{ and } usu^{-1}=wsw^{-1}\in W_{I,p}.}
		\end{split}
	\end{equation*}
The second equality follows from Corollary~\ref{TransZ} and Equality~(\ref{dimHom}).

Note that if $usu^{-1}\notin W_{I,p}$ then  $$\left\vert\Pi_1\right\vert=\left\vert\{\ast\}\right\vert=1$$ and  $$\left\vert\Pi_2\right\vert=\left\vert\emptyset\right\vert=0$$ and thus $\Pi'=\{\ast\}$ with $e'(\ast)=u\cdot\mu$. On the other hand, if $usu^{-1}\in W_{I,p}$ then $$\left\vert\Pi_1\right\vert=\left\vert\emptyset\right\vert=0$$ and  $$\left\vert\Pi_2\right\vert=\left\vert\{\ast\}\right\vert=1$$ and thus $\Pi'=\{\ast\}\times\{1,2\}$ with $e'(\ast,1)=e'(\ast,2)=u\cdot\mu$.

Let us first deal with the case $usu^{-1}\notin W_{I,p}$. We need to treat differently the cases where $u\cdot\lambda<us\cdot\lambda$ and $us\cdot\lambda<u\cdot\lambda$. In the former, we need to show that the map $$\phi_{\ast}^{T_\lambda^\mu({\nabla(u\cdot\lambda)})}:T_\chi(u\cdot\mu)\hookrightarrow T_\chi(u\cdot\mu)\oplus T_{\prec u\cdot\mu}\cong T_\lambda^\mu(T_\chi(u\cdot\lambda))\xrightarrow{T_\lambda^\mu(\phi_\ast^{\nabla(u\cdot\lambda)})} T_\lambda^\mu({\nabla(u\cdot\lambda)})$$ is such that $\phi_{\ast}^{T_\lambda^\mu({\nabla(u\cdot\lambda)})}\circ\overline{\iota}_{u\cdot\mu}$ is a basis of $\Hom_{\sC_\chi}(\overline{\Delta}(u\cdot\mu),T_\lambda^\mu({\nabla(u\cdot\lambda)}))$, i.e. that $\phi_{\ast}^{T_\lambda^\mu({\nabla(u\cdot\lambda)})}\circ\overline{\iota}_{u\cdot\mu}\neq 0$. Since $\overline{\Delta}(u\cdot\lambda)\xhookrightarrow{\overline{\iota}_{u\cdot\lambda}} T_\chi(u\cdot\lambda) \xrightarrow{\phi_\ast^{\nabla(u\cdot\lambda)}} \nabla(u\cdot\lambda)$ is non-zero (by definition it forms a basis of $\Hom_{\sC_\chi}(\overline{\Delta}(u\cdot\lambda),\nabla(u\cdot\lambda))$) we have by Proposition~\ref{NZeroSurj} that $\phi_\ast^{\nabla(u\cdot\lambda)}$ is surjective and thus $T_\lambda^\mu(\phi_\ast^{\nabla(u\cdot\lambda)})$ is surjective. In particular, $[\ker(T_\lambda^\mu(\phi_\ast^{\nabla(u\cdot\lambda)})):L_\chi(u\cdot\mu)]=[T_\lambda^\mu(T_\chi(u\cdot\lambda)):L_\chi(u\cdot\mu)]-[T_\lambda^\mu(\nabla(u\cdot\lambda)):L_\chi(u\cdot\mu)]=0$. This suffices to show that $\phi_{\ast}^{T_\lambda^\mu({\nabla(u\cdot\lambda)})}\circ\overline{\iota}_{u\cdot\mu}\neq 0$.

When $us\cdot\lambda<u\cdot\lambda$, we need to show that the image $\phi_\ast^{T_\lambda^\mu({\nabla(u\cdot\lambda)})}$ of the map $\phi_\ast^{\nabla(u\cdot\lambda)}$ under the isomorphism $$\Hom_{\sC_\chi}(T_\chi(u\cdot\lambda),{\nabla(u\cdot\lambda)})\cong \Hom_{\sC_\chi}(T_\mu^\lambda(T_\chi(u\cdot\mu)),{\nabla(u\cdot\lambda)})\cong \Hom_{\sC_\chi}(T_\chi(u\cdot\mu),T_\lambda^\mu({\nabla(u\cdot\lambda)}))$$ is such that $\phi_\ast^{T_\lambda^\mu({\nabla(u\cdot\lambda)})}\circ \overline{\iota}_{u\cdot\lambda}\neq 0$. By considering the commutative diagram 

 	\begin{eqnarray*}
	\begin{array}{c}\xymatrix{
			\Hom_{\sC_\chi}(T_\chi(u\cdot\lambda),{\nabla(u\cdot\lambda)}) \ar@{->}[d]^{\simeq} & & \\
			\Hom_{\sC_\chi}(T_\mu^\lambda(T_\chi(u\cdot\mu)),{\nabla(u\cdot\lambda)}) \ar@{->}[d]^{\simeq} \ar@{->}[rr]^{\circ T_\mu^\lambda(\overline{\iota}_{u\cdot\mu})} & & \Hom_{\sC_\chi}(T_\mu^\lambda(\overline{\Delta}(u\cdot\mu)),{\nabla(u\cdot\lambda)}) \ar@{->}[d]^{\simeq} \\
			\Hom_{\sC_\chi}(T_\chi(u\cdot\mu),T_\lambda^\mu({\nabla(u\cdot\lambda)})) \ar@{->}[rr]^{\circ \overline{\iota}_{u\cdot\mu}} & & \Hom_{\sC_\chi}(\overline{\Delta}(u\cdot\mu),T_\lambda^\mu({\nabla(u\cdot\lambda)}))			
			}\end{array} 
	\end{eqnarray*}

it is enough to check that the composition $$T_\mu^\lambda(\overline{\Delta}(u\cdot\mu))\hookrightarrow T_\mu^\lambda(T_\chi(u\cdot\mu))\xrightarrow{\sim} T_\chi(u\cdot\lambda)\xrightarrow{\phi_\ast^{\nabla(u\cdot\lambda)}} {\nabla(u\cdot\lambda)}$$ is non-zero. As above, $\phi_\ast^{\nabla(u\cdot\lambda)}$ is surjective and thus $[\ker(\phi_\ast^{\nabla(u\cdot\lambda)}):L_\chi(u\cdot\lambda)]=0$. Since $[T_\mu^\lambda(\overline{\Delta}(u\cdot\mu)):L_\chi(u\cdot\lambda)]=1$, the result follows.

Finally, we need to deal with the case where $usu^{-1}\in W_{I,p}$. In this case, we need to show that the two maps $$\phi_{\ast,i}^{T_\lambda^\mu({\nabla(u\cdot\lambda)})}:T_\chi(u\cdot\mu)\xhookrightarrow{\iota_i} T_\chi(u\cdot\mu)\oplus T_\chi(u\cdot\mu)\cong T_\lambda^\mu(T_\chi(u\cdot\lambda))\xrightarrow{T_\lambda^\mu(\phi_\ast^{\nabla(u\cdot\lambda)})} T_\lambda^\mu({\nabla(u\cdot\lambda)})$$ are such that $\phi_{\ast,1}^{T_\lambda^\mu({\nabla(u\cdot\lambda)})}\circ\overline{\iota}_{u\cdot\mu}$ and $ \phi_{\ast,2}^{T_\lambda^\mu({\nabla(u\cdot\lambda)})}\circ \overline{\iota}_{u\cdot\mu}$ form a basis of $\Hom_{\sC_\chi}(\overline{\Delta}(u\cdot\mu),T_\lambda^\mu({\nabla(u\cdot\lambda)}))$. Since $\dim\Hom_{\sC_\chi}(\overline{\Delta}(u\cdot\mu),T_\lambda^\mu({\nabla(u\cdot\lambda)}))=2$ it is enough to show that $\phi_{\ast,1}^{T_\lambda^\mu({\nabla(u\cdot\lambda)})}\circ\overline{\iota}_{u\cdot\mu}$ and $\phi_{\ast,2}^{T_\lambda^\mu({\nabla(u\cdot\lambda)})}\circ \overline{\iota}_{u\cdot\mu}$ are linearly independent. Suppose $a,b\in\bK$ are such that $$a\left(\phi_{\ast,1}^{T_\lambda^\mu({\nabla(u\cdot\lambda)})}\circ\overline{\iota}_{u\cdot\mu}\right)+b \left(\phi_{\ast,2}^{T_\lambda^\mu({\nabla(u\cdot\lambda)})}\circ \overline{\iota}_{u\cdot\mu}\right)=0.$$ By linearity, this means that the following composition is zero:
$$\overline{\Delta}(u\cdot\mu)\hookrightarrow T_\chi(u\cdot\mu)\xrightarrow{a\iota_1+b\iota_2} T_\chi(u\cdot\mu)\oplus T_\chi(u\cdot\mu)\cong T_\lambda^\mu(T_\chi(u\cdot\lambda))\xrightarrow{T_\lambda^\mu(\phi_\ast^{\nabla(u\cdot\lambda)})} T_\lambda^\mu({\nabla(u\cdot\lambda)}).$$

If $a$ and $b$ are not both zero then it is clear that $a\iota_1+b\iota_2$ is injective. Furthermore, as above we know that $\phi_\ast^{\nabla(u\cdot\lambda)}$ is surjective, and hence $T_\lambda^\mu(\phi_\ast^{\nabla(u\cdot\lambda)})$ is surjective. We also know that $T_\lambda^\mu(\nabla(u\cdot\lambda))$ has a filtration of two copies of $\nabla(u\cdot\mu)$, and so $[T_\lambda^\mu({\nabla(u\cdot\lambda)}):L_\chi(u\cdot\mu)]=2[\nabla(u\cdot\mu):L_\chi(u\cdot\mu)]=2[T_\chi(u\cdot\mu):L_\chi(u\cdot\mu)]=[T_\lambda^\mu(T_\chi(u\cdot\mu)):L_\chi(u\cdot\mu)]$. We thus have $[\ker(T_\lambda^\mu(\phi_\ast^{\nabla(u\cdot\lambda)})):L_\chi(u\cdot\mu)]=0$. Putting all this together, if $a$ and $b$ are not both zero then the above composition cannot be zero. Hence, $\phi_{\ast,1}^{T_\lambda^\mu({\nabla(u\cdot\lambda)})}\circ\overline{\iota}_{u\cdot\mu}$ and $\phi_{\ast,2}^{T_\lambda^\mu({\nabla(u\cdot\lambda)})}\circ \overline{\iota}_{u\cdot\mu}$ are linearly independent as required.
\end{proof}

We now tackle the case where $M$ is a direct sum of copies of a single costandard module.

\begin{prop}\label{SumSection1}
	When $M=\bigoplus_{i=1}^n\nabla(u\cdot\lambda)$ for some $u\in W^{I,\lambda}$ and $n\geq 1$, the collection $(\Pi',e',(\phi_\pi^{T_\lambda^\mu(M)})_{\pi\in \Pi'})$ is a $\overline{\Delta}$-section of the $\nabla$-flag of $T_\lambda^\mu(M)$.
\end{prop}

\begin{proof}
	Arguing as in the proof of Proposition~\ref{LMNabla} we see that $(\Pi,e,(\phi_\pi^M)_{\pi\in \Pi})$ consists of $\Pi=\{1,\ldots,n\}$, $e(i)=u\cdot\lambda$ for all $1\leq i\leq n$, and maps $\phi_1^M,\ldots,\phi_n^M:T_\chi(u\cdot\lambda)\to M$ such that $\phi_1^M\circ\overline{\iota}_{u\cdot\lambda},\ldots,\phi_n^M\circ\overline{\iota}_{u\cdot\lambda}$ form a basis of $\Hom_{\sC_\chi}(\overline{\Delta}(u\cdot\lambda),M)$. We claim that $$M\cong\bigoplus_{i=1}^n\im(\phi_i^M)$$ and that $$\im(\phi_i^M)\cong\nabla(u\cdot\lambda)$$ for each $1\leq i\leq n$. The result will then follow easily from Proposition~\ref{LMNabla}, using that $(\{i\},e\vert_{\{i\}},\phi_i^M)$ is a $\overline{\Delta}$-section of the $\nabla$-flag of $\im(\phi_i^M)\cong\nabla(u\cdot\lambda)$ for each $1\leq i\leq n$.
	
	We begin by showing that $\im(\phi_i^M)\cong\nabla(u\cdot\lambda)$ for each $1\leq i\leq n$. Let us write $p_j:M=\bigoplus_{k=1}^n\nabla(u\cdot\lambda)\to\nabla(u\cdot\lambda)$ for projection to the $j$-th coordinate. Since  $\phi_1^M\circ\overline{\iota}_{u\cdot\lambda},\ldots,\phi_n^M\circ\overline{\iota}_{u\cdot\lambda}$ form a basis of $\Hom_{\sC_\chi}(\overline{\Delta}(u\cdot\lambda),M)$ there must exist $1\leq j\leq n$ such that $$p_j\circ \phi_i^M\circ\overline{\iota}_{u\cdot\lambda}\neq 0.$$ By Proposition~\ref{NZeroSurj}, $p_j\circ \phi_i^M:T_\chi(u\cdot\lambda)\to \nabla(u\cdot\lambda)$ must be surjective, and thus $q\coloneqq p_j\vert_{\tiny \im(\phi_i^M)}:\im(\phi_i^M)\to\nabla(u\cdot\lambda)$ must also be surjective. We claim that $\ker(q)=0$. If not, then there exists $1\leq k\leq n$ such that $p_k(\ker(q))\neq 0$, so that $p_k$ induces a non-zero map $q_k=p_k\vert_{\ker(q)}:\ker(q)\to\nabla(u\cdot\lambda)$. 
	We have inequalities $$[T_\chi(u\cdot\lambda):L_\chi(u\cdot\lambda)]\geq [\im(\phi_i^M):L_\chi(u\cdot\lambda)]\geq [\nabla(u\cdot\lambda): L_\chi(u\cdot\lambda)]=[T_\chi(u\cdot\lambda):L_\chi(u\cdot\lambda)]$$ and thus $[\im(\phi_i^M):L_\chi(u\cdot\lambda)]= [\nabla(u\cdot\lambda): L_\chi(u\cdot\lambda)]$. This implies $[\ker(q):L_\chi(u\cdot\lambda)]=0$. Since $\overline{\nabla}(u\cdot\lambda)$ has an irreducible socle isomorphic to $L_\chi(u\cdot\lambda)$, this means that there are no non-zero homomorphisms $\ker(q)\to\overline{\nabla}(u\cdot\lambda)$. Since $\nabla(u\cdot\lambda)$ has a filtration whose factors are all copies of $\overline{\nabla}(u\cdot\lambda)$, this also means that there are no non-zero homomorphisms $\ker(q)\to\nabla(u\cdot\lambda)$ and thus $q_k=0$. This is a contradiction, so we must have $\ker(q)=0$ and hence that $q:\im(\phi_i^M)\to \nabla(u\cdot\lambda)$ is an isomorphism. This proves the first part of the claim.
	
	We now need to show that $M=\bigoplus_{i=1}^n \im(\phi_i^M)$; comparing dimensions, it is enough to show that $M=\sum_{i=1}^n\im(\phi_i^M)$. Let us pick maps $\epsilon_1,\ldots,\epsilon_m:T_\chi(u\cdot\lambda)\to T_\chi(u\cdot\lambda)$ such that $\pi_{u\cdot\lambda}\circ\epsilon_1,\ldots,\pi_{u\cdot\lambda}\circ\epsilon_m$ is a basis of $\Hom_{\sC_\chi}(T_\chi(u\cdot\lambda),\nabla(u\cdot\lambda))$ (we can do this by Corollary~\ref{TiltBasCor}), so we have the diagram 
	\begin{eqnarray*}
		\begin{array}{c}\xymatrix{
				& & \overline{\Delta}(u\cdot\lambda) \ar@{->}[drr] \ar@{^{(}->}[d]^{\overline{\iota}_{u\cdot\lambda}}  & & \\ 
				T_\chi(u\cdot\lambda) \ar@{->}[rr]^{\varepsilon_j} \ar@{->}[drr] & & T_\chi(u\cdot\lambda) \ar@{->}[rr]^{\phi_i^{M}} \ar@{->>}[d]^{\pi_{u\cdot\lambda}} & & M  \\
				& & \nabla(u\cdot\lambda). & & 
				}\end{array}
		\end{eqnarray*}
	By \cite[Thm 4.43]{BS}, the maps $\phi_i^M\circ\varepsilon_j$ form a basis of $\Hom_{\sC_\chi}(T_\chi(u\cdot\lambda),M)$. Since we know that there exists a surjection $T_\chi(u\cdot\lambda)\twoheadrightarrow \nabla(u\cdot\lambda)$, there exists a surjective map $\Psi:\bigoplus_{k=1}^n T_\chi(u\cdot\lambda)\twoheadrightarrow M$. As $\Hom_{\sC_\chi}(\bigoplus_{k=1}^n T_\chi(u\cdot\lambda),M)\cong \bigoplus_{k=1}^n \Hom_{\sC_\chi}(T_\chi(u\cdot\lambda),M)$ there thus exist maps $\Phi_1,\ldots,\Phi_n:T_\chi(u\cdot\lambda)\to M$ such that $\Psi=\Phi_1\circ \tilde{p}_1+\cdots +\Phi_n\circ \tilde{p}_n$ (where $\tilde{p}_i:\bigoplus_{k=1}^n T_\chi(u\cdot\lambda)\to T_\chi(u\cdot\lambda)$ is projection to the $i$th coordinate). Since the maps $\phi_i^M\circ\varepsilon_j$ form a basis of $\Hom_{\sC_\chi}(T_\chi(u\cdot\lambda),M)$, there exist $a_{ij}^k\in\bK$ such that $$\Phi_k=\sum_{i,j} a_{ij}^k (\phi_i^M\circ \varepsilon_j)$$ for each $1\leq k\leq n$. 
	
	Let $x\in M$. Then there exists $y=(y_1,\ldots,y_n)\in \bigoplus_{k=1}^n T_\chi(u\cdot\lambda)$ such that $\Phi(y)=x$. Combining all the above, we have $$x=\Phi_1(y_1)+\cdots +\Phi_n(y_n)=\sum_{k=1}^n\left(\sum_{i,j} a_{ij}^k \phi_i^M( \varepsilon_j(y_k))\right) = \phi_1^M\left(\sum_{j,k} a_{1j}^k \varepsilon_j(y_k)\right)+\cdots +\phi_n^M\left(\sum_{j,k} a_{nj}^k \varepsilon_j(y_k)\right).$$ Hence, $x\in \sum_{i=1}^n\im(\phi_i^M)$, and so $M=\sum_{i=1}^n\im(\phi_i^M)$ as required.
\end{proof}

\begin{theorem}\label{LMSect}
	Let $M\in\sC_\chi(\lambda)$ admit a costandard filtration and let $(\Pi,e,(\phi_\pi^M)_{\pi\in\Pi})$ be a $\overline{\Delta}$-section of the $\nabla$-flag of $M$. Then the collection $(\Pi',e',(\phi_\pi^{T_\lambda^\mu(M)})_{\pi\in \Pi'})$ is a $\overline{\Delta}$-section of the $\nabla$-flag of $T_\lambda^\mu(M)$.
\end{theorem}

\begin{proof}
	
	For each $w\in W^{I,\mu}$ we need to show that $$\{\phi_\pi^{T_\lambda^\mu(M)}\mid \pi\in\Pi'\mbox{ such that }e'(\pi)=w\cdot\mu\}\subseteq \Hom_{\sC_\chi}(T_\chi(w\cdot\mu),T_\lambda^\mu(M))$$ is such that the set $$\{\phi_\pi^{T_\lambda^\mu(M)}\circ\overline{\iota}_{w\cdot\mu}\mid \pi\in\Pi'\mbox{ such that }e'(\pi)=w\cdot\mu\}$$ is a basis of $\Hom_{\sC_\chi}(\overline{\Delta}(w\cdot\mu),T_\lambda^\mu(M))$. The proof is almost identical to the proof of \cite[Prop. 3.5]{RW1}; we therefore only include the proof when $wsw^{-1}\in W_{I,p}$ since this case doesn't have an analogue in \cite{RW1}. Filling in the details of the case when $wsw^{-1}\notin W_{I,p}$ is left to the reader.
	
	Let us assume, therefore, that $wsw^{-1}\in W_{I,p}$. Let $$\Omega\coloneqq\{\kappa\in\Lambda_I\mid \kappa \preceq w\cdot\lambda\}\subseteq\Lambda_I\quad\mbox{   and   }\quad  \Omega'\coloneqq\Omega\setminus\{w\cdot\lambda\}\subseteq\Omega\subseteq\Lambda_I,$$ which are both lower sets of $\Lambda_I$. We have $$\Gamma_{\Omega'}(M)\subseteq \Gamma_{\Omega}(M)\subseteq M$$ and we know that $\Gamma_\Omega(M)/\Gamma_{\Omega'}(M)$ is isomorphic to a direct sum of copies of $\nabla(w\cdot\lambda)$. 

There are isomorphisms $$\Hom_{\sC_\chi}(\overline{\Delta}(w\cdot\mu),T_\lambda^\mu(M))\cong \Hom_{\sC_\chi}(\overline{\Delta}(w\cdot\mu),T_\lambda^\mu(\Gamma_\Omega(M)))\cong \Hom_{\sC_\chi}(\overline{\Delta}(w\cdot\mu),T_\lambda^\mu(\Gamma_\Omega(M))/T_\lambda^\mu(\Gamma_{\Omega'}(M)))$$ and $$\Hom_{\sC_\chi}(\overline{\Delta}(w\cdot\mu),T_\lambda^\mu(\Gamma_\Omega(M))/T_\lambda^\mu(\Gamma_{\Omega'}(M)))\cong \Hom_{\sC_\chi}(\overline{\Delta}(w\cdot\mu),T_\lambda^\mu(\Gamma_\Omega(M)/\Gamma_{\Omega'}(M))),$$ arising out of the facts that $$\dim\Hom_{\sC_\chi}(\overline{\Delta}(w\cdot\mu), T_\lambda^\mu(M)/T_\lambda^\mu(\Gamma_\Omega(M)))=(T_\lambda^\mu(M)/T_\lambda^\mu(\Gamma_\Omega(M)):\nabla(w\cdot\mu))=0$$ and $$\dim\Hom_{\sC_\chi}(\overline{\Delta}(w\cdot\mu),T_\lambda^\mu(\Gamma_{\Omega'}(M)))=(T_\lambda^\mu(\Gamma_{\Omega'}(M)):\nabla(w\cdot\mu))=0.$$

The maps $\phi_\pi^{T_\lambda^\mu(M)}\circ\overline{\iota}_{w\cdot\mu}$, $\pi\in \Pi'$ with $e'(\pi)=w\cdot\mu$, correspond to the maps $\phi_\pi^{T_\lambda^\mu(\Gamma_\Omega(M)/\Gamma_{\Omega'}(M))}\circ\overline{\iota}_{w\cdot\mu}$ under these isomorphisms, where the latter are defined using Lemmas~\ref{SubSect} and \ref{QuotSect} and the beginning of this subsection. This follows from the commutativity of the diagram
	\begin{equation*}
	\begin{array}{c}\xymatrix{
			\overline{\Delta}(w\cdot\mu) \ar@{^{(}->}[r] \ar@{->}[rdd]  & T_\chi(w\cdot\mu) \ar@{^{(}->}[r]^{\iota_i\qquad}  & T_\chi(w\cdot\mu)\oplus T_\chi(w\cdot\mu) \ar@{->}[r]^{\sim} & T_\lambda^\mu(T_\chi(w\cdot\lambda))\ar@{->}[ddl]^{T_\lambda^\mu(\phi_\pi^M)} \ar@{->}[ddl] \ar@{->}[ddd]^{T_\lambda^\mu\left(\phi_\pi^{\Gamma_\Omega(M)/\Gamma_{\Omega'}(M)}\right)}  \\
			& & & \\
			& T_\lambda^\mu(\Gamma_\Omega(M)) \ar@{->>}[drr] \ar@{^{(}->}[r] \ar@{->>}[d]  \ar@{<-}[uurr]^{T_\lambda^\mu\left(\phi_\pi^{\Gamma_\Omega(M)}\right)} & T_\lambda^\mu(M) & \\
			& T_\lambda^\mu(\Gamma_\Omega(M))/T_\lambda^\mu(\Gamma_{\Omega'}(M))   \ar@{->}[rr]^{\sim} & &  T_\lambda^\mu(\Gamma_\Omega(M)/\Gamma_{\Omega'}(M))
		}
	\end{array}
\end{equation*}
for each $i=1,2$ and $\pi\in\Pi$ such that $e(\pi)=w\cdot\lambda$.

Since $\Gamma_\Omega(M)/\Gamma_{\Omega'}(M)$ is isomorphic to a direct sum of copies of $\nabla(w\cdot\lambda)$, the maps $\phi_\pi^{T_\lambda^\mu(\Gamma_\Omega(M)/\Gamma_{\Omega'}(M))}\circ\overline{\iota}_{w\cdot\mu}$ give a basis of $\Hom_{\sC_\chi}(\overline{\Delta}(w\cdot\mu),T_\lambda^\mu(\Gamma_\Omega(M)/\Gamma_{\Omega'}(M)))$ by Lemmas~\ref{SubSect} and \ref{QuotSect} and Proposition~\ref{SumSection1}. This implies the result.

%

\end{proof}

The procedure described in this subsection therefore allow us to construct a $\overline{\Delta}$-section of the $\nabla$-flag of $T_\lambda^\mu(M)$ out of a $\overline{\Delta}$-section of the $\nabla$-flag of $M$, for any $M\in\sC_\chi(\lambda)$ admitting a costandard filtration. Combining this with Lemmas~ \ref{Stand2PropStand} and \ref{PropStand2Stand}, we also get a procedure to construct a $\Delta$-section of the $\nabla$-flag of $T_\lambda^\mu(M)$ out of a $\Delta$-section of the $\nabla$-flag of $M$.

\subsection{Sections of $\nabla$-Flags: Translation off of a Wall}\label{ss5.3}

For this subsection, let $s,\lambda$ and $\mu$ be as in Notation~\ref{NOT}.

Let $M\in\sC_\chi(\mu)$ have a costandard filtration and let $(\Pi,e,(\phi_\pi^M)_{\pi\in\Pi})$ be a $\overline{\Delta}$-section of the $\nabla$-flag of $M$. In this subsection, we construct a $\overline{\Delta}$-section of the $\nabla$-flag of $T_\mu^\lambda(M)$.

Define 
$$\Pi_1=\{\pi\in\Pi\mid e(\pi)=w\cdot\mu\mbox{ for }w\in W^{I,\mu}\mbox{ such that } wsw^{-1}\notin W_{I,p}\}$$ and $$\Pi_2=\{\pi\in\Pi\mid e(\pi)=w\cdot\mu\mbox{ for }w\in W^{I,\mu}\mbox{ such that } wsw^{-1}\in W_{I,p}\}.$$ We then set $$\Pi'=\left(\Pi_1\times\{1,2\}\right)\cup \Pi_2$$ and define $e':\Pi'\to \Lambda_I$ as follows. If $\pi\in \Pi_1$ we define $$e'(\pi,1)=ws\cdot\lambda\qquad \mbox{and} \qquad e'(\pi,2)=w\cdot\lambda.$$ Alternatively, if  $\pi\in \Pi_2$, we define $$e'(\pi)=w\cdot\lambda.$$

Now, for each $w\in W^{I,\lambda}$ we define a family $(\phi_\pi^{T_\mu^\lambda(M)})_{\pi\in (e')^{-1}(w\cdot\lambda)}$ of maps $T_\chi(w\cdot\mu)\to T_\lambda^\mu(M)$. We will need to distinguish between two cases.

{\bf Case 1:} $wsw^{-1}\in W_{I,p}$.

In this case, for each $\pi\in\Pi_2$ with $e(\pi)=w\cdot\mu$ we define the map $\phi_{\pi}^{T_\mu^\lambda(M)}:T_\chi(w\cdot\lambda)\to T_\mu^\lambda(M)$ via $$\phi_{\pi}^{T_\lambda^\mu(M)}:T_\chi(w\cdot\lambda)\cong T_\mu^\lambda(T_\chi(w\cdot\mu))\xrightarrow{T_\mu^\lambda(\phi_{\pi}^M)} T_\mu^\lambda(M).$$

{\bf Case 2:} $wsw^{-1}\notin W_{I,p}$.

We define a map $T_\chi(w\cdot\lambda)\to T_\mu^\lambda(M)$ and a map $T_\chi(ws\cdot\lambda)\to T_\mu^\lambda(M)$ for each $\pi\in \Pi_1$ with $e(\pi)=w\cdot\mu$. Recalling that $ws\cdot\lambda < w\cdot\lambda$, we make the following definitions:

For each $\pi\in\Pi_1$ with $e(\mu)=w\cdot\mu$, we define the map $\phi_{(\pi,1)}^{T_\mu^\lambda(M)}:T_\chi(ws\cdot\lambda)\to T_\mu^\lambda(M)$ to be the image of the map $\phi_\pi^M$ under the composition $$\Hom_{\sC_\chi}(T_\chi(w\cdot\mu),M)\to\Hom_{\sC_\chi}(T_\chi(w\cdot\mu)\oplus T_{\prec w\cdot\mu},M)\cong \Hom_{\sC_\chi}(T_\lambda^\mu(T_\chi(ws\cdot\lambda)),M)\cong \Hom_{\sC_\chi}(T_\chi(ws\cdot\lambda),T_\mu^\lambda(M)).$$ 

Furthermore, for each $\pi\in \Pi_1$ with $e(\pi)=w\cdot\mu$, we define the map $\phi_{(\pi,2)}^{T_\mu^\lambda(M)}:T_\chi(w\cdot\lambda)\to T_\mu^\lambda(M)$ via the composition $$\phi_{(\pi,2)}^{T_\mu^\lambda(M)}:T_\chi(w\cdot\lambda)\cong T_\mu^\lambda(T_\chi(w\cdot\mu))\xrightarrow{T_\mu^\lambda(\phi_\pi^M)} T_\mu^\lambda(M).$$

\begin{prop}
	When $M=\nabla(u\cdot\mu)$ for some $u\in W^{I,\mu}$, the collection $(\Pi',e',(\phi_\pi^{T_\mu^\lambda(M)})_{\pi\in \Pi})$ is a $\overline{\Delta}$-section of the $\nabla$ of $T_\mu^\lambda(M)$.
\end{prop}

\begin{proof}
	A $\overline{\Delta}$-section of the $\nabla$-flag of ${\nabla(u\cdot\mu)}$ consists of $\Pi=\{\ast\}$, $e(\ast)=u\cdot\mu$ and a map $\phi_\ast^{\nabla(u\cdot\mu)}:T_\chi(u\cdot\mu)\to M$ such that $\phi_\ast^{\nabla(u\cdot\mu)}\circ \overline{\iota}_{u\cdot\lambda}\neq 0$. 
	
	For $w\in W^{I,\lambda}$ we can compute that 
	\begin{equation*}
		\begin{split}
			\dim\Hom_{\sC_\chi}(\overline{\Delta}(w\cdot\lambda),T_\mu^\lambda(\nabla(u\cdot\mu))) & = \dim\Hom_{\sC_\chi}(T_\lambda^\mu(\overline{\Delta}(w\cdot\lambda)),\nabla(u\cdot\mu)) \\ 
			& = \dim\Hom_{\sC_\chi}(\overline{\Delta}(w\cdot\mu),\nabla(u\cdot\mu)) \\
			& =  \twopartdef{1}{u=w\mbox{ or }u=ws,}{0}{\mbox{not.}}
		\end{split}
	\end{equation*}
	We will tackle two cases separately.
	
	{\bf Case 1:} $usu^{-1}\in W_{I,p}$.
	
	In this case we have to prove that $\phi_\ast^{T_\mu^\lambda({\nabla(u\cdot\mu)})}\circ\overline{\iota}_{u\cdot\lambda}\neq 0$, i.e. that the composition
	$$\overline{\Delta}(u\cdot\lambda)\hookrightarrow T_\chi(u\cdot\lambda)\cong T_\mu^\lambda(T_\chi(u\cdot\mu))\xrightarrow{T_\mu^\lambda(\phi_{\ast}^{\nabla(u\cdot\mu)})} T_\mu^\lambda({\nabla(u\cdot\mu)})$$ is non-zero. By Proposition~\ref{NZeroSurj}, we know that $\phi_\ast^M$ is surjective and so $T_\mu^\lambda(\phi_\ast^{\nabla(u\cdot\mu)})$ is surjective. Hence $$[\ker(T_\mu^\lambda(\phi_\ast^{\nabla(u\cdot\mu)})):L_\chi(u\cdot\lambda)]=[T_\mu^\lambda(T_\chi(u\cdot\mu)):L_\chi(u\cdot\lambda)]-[T_\mu^\lambda(\nabla(u\cdot\mu)):L_\chi(u\cdot\lambda)]=0$$ and thus $\phi_\ast^{T_\mu^\lambda({\nabla(u\cdot\mu)})}\circ\overline{\iota}_{u\cdot\lambda}\neq 0$ as required.
	
	{\bf Case 2:} $usu^{-1}\notin W_{I,p}$.
	
	In this case, we have to prove that $\phi_{(\ast,1)}^{T_\mu^\lambda({\nabla(u\cdot\mu)})}\circ \overline{\iota}_{us\cdot\lambda}\neq 0$ and $\phi_{(\ast,2)}^{T_\mu^\lambda({\nabla(u\cdot\mu)})}\circ \overline{\iota}_{u\cdot\lambda}\neq 0$. For the first, we consider the commutative diagram  	\begin{eqnarray*}
		\begin{array}{c}\xymatrix{
				\Hom_{\sC_\chi}(T_\chi(u\cdot\mu),{\nabla(u\cdot\mu)}) \ar@{->}[r] &  \Hom_{\sC_\chi}(T_\lambda^\mu(T_\chi(us\cdot\lambda)),{\nabla(u\cdot\mu)}) \ar@{->}[r]^{\sim} \ar@{->}[d]^{\circ T_\lambda^\mu(\overline{\iota}_{us\cdot\lambda})} & \Hom_{\sC_\chi}(T_\chi(us\cdot\lambda),T_\mu^\lambda({\nabla(u\cdot\mu)})) \ar@{->}[d]^{\circ \overline{\iota}_{us\cdot\lambda}}\\
				 & \Hom_{\sC_\chi}(T_\lambda^\mu(\overline{\Delta}(us\cdot\lambda)),{\nabla(u\cdot\mu)}) \ar@{->}[r]^{\sim} &  \Hom_{\sC_\chi}(\overline{\Delta}(us\cdot\lambda),T_\mu^\lambda({\nabla(u\cdot\mu)})).
		}\end{array} 
	\end{eqnarray*}
	From this diagram, showing that $\phi_{(\ast,1)}^{T_\mu^\lambda({\nabla(u\cdot\mu)})}\circ \overline{\iota}_{us\cdot\lambda}\neq 0$ reduces to showing that the composition $$T_\lambda^\mu(\overline{\Delta}(us\cdot\lambda))\hookrightarrow T_\lambda^\mu(T_\chi(us\cdot\lambda))\cong T_\chi(u\cdot\mu)\oplus T_{\prec u\cdot\mu}\twoheadrightarrow T_\chi(u\cdot\mu)\xrightarrow{\phi_\ast^{\nabla(u\cdot\mu)}}{\nabla(u\cdot\mu)}$$ is non-zero. As above, $\phi_\ast^{\nabla(u\cdot\mu)}$ is surjective, and we note that $[T_\chi(u\cdot\mu)\oplus T_{\prec u\cdot\mu}:L_\chi(u\cdot\mu)]=[\nabla(u\cdot\lambda):L_\chi(u\cdot\mu)]$. Since $T_\lambda^\mu(\overline{\Delta}(us\cdot\lambda))=\overline{\Delta}(u\cdot\mu)$, this is enough to show that the composition is non-zero. This implies $\phi_{(\ast,1)}^{T_\mu^\lambda({\nabla(u\cdot\mu)})}\circ \overline{\iota}_{us\cdot\lambda}\neq 0$.
	
	To show that $\phi_{(\ast,2)}^{T_\mu^\lambda({\nabla(u\cdot\mu)})}\circ \overline{\iota}_{u\cdot\lambda}\neq 0$, we need that the composition $$\overline{\Delta}(u\cdot\lambda)\hookrightarrow  T_\chi(u\cdot\lambda)\cong T_\mu^\lambda(T_\chi(u\cdot\mu))\xrightarrow{T_\mu^\lambda(\phi_\ast^{\nabla(u\cdot\mu)})} T_\mu^\lambda({\nabla(u\cdot\mu)})$$ is non-zero. Since $T_\mu^\lambda(\phi_\ast^{\nabla(u\cdot\mu)})$ is surjective, and $[T_\mu^\lambda(T_\chi(u\cdot\mu)):L_\chi(u\cdot\lambda)]=[T_\mu^\lambda(\nabla(u\cdot\mu)):L_\chi(u\cdot\lambda)]$, the result follows by similar arguments to those used in the proof of Proposition~\ref{LMNabla}.
\end{proof}

\begin{prop}\label{SumSection2}
	When $M=\bigoplus_{i=1}^n\nabla(u\cdot\mu)$ for some $u\in W^{I,\mu}$ and $n\geq 1$, the collection $(\Pi',e',(\phi_\pi^{T_\mu^\lambda(M)})_{\pi\in \Pi'})$ is a $\overline{\Delta}$-section of the $\nabla$-flag of $T_\mu^\lambda(M)$.
\end{prop}

\begin{proof}
	The proof is identical to the proof of Proposition~\ref{SumSection1}.
\end{proof}

\begin{theorem}\label{MLSect}
	Let $M\in\sC_\chi(\mu)$ admit a costandard filtration and let $(\Pi,e,(\phi_\pi^M)_{\pi\in\Pi})$ be a $\overline{\Delta}$-section of the $\nabla$-flag of $M$. Then the collection $(\Pi',e',(\phi_\pi^{T_\lambda^\mu(M)})_{\pi\in \Pi'})$ is a $\overline{\Delta}$-section of the $\nabla$-flag of $T_\mu^\lambda(M)$.
\end{theorem}

\begin{proof}
	For each $w\in W^{I,\lambda}$ we need to show that $$\{\phi_\pi^{T_\mu^\lambda(M)}\mid \pi\in\Pi'\mbox{ such that }e'(\pi)=w\cdot\lambda\}\subseteq \Hom_{\sC_\chi}(T_\chi(w\cdot\lambda),T_\mu^\lambda(M))$$ is such that the set $$\{\phi_\pi^{T_\mu^\lambda(M)}\circ\overline{\iota}_{w\cdot\lambda}\mid \pi\in\Pi'\mbox{ such that }e'(\pi)=w\cdot\lambda\}\subseteq \Hom_{\sC_\chi}(\overline{\Delta}(w\cdot\lambda),T_\mu^\lambda(M))$$ is a basis of $\Hom_{\sC_\chi}(\overline{\Delta}(w\cdot\lambda),T_\mu^\lambda(M))$. The proof of this is very similar to the proofs of \cite[Prop. 3.8]{RW1} and Theorem~\ref{LMSect}, and thus is omitted.
\end{proof}

\subsection{Sections of $\nabla$-Flags: Wall-Crossing Functors}\label{ss5.4}

Combining Theorem~\ref{LMSect} and \ref{MLSect}, we therefore get a method to construct a $\overline{\Delta}$-section of the $\nabla$-flag of $\Theta_s(M)$ out of a  $\overline{\Delta}$-section of the $\nabla$-flag of $M$ for any $M\in\sC_\chi(\lambda)$ admitting a costandard filtration. More precisely, we get the following theorem.

\begin{theorem}
	Let $s,\lambda$ and $\mu$ be as in Notation~\ref{NOT}, and let $w\in W^{I,\lambda}$. Let $M\in\sC_\chi(\lambda)$ have a costandard filtration.
	
	\begin{enumerate}
		\item Suppose that $wsw^{-1}\notin W_{I,p}$ and that $ws\cdot\lambda < w\cdot\lambda$. Let  $$f_1,\ldots,f_n\in\Hom_{\sC_\chi}(T_\chi(w\cdot\lambda),M)$$ be such that $f_1\circ\overline{\iota}_{w\cdot\lambda},\ldots,f_n\circ\overline{\iota}_{w\cdot\lambda}$ form a basis of $\Hom_{\sC_\chi}(\overline{\Delta}(w\cdot\lambda),M)$ and let $$g_1,\ldots,g_m\in\Hom_{\sC_\chi}(T_\chi(ws\cdot\lambda),M)$$ be such that $g_1\circ\overline{\iota}_{ws\cdot\lambda},\ldots,g_m\circ\overline{\iota}_{ws\cdot\lambda}$ form a basis of $\Hom_{\sC_\chi}(\overline{\Delta}(ws\cdot\lambda),M)$. Then the following maps $T_\chi(w\cdot\lambda)\to \Theta_s(M)$ give a basis of $\Hom_{\sC_\chi}(\overline{\Delta}(w\cdot\lambda),\Theta_s(M))$ when precomposed with $\overline{\iota}_{w\cdot\lambda}$:
		$$T_\chi(w\cdot\lambda)\xrightarrow{\sim} T_\mu^\lambda(T_\chi(w\cdot\mu))\to T_\mu^\lambda (T_\lambda^\mu T_\mu^\lambda(T_\chi(w\cdot\mu)))\xrightarrow{\sim} \Theta_s(T_\chi(w\cdot\lambda))\xrightarrow{\Theta_s(f_i)} \Theta_s(M)$$ for $i=1,\ldots,n$, and
		$$T_\chi(w\cdot\lambda)\xrightarrow{\sim} T_\mu^\lambda(T_\chi(w\cdot\mu))\hookrightarrow T_\mu^\lambda(T_\chi(w\cdot\mu)\oplus T_{\prec w\cdot\mu})\xrightarrow{\sim} \Theta_s(T_\chi(ws\cdot\lambda))\xrightarrow{\Theta_s(g_j)} \Theta_s(M)$$ for $j=1,\ldots,m$. 
		
		Similarly, the following maps $T_\chi(ws\cdot\lambda)\to \Theta_s(M)$ give a basis of $\Hom_{\sC_\chi}(\overline{\Delta}(ws\cdot\lambda),\Theta_s(M))$ when precomposed with $\overline{\iota}_{ws\cdot\lambda}$:
		$$T_\chi(ws\cdot\lambda)\to T_\mu^\lambda(T_\chi(w\cdot\mu))\to T_\mu^\lambda ( T_\lambda^\mu T_\mu^\lambda(T_\chi(w\cdot\mu)))\xrightarrow{\sim} \Theta_s(T_\chi(w\cdot\lambda))\xrightarrow{\Theta_s(f_i)} \Theta_s(M)$$ for $i=1,\ldots,n$, and
		$$T_\chi(ws\cdot\lambda)\to  T_\mu^\lambda(T_\chi(w\cdot\mu))\hookrightarrow T_\mu^\lambda(T_\chi(w\cdot\mu)\oplus T_{\prec w\cdot\mu})\xrightarrow{\sim} \Theta_s(T_\chi(ws\cdot\lambda))\xrightarrow{\Theta_s(g_j)} \Theta_s(M)$$ for $j=1,\ldots,m$, where the map $T_\chi(ws\cdot\lambda)\to  T_\mu^\lambda(T_\chi(w\cdot\mu))$ is the composition 		
		$$T_\chi(ws\cdot\lambda)\to T_\mu^\lambda T_\lambda^\mu(T_\chi(ws\cdot\lambda))\xrightarrow{\sim} T_\mu^\lambda(T_\chi(w\cdot\mu) \oplus T_{\prec w\cdot\mu})\twoheadrightarrow T_\mu^\lambda(T_\chi(w\cdot\mu))$$
		
		\item Suppose that $wsw^{-1}\in W_{I,p}$. Suppose that $f_1,\ldots,f_n\in\Hom_{\sC_\chi}(T_\chi(w\cdot\lambda),M)$ are such that $f_1\circ\overline{\iota}_{w\cdot\lambda},\ldots,f_n\circ\overline{\iota}_{w\cdot\lambda}$ form a basis of $\Hom_{\sC_\chi}(\overline{\Delta}(w\cdot\lambda),M)$. Then the following maps $T_\chi(w\cdot\lambda)\to \Theta_s(M)$ give a basis of $\Hom_{\sC_\chi}(\overline{\Delta}(w\cdot\lambda),\Theta_s(M))$ when precomposed with $\overline{\iota}_{w\cdot\lambda}$:
		$$T_\chi(w\cdot\lambda)\xrightarrow{\sim} T_\mu^\lambda(T_\chi(w\cdot\mu))\xhookrightarrow{T_\mu^\lambda(\iota_i)} T_\mu^\lambda(T_\chi(w\cdot\mu)\oplus T_\chi(w\cdot\mu))\xrightarrow{\sim} \Theta_s(T_\chi(w\cdot\lambda))\xrightarrow{\Theta_s(f_j)} \Theta_s(M)$$ for $i=1,2$ and $j=1,\ldots,n$.
	\end{enumerate}
	
\end{theorem}	

The analogous statement for morphisms $\Theta_s(M)\to T_\chi(w\cdot\lambda)$ can be obtained by duality. Details are left to the reader.

We can also combine Theorems~\ref{LMSect} and \ref{MLSect} with Lemmas~\ref{Stand2PropStand} and \ref{PropStand2Stand} to get a procedure for constructing a $\Delta$-section of the $\nabla$-flag of $\Theta_s(M)$ out of a $\Delta$-section of the $\nabla$-flag of $M$, for any $M\in\sC_\chi(\mu)$ admitting a costandard filtration, via the following diagram:

	\begin{equation*}
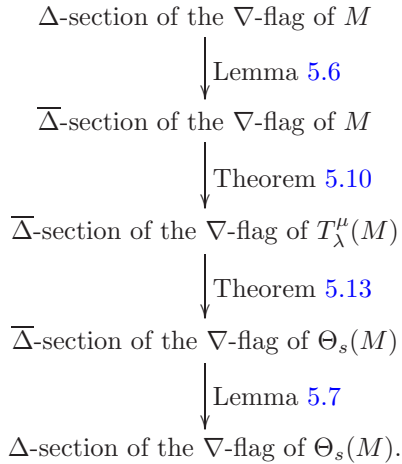
	\begin{array}{c}\xymatrix{
			\Delta\mbox{-section of the }\nabla\mbox{-flag of }M \ar@{->}[d]^{\mbox{Lemma~\ref{Stand2PropStand}}}\\
			\overline{\Delta}\mbox{-section of the }\nabla\mbox{-flag of }M \ar@{->}[d]^{\mbox{Theorem~\ref{LMSect}}}\\
			\overline{\Delta}\mbox{-section of the }\nabla\mbox{-flag of }T_\lambda^\mu(M) \ar@{->}[d]^{\mbox{Theorem~\ref{MLSect}}}\\
			\overline{\Delta}\mbox{-section of the }\nabla\mbox{-flag of }\Theta_s(M) \ar@{->}[d]^{\mbox{Lemma~\ref{PropStand2Stand}}}\\ 
			\Delta\mbox{-section of the }\nabla\mbox{-flag of }\Theta_s(M).	
		}
	\end{array}
	\end{equation*}

We can combine all this into the following theorem.
	
	\begin{theorem}\label{BasTilt1}
			Let $s,\lambda$ and $\mu$ be as in Notation~\ref{NOT} and let $M\in\sC_\chi(\lambda)$ have a costandard filtration. Let $u\in W^{I,\lambda}$ be such that $usu^{-1}\notin W_{I,p}$ and $us\cdot\lambda<u\cdot\lambda$. Let $f_1,\ldots,f_n$ be a family of morphisms in $\Hom_{\sC_\chi}(T_\chi(u\cdot\lambda),M)$ such that the maps $f_1\circ \iota_{u\cdot\lambda},\ldots,f_n\circ \iota_{u\cdot\lambda}$ form a basis of $\Hom_{\sC_\chi}(\Delta(u\cdot\lambda),M)$. Let also $g_1,\ldots,g_m$ be a family of morphisms in $\Hom_{\sC_\chi}(T_\chi(us\cdot\lambda),M)$ such that the maps $g_1\circ \iota_{us\cdot\lambda},\ldots,g_m\circ \iota_{us\cdot\lambda}$ form a basis of $\Hom_{\sC_\chi}(\Delta(us\cdot\lambda),M)$. Then there exist families of morphisms $F_1,\ldots,F_r:T_\chi(u\cdot\lambda)\to \Theta_s(T_\chi(u\cdot\lambda))$ and $G_1,\ldots,G_t:T_\chi(u\cdot\lambda)\to \Theta_s(T_\chi(us\cdot\lambda))$ such that the compositions $$\Delta(u\cdot\lambda)\hookrightarrow T_\chi(u\cdot\lambda)\xrightarrow{F_i} \Theta_s(T_\chi(u\cdot\lambda))\xrightarrow{\Theta_s(f_j)} \Theta_s(M)$$ and $$\Delta(u\cdot\lambda)\hookrightarrow T_\chi(u\cdot\lambda)\xrightarrow{G_k} \Theta_s(T_\chi(us\cdot\lambda))\xrightarrow{\Theta_s(g_l)} \Theta_s(M)$$ form a basis of $\Hom_{\sC_\chi}(\Delta(u\cdot\lambda),\Theta_s(M))$.
	\end{theorem}

\begin{proof}
	We may assume that the maps $f_j$ and $g_l$ form part of a $\Delta$-section of the $\nabla$-flag of $M$. The process described above can be used to get a $\Delta$-section of the $\nabla$-flag of $\Theta_s(M)$. In particular, this process results in a collection of maps $T_\chi(u\cdot\lambda)\to\Theta_s(M)$ which induce a basis of $\Hom_{\sC_\chi}(\Delta(u\cdot\lambda),\Theta_s(M))$. Analysing precisely which maps we get through this process, we see that they are the maps $$\Delta(u\cdot\lambda)\hookrightarrow T_\chi(u\cdot\lambda)\xrightarrow{\sigma} T_\chi(u\cdot\lambda)\xrightarrow{\sim} T_\mu^\lambda(T_\chi(u\cdot\mu))\hookrightarrow T_\mu^\lambda(T_\chi(u\cdot\mu)\oplus T_{\prec u\cdot\mu})\xrightarrow{\sim} \Theta_s(T_\chi(us\cdot\lambda))\xrightarrow{\Theta_s(g_l)} \Theta_s(M)$$ for $l=1,\ldots,m$ and the maps
	$$\Delta(u\cdot\lambda)\hookrightarrow T_\chi(u\cdot\lambda)\xrightarrow{\sigma} T_\chi(u\cdot\lambda)\xrightarrow{\sim} T_\mu^\lambda(T_\chi(u\cdot\mu))\hookrightarrow T_\mu^\lambda T_\lambda^\mu T_\mu^\lambda (T_\chi(u\cdot\mu))\xrightarrow{\sim} \Theta_s(T_\chi(u\cdot\lambda))\xrightarrow{\Theta_s(f_j)} \Theta_s(M)$$ for $j=1,\ldots,n$, where in both cases $\sigma$ runs over a family of maps $T_\chi(u\cdot\lambda)\to T_\chi(u\cdot\lambda)$ such that the maps $\pi_{u\cdot\lambda}\circ \sigma\circ \iota_{u\cdot\lambda}$ form a basis of $\Hom_{\sC_\chi}(\Delta(u\cdot\lambda),\nabla(u\cdot\lambda))$. From this description, it is clear how to chose the maps $F_i$ and $G_k$ to get the desired result.
\end{proof}

\begin{theorem}\label{BasTilt2}
	Let $s,\lambda$ and $\mu$ be as in Notation~\ref{NOT} and let $M\in\sC_\chi(\lambda)$ have a costandard filtration. Let $u\in W^{I,\lambda}$ be such that $usu^{-1}\notin W_{I,p}$ and $us\cdot\lambda<u\cdot\lambda$. Let $f_1,\ldots,f_n$ be a family of morphisms in $\Hom_{\sC_\chi}(T_\chi(u\cdot\lambda),M)$ such that the maps $f_1\circ \iota_{u\cdot\lambda},\ldots,f_n\circ \iota_{u\cdot\lambda}$ form a basis of $\Hom_{\sC_\chi}(\Delta(u\cdot\lambda),M)$. Let also $g_1,\ldots,g_m$ be a family of morphisms in $\Hom_{\sC_\chi}(T_\chi(us\cdot\lambda),M)$ such that the maps $g_1\circ \iota_{us\cdot\lambda},\ldots,g_m\circ \iota_{us\cdot\lambda}$ form a basis of $\Hom_{\sC_\chi}(\Delta(us\cdot\lambda),M)$. Then there exists families of morphisms $F_1,\ldots,F_r:T_\chi(us\cdot\lambda)\to \Theta_s(T_\chi(u\cdot\lambda))$ and $G_1,\ldots,G_t:T_\chi(us\cdot\lambda)\to \Theta_s(T_\chi(us\cdot\lambda))$ such that the compositions $$\Delta(us\cdot\lambda)\hookrightarrow T_\chi(us\cdot\lambda)\xrightarrow{F_i} \Theta_s(T_\chi(u\cdot\lambda))\xrightarrow{\Theta_s(f_j)} \Theta_s(M)$$ and $$\Delta(us\cdot\lambda)\hookrightarrow T_\chi(us\cdot\lambda)\xrightarrow{G_k} \Theta_s(T_\chi(us\cdot\lambda))\xrightarrow{\Theta_s(g_l)} \Theta_s(M)$$ form a basis of $\Hom_{\sC_\chi}(\Delta(us\cdot\lambda),\Theta_s(M))$.
\end{theorem}

\begin{proof}
	The proof is very similar to the proof of Theorem~\ref{BasTilt1}.
\end{proof}

\begin{theorem}\label{BasTilt3}
	Let $s,\lambda$ and $\mu$ be as in Notation~\ref{NOT} and let $M\in\sC_\chi(\lambda)$ have a costandard filtration. Let $u\in W^{I,\lambda}$ be such that $usu^{-1}\in W_{I,p}$. Let $f_1,\ldots,f_n$ be a family of morphisms in $\Hom_{\sC_\chi}(T_\chi(u\cdot\lambda),M)$ such that the maps $f_1\circ \iota_{u\cdot\lambda},\ldots,f_n\circ \iota_{u\cdot\lambda}$ form a basis of $\Hom_{\sC_\chi}(\Delta(u\cdot\lambda),M)$. Then there exists families of morphisms $F_1,\ldots,F_m:T_\chi(u\cdot\lambda)\to \Theta_s(T_\chi(u\cdot\lambda))$ such that the compositions $$\Delta(u\cdot\lambda)\hookrightarrow T_\chi(u\cdot\lambda)\xrightarrow{F_i} \Theta_s(T_\chi(u\cdot\lambda))\xrightarrow{\Theta_s(f_j)} \Theta_s(M)$$ form a basis of $\Hom_{\sC_\chi}(\Delta(u\cdot\lambda),\Theta_s(M))$.
\end{theorem}

\begin{proof}
	Again, the proof is very similar to the proof of Theorem~\ref{BasTilt1}.
\end{proof}

Recall that we write $S_p$ for the subset of $W_p$ consisting of reflections in a wall of $\overline{C}$. Note that $(W_p,S_p)$ is a Coxeter system, and let us denote by $l:W_p\to \bN$ the associated length function. Given a sequence $\underline{w}=(s_1,\ldots,s_d)$ of elements of $S_p$ we write $w$ for the product $s_1\cdots s_d$, and we say that $\underline{w}$ is an {\bf expression} (of $w$). We call $\underline{w}$ a {\bf reduced expression} if $l(w)=d$.

Given an expression $\underline{w}=(s_1,\ldots,s_d)$, let us write $$T_\chi(\underline{w})\coloneqq\Theta_{s_d}\circ\cdots\circ \Theta_{s_1}(T_\chi(\lambda)).$$

Propositions~\ref{DomExp} and \ref{TiltSummand} tell us that for each $w\in W^{I,\lambda}$ such that $w\cdot\lambda\in X(T)_{+}$, there exists an expression $\underline{w}=(s_1,\ldots,s_d)$ of $w$ such that $T_\chi(w\cdot\lambda)$ appears with multiplicity 1 in $T_\chi(\underline{w})$. As an application of this observation we have the following result, which should be compared with \cite[Prop. 3.10]{RW1}.

%

\begin{prop}\label{TruncWall}
	Let $s,\lambda$ and $\mu$ be as in Notation~\ref{NOT} and let $w\in W^{I,\lambda}$ such that $w\cdot\lambda\in X(T)_{+}$. Let $\underline{w}=(s_1,\ldots,s_d)\in S_p^d$ be a reduced expression of $w$ as in Proposition~\ref{DomExp}. Let also $M\in\sC_\chi(\lambda)$ have a costandard filtration. 
	
	Suppose $wsw^{-1}\notin W_{I,p}$ and $ws\cdot\lambda < w\cdot\lambda$, and suppose that $s=s_d$ so that $\underline{y}=(s_1,\ldots,s_{d-1})$ is a reduced expression of $ws\cdot\lambda\in X(T)_{+}$ satisfying the results of Proposition~\ref{DomExp}. Let $f_1,\ldots,f_n$ be a family of morphisms in $\Hom_{\sC_\chi}(T_\chi(\underline{w}),M)$ whose images under $j_{\succeq w\cdot\lambda}$ form a basis of $\Hom_{\sC_\chi^{\succeq w\cdot\lambda}}(T_\chi(\underline{w}),M)$. Let also $g_1,\ldots,g_m$ be a family of morphisms in $\Hom_{\sC_\chi}(T_\chi(\underline{y}),M)$ whose images under $j_{\succeq ws\cdot\lambda}$ form a basis of $\Hom_{\sC_\chi^{\succeq ws\cdot\lambda}}(T_\chi(\underline{y}),M)$. Then there exist families of morphisms $F'_1,\ldots,F'_r:T_\chi(\underline{w})\to \Theta_s(T_\chi(\underline{w}))$ and $G'_1,\ldots,G'_t:T_\chi(\underline{w})\to \Theta_s(T_\chi(\underline{y}))$ such that the images under $j_{\succeq w\cdot\lambda}$ of the compositions $$T_\chi(\underline{w})\xrightarrow{F'_i} \Theta_s(T_\chi(\underline{w}))\xrightarrow{\Theta_s(f_j)} \Theta_s(M)$$ and $$ T_\chi(\underline{w})\xrightarrow{G'_k} \Theta_s(T_\chi(\underline{y}))\xrightarrow{\Theta_s(g_l)} \Theta_s(M)$$ span $\Hom_{\sC_\chi^{\succeq w\cdot\lambda}}(T_\chi(\underline{w}),\Theta_s(M))$.
	
	Furthermore, there exist families of morphisms $F''_1,\ldots,F''_{r'}:T_\chi(\underline{y})\to \Theta_s(T_\chi(\underline{w}))$ and $G''_1,\ldots,G''_{t'}:T_\chi(\underline{y})\to \Theta_s(T_\chi(\underline{y}))$ such that the images under $j_{\succeq ws\cdot\lambda}$ of the compositions $$T_\chi(\underline{y})\xrightarrow{F''_i} \Theta_s(T_\chi(\underline{w}))\xrightarrow{\Theta_s(f_j)} \Theta_s(M)$$ and $$ T_\chi(\underline{y})\xrightarrow{G''_k} \Theta_s(T_\chi(\underline{y}))\xrightarrow{\Theta_s(g_l)} \Theta_s(M)$$ span $\Hom_{\sC_\chi^{\succeq ws\cdot\lambda}}(T_\chi(\underline{y}),\Theta_s(M))$.
\end{prop}

\begin{proof}
	(1) Since $w\cdot\lambda\in X(T)_{+}$ and the expression $\underline{w}$ satisfies the properties discussed in Proposition~\ref{DomExp}, Proposition~\ref{TiltSummand} shows that $T_\chi(w\cdot\lambda)$ appears as a direct summand of $T_\chi(\underline{w})$ with multiplicity 1, and that all other direct summands are of the form $T_\chi(u\cdot\lambda)$ with $d(u\cdot\lambda)< d(w\cdot\lambda)$. Accordingly, let us fix $\underline{\iota}$ and $\underline{\pi}$ via $$T_\chi(w\cdot\lambda)\xhookrightarrow{\underline{\iota}} T_\chi(\underline{w}) \overset{\underline{\pi}}{\twoheadrightarrow} T_\chi(w\cdot\lambda).$$ Note that $d(u\cdot\lambda)< d(w\cdot\lambda)$ implies $u\cdot\lambda\nsucceq w\cdot\lambda$. The maps $f_j\circ \underline{\iota}\in \Hom_{\sC_\chi}(T_\chi(w\cdot\lambda),M)$ are thus such that the compositions $f_j\circ \underline{\iota}\circ \iota_{w\cdot\lambda}$ form a basis of $\Hom_{\sC_\chi}(\Delta(w\cdot\lambda),M)$, by Proposition~\ref{Trunc} and the first isomorphism in the proof of Proposition~\ref{EndStand}.
	
	Similarly, under the given assumptions we have $ws\cdot\lambda\in X(T)_{+}\cap C_I$ and $\underline{y}$ (an expression of $ws$) has a description as given in Proposition~\ref{DomExp}. So we once again have $T_\chi(ws\cdot\lambda)$ appearing as a direct summand with multiplicity 1 in $T_\chi(\underline{y})$, and we have the maps $$T_\chi(ws\cdot\lambda)\xhookrightarrow{\underline{\iota}'} T_\chi(\underline{y}) \overset{\underline{\pi}'}{\twoheadrightarrow} T_\chi(ws\cdot\lambda).$$ The maps $g_l\circ \underline{\iota}'\in \Hom_{\sC_\chi}(T_\chi(ws\cdot\lambda),M)$ are such that the compositions $g_l\circ \underline{\iota}'\circ \iota_{ws\cdot\lambda}$ form a basis of $\Hom_{\sC_\chi}(\Delta(ws\cdot\lambda),M)$ as above.
	
	By Theorem~\ref{BasTilt1} there exist families of morphisms $F_j:T_\chi(w\cdot\lambda)\to \Theta_s(T_\chi(w\cdot\lambda))$ and $G_l:T_\chi(w\cdot\lambda)\to \Theta_s(T_\chi(ws\cdot\lambda))$ such that the compositions $$\Delta(w\cdot\lambda)\hookrightarrow T_\chi(w\cdot\lambda)\xrightarrow{F_i'} \Theta_s(T_\chi(w\cdot\lambda))\xrightarrow{\Theta_s(\underline{\iota})} \Theta_s(T_\chi(\underline{w}))\xrightarrow{\Theta_s(f_j)} \Theta_s(M)$$ and $$\Delta(w\cdot\lambda)\hookrightarrow T_\chi(w\cdot\lambda)\xrightarrow{G_k} \Theta_s(T_\chi(ws\cdot\lambda))\xrightarrow{\Theta_s(\underline{\iota}')} \Theta_s(T_\chi(\underline{y}))\xrightarrow{\Theta_s(g_l)} \Theta_s(M)$$ form a basis of $\Hom_{\sC_\chi}(\Delta(w\cdot\lambda),\Theta_s(M))$. Arguing as in Propositions~\ref{Trunc} and \ref{EndStand} we have an isomorphism $$\Hom_{\sC_\chi}(\Delta(w\cdot\lambda),\Theta_s(M))\cong \Hom_{\sC_\chi^{\succeq w\cdot\lambda}}(T_\chi(w\cdot\lambda),\Theta_s(M))\cong\Hom_{\sC_\chi^{\succeq w\cdot\lambda}}(T_\chi(\underline{w}),\Theta_s(M)),$$ and we get that the compositions $$ T_\chi(\underline{w})\overset{\underline{\pi}}{\twoheadrightarrow} T_\chi(w\cdot\lambda)\xrightarrow{F_i} \Theta_s(T_\chi(w\cdot\lambda))\xrightarrow{\Theta_s(\underline{\iota})} \Theta_s(T_\chi(\underline{w}))\xrightarrow{\Theta_s(f_j)} \Theta_s(M)$$ and $$T_\chi(\underline{w})\overset{\underline{\pi}}{\twoheadrightarrow} T_\chi(w\cdot\lambda)\xrightarrow{G_k} \Theta_s(T_\chi(ws\cdot\lambda))\xrightarrow{\Theta_s(\underline{\iota}')} \Theta_s(T_\chi(\underline{y}))\xrightarrow{\Theta_s(g_l)} \Theta_s(M)$$ form a basis of $\Hom_{\sC_\chi^{\succeq w\cdot\lambda}}(T_\chi(\underline{w}),\Theta_s(M))$. Setting the $F'_i$ to be the maps $\Theta_s(\underline{\iota})\circ F_i\circ \underline{\pi}$, and the $G'_k$ to be the maps $\Theta_s(\underline{\iota}')\circ G_k\circ \underline{\pi}$, we get the desired result.
	
	(2) The same argument applies.
\end{proof}
%
%


\begin{thebibliography}{9999}
	
	\bibitem[AMRW]{AMRW} P. Achar, S. Makisumi, S. Riche, G. Williamson, {\em Koszul duality for Kac-Moody groups and characters of tilting modules}, J. Amer. Math. Soc. {\bf 32} (2019), 261--310.
	
	\bibitem[AJS]{AJS} H.~H.~Andersen, J.~C.~Jantzen, W. Soergel,  \emph{Representations of quantum groups at a $p$th root of unity and of semisimple groups in characteristic $p$:  independence of $p$}, Ast\'{e}risque {\bf 220} (1994), 321 pp.
	
	\bibitem[BR]{BR} R. Bezrukavnikov, S. Riche, {\em Hecke action on the principal block}, Compos. Math. {\bf 158} (2022), 953--1019.
	
	\bibitem[BL]{BL} R. Bezrukavnikov, I. Losev, {\em On dimension growth of modular irreducible representations of semisimple Lie algebras}, Lie groups, geometry, and representation theory, 59--89, Progr. Math. 326, Birkh\"{a}user/Springer, Cham. 2018.
	
	\bibitem[BD]{BD} J. Brundan, N. Davidson, {\em Categorical actions and crystals}, Categorification and higher representation theory, 105--147, Comtemp. Math., 683, Amer. Math. Soc., Providence, RI, 2017.
	
	\bibitem[BBD]{BBD} A.~A.~Beilinson, J. Bernstein, P. Deligne, {\em Faisceaux pervers}, Analysis and topology on singular spaces, I (Luminy, 1981), 5--171, Ast\'{e}risque, 100, Soc. Math. France, Paris, 1982. 
	
	\bibitem[BS]{BS} J. Brundan, C. Stroppel, {\em Semi-infinite highest weight categories}, \arxiv{1808.08022}.
	
	\bibitem[Ci]{C} J. Ciappara, {\em Hecke category actions via Smith-Treumann theory}, \arxiv{2103.07091}, (2021).
	
	\bibitem[CPS]{CPS} E. Cline, B. Parshall, L. Scott, {\em Finite-dimensional algebras and highest weight categories}, J. Reine Angew. Math. {\bf 391} (1988), 85--99.
	
	\bibitem[Dl]{Dlab} V. Dlab, {\em Properly stratified algebras}, C. R. Acad. Sci. Paris S\'{e}r. I Math. {\bf 331} (2000), 191--196.
	
	\bibitem[GW]{GW} S.~M.~Goodwin, M.~Westaway, {\em On categories of graded modules for reduced enveloping algebras}, In progress.
	
	\bibitem[J1]{Jan4} J.~C.~Jantzen, 
	{\em Modular representations of reductive Lie algebras}, J.\ Pure Appl.\ Algebra {\bf 152} (2000),
	133--185.
	
	\bibitem[J2]{Jan2}
	J.~C.~Jantzen, 
	{\em Representations of Lie algebras in positive characteristic}, Representation theory of algebraic groups and quantum groups, 175--218, Adv. Stud. Pure Math., 40, Math. Soc. Japan, Tokyo, 2004.
	
	\bibitem[J3]{Jan} J.~C.~Jantzen, {\em  Representations of Lie algebras in prime characteristic}, in Representation
	Theories and Algebraic Geometry, Proceedings (A. Broer, Ed.), pp.\ 185–235. Montreal,
	NATO ASI Series, Vol. C 514, Kluwer, Dordrecht, 1998.
	
	\bibitem[J4]{Jan5} J.~C.~Jantzen, {\em  Subregular nilpotent representations of $\fs\fl_n$ and $\fs\fo_{2n+1}$}, Math. Proc. Cambridge Philos. Soc. {\bf 126} (1999), 223--257.
	
	\bibitem[J5]{RAGS} J.~C.~Jantzen, {\em Representations of algebraic groups. Second edition}, Mathematical Surveys and Monographs, 107. Amer. Math. Soc., Providence, RI, 2003.
	
	\bibitem[KL1]{KL93} D.~Kazhdan, G.~Lusztig, {\em Tensor structures arising from affine Lie algebras. I, II}, J. Amer. Math. Soc. {\bf 6} (1993), 905--947, 949--1011.
	
	\bibitem[KL2]{KL94a} D.~Kazhdan, G.~Lusztig, {\em Tensor structures arising from affine Lie algebras. III}, J. Amer. Math. Soc. {\bf 7} (1994), 335--381.
	
	\bibitem[KL3]{KL94b} D.~Kazhdan, G.~Lusztig, {\em Tensor structures arising from affine Lie algebras. IV}, J. Amer. Math. Soc. {\bf 7} (1994), 383--453.
	
	\bibitem[KT1]{KT95} M.~Kashiwara, T.~Tanisaki, {\em Kazhdan-Lusztig conjecture for affine Lie algebras with negative level}, Duke Math. {\bf 77} (1995), 21--62.
	
	\bibitem[KT2]{KT96} M.~Kashiwara, T.~Tanisaki, {\em Kazhdan-Lusztig conjecture for affine Lie algebras with negative level II. Nonintegral case}, Duke Math. {\bf 84} (1996), 771--813.
	
	
	\bibitem[Li]{Li} Y. Li, {\em Tilting modules for the nonrestricted representations of modular Lie algebras}, J. East China Norm. Univ. Natur. Sci. Ed. (2021), 17--22.
	
	\bibitem[LS]{LS} Y. Li, B. Shu, {\em Filtrations on modular representations of reductive Lie algebras}, Algebra Colloq. {\bf 17} (2010), 265--282.
	
	
	\bibitem[L1]{L2} G. Lusztig, {\em Some problems in the representation theory of finite Chevalley groups} in {\em The Santa Cruz Conference on Finite Groups}, pp. 313--317, Proc. Sympos. Pure Math. 37, Amer. Math. Soc., 1980.
	
	\bibitem[L2]{L90a} G.~Lusztig, {\em Finite-dimensional Hopf algebras arising from quantized universal enveloping algebra}, J. Amer. Math. Soc. {\bf 3} (1990), 257--296.
	
	\bibitem[L3]{L90b} G.~Lusztig, {\em On quantum groups}, J. Algebra {\bf 231} (1990), 466--475.
	
	\bibitem[L4]{L94} G.~Lusztig, {\em Monodromic systems on affine flag manifolds}, Proc. Roy. Soc. London Ser. A {\bf 445} (1994), 231--246.
	
	\bibitem[L5]{L95} G.~Lusztig, Errata: {\em Monodromic systems on affine flag manifolds}, [Proc. Roy. Soc. London Ser. A {\bf 445} (1994), 231--246.] Proc. Roy. Soc. London Ser. A {\bf 450} (1995), 731--732.
	
	\bibitem[L6]{L3} G. Lusztig, {\em Periodic $W$-graphs}, Represent. Theory {\bf 1} (1997), 207--279.
	
	\bibitem[Mi]{M} B. Mitchell, {\em Rings with several objects}, Adv. Math. {\bf 8} (1972), 1--161,
	
	\bibitem[MR]{MR} I. Mirkovi\'{c}, D. Rumynin, {\em Centers of reduced enveloping algebras}, Math. Z. {\bf 231} (1999), 123--132.
	
	\bibitem[MV]{MV} I. Mirkovi\'{c}, K. Vilonen, {\em Geometric Langlands duality and representations of algebraic groups over commutative rings}, Ann. of Math. {\bf 166} (2007), 95--143.
	
	
	\bibitem[RW]{RW1} S. Riche, G. Williamson, {\em Tilting modules and the $p$-canonical basis}, Ast\'{e}risque {\bf 397} (2018), 184pp.
	
	\bibitem[RW2]{RW2} S. Riche, G. Williamson, {\em Smith-Treumann theory and the linkage principle}, Publ. Math. Inst. Hautes \'{E}tudes Sci. {\bf 136} (2022), 225--292.
	
	
	\bibitem[We]{West} M. Westaway, {\em On graded representations of modular Lie algebras over commutative algebras}, J. Pure. Appl. Algebra {\bf 226} (2022), Paper No. 107033, 52pp.
	
	\bibitem[W1]{W1} G. Williamson, {\em Schubert calculus and torsion explosion} with an appendix by A. Kontorovich and P. McNamara, J. Amer. Math. Soc. {\bf 30} (2017), 1023--1046.
	
	\bibitem[W2]{W2} G. Williamson, {\em On torsion in the intersection cohomology of Schubert varieties}, J. Algebra {\bf 475} (2017), 207--228.
	
\end{thebibliography}
\end{document}